\documentclass[10pt]{article}

\usepackage{graphicx}
\usepackage{graphics}
\usepackage{amsmath,amssymb}
\usepackage{enumerate}
\usepackage{wrapfig}
\usepackage{graphicx, amsthm, amsmath, amssymb}
\usepackage{subfig}
\usepackage{url}

\usepackage[margin=1in, paperwidth=8.5in, paperheight=11in]{geometry}

\numberwithin{equation}{section}

\DeclareMathAlphabet{\mathscr}{OT1}{pzc}%
                                 {m}{it}

\newtheorem{theorem}{Theorem}[section]
\newtheorem{lemma}[theorem]{Lemma}
\newtheorem{proposition}[theorem]{Proposition}
\newtheorem{corollary}[theorem]{Corollary}
\newtheorem{claim}[theorem]{Claim}

\theoremstyle{definition}
\newtheorem{definition}[theorem]{Definition}

\theoremstyle{remark}

\newcommand{\cK}{\mathcal{K}}

\newcommand{\threeplus}{$3^+$-cluster}
\newcommand{\fourplus}{$4^+$-cluster}
\newcommand{\threepluss}{$3^+$-cluster }
\newcommand{\fourpluss}{$4^+$-cluster }

\newcommand{\openthree}{\mathcal{K}_3^o}
\newcommand{\poorone}{D_1^p}

\begin{document}


\title{New Bounds on the Minimum Density of a Vertex Identifying Code for the Infinite Hexagonal Grid~\thanks{The research is part of the first author's honors project at the College of William and Mary and is supported by NSF CSUMS grant DMS-0703532.  The second author's research is also supported by NSF grant DMS-0852452.}}
\author{Ari Cukierman  \and Gexin Yu}
\date{\today}
\maketitle


\begin{abstract}

For a graph, $G$, and a vertex $v \in V(G)$, let $N[v]$ be the set of vertices adjacent to and including $v$.  A set $D \subseteq V(G)$ is a vertex identifying code if for any two distinct vertices $v_1, v_2 \in V(G)$, the vertex sets $N[v_1] \cap D$ and $N[v_2] \cap D$ are distinct and non-empty.  We consider the minimum density of a vertex identifying code for the infinite hexagonal grid.  In 2000, Cohen et al. constructed two codes with a density of $\frac{3}{7} \approx 0.428571$, and this remains the best known upper bound.  Until now, the best known lower bound was $\frac{12}{29} \approx 0.413793$ and was proved by Cranston and Yu in 2009.  We present three new codes with a density of $\frac{3}{7}$, and we improve the lower bound to $\frac{5}{12} \approx 0.416667$.
\end{abstract}

\section{Introduction}

The study of vertex identifying codes is motivated by the desire to detect failures efficiently in a multi-processor network.  Such a network can be modelled as an undirected graph, $G$, where $V(G)$ represents the set of processors and $E(G)$ represents the set of connections among processors.  Suppose we place detectors on a subset of these processors.  These detectors monitor all processors within a neighborhood of radius $r$ and send a signal to a central controller when a failure occurs.  We assume that no two failures occur simultaneously.  A signal from a detector, $d$, indicates that a processor in the $r$-neighborhood of $d$ has failed but provides no further information.  Now, any given processor, $p$, might be in the $r$-neighborhood of several detectors, $d_1$, $d_2$, $d_3$...  Then, when $p$ fails, the central controller receives signals from $d_1$, $d_2$, $d_3$...  Let us call $\{d_1, d_2, d_3, ... \}$ the trace of $p$ in $G$.  If each processor has a unique and non-empty trace, then the central controller can determine which processor failed simply by noting the detectors from which signals were received.  In this case, we call the subset of processors on which detectors were placed an identifying code.

Vertex identifying codes were first introduced in 1998 by Karpovsky, Chakrabarty and Levitin \cite{karpovsky}.  The processors of the preceding paragraph become the vertices of a graph, and the processors on which detectors have been placed become the vertex subset called a vertex identifying code.  In the example above, we considered detectors which monitor a neighborhood of radius $r$.  In this paper, we concern ourselves with the case in which $r=1$.  

Let $N_i(v)$ be the set of vertices at distance-$i$ from a vertex, $v$, and let $N[v] = N_1(v) \cup \{v\}$.

\begin{definition}
\label{VIC def}

Consider a graph, $G$.  A set $D \subseteq V(G)$ is a {\bf vertex identifying code} if

\begin{itemize}

\item[(i)] For all $v \in V(G)$, $N[v] \cap D \not = \emptyset$

\item[(ii)] For all $v_1,v_2 \in V(G)$ where $v_1 \not = v_2$, $N[v_1] \cap D \not = N[v_2] \cap D$

\end{itemize}

\end{definition}

From Definition \ref{VIC def}, we see that some graphs do not admit vertex identifying codes.  In particular, if $N[v_1] = N[v_2]$ for some distinct $v_1, v_2 \in V(G)$ then $G$ does not admit a vertex identifying code because $N[v_1] \cap D = N[v_2] \cap D$ for any $D \subseteq V(G)$.  On the other hand, if $N[v_1] \not = N[v_2]$ for all distinct $v_1, v_2 \in V(G)$ then $G$ admits a vertex identifying code because $V(G)$ is such a code.


Of particular interest are vertex identifying codes of minimal cardinality.  When dealing with infinite graphs, we consider instead the \emph{density} of a vertex identifying code, i.e., the ratio of the number of vertices in the code to the total number of vertices.  Let $G$ be an infinite graph, and let $D \subseteq V(G)$ be a vertex identifying code for $G$.  Then, for some $v \in V(G)$, the set of vertices in $D$ within distance-$k$ of $v$ is given by $\bigcup_{i=0}^{k} N_i(v) \cap D$.  Let $\sigma(D,G)$ be the density of $D$ in $G$.  Then,
\begin{equation}
\sigma(D,G) = \displaystyle\limsup_{k \to \infty} \frac{ \left | \bigcup_{i=0}^{k} N_i(v) \cap D \right | }{ \left | \bigcup_{i=0}^{k} N_i(v) \right | }
\end{equation}
Let $\sigma_0(G)$ be the minimum density of a vertex identifying code for $G$; that is,
\begin{equation}
\sigma_0(G) = \min_D\{ \sigma(D,G) \}
\end{equation}

Karpovsky et al. \cite{karpovsky} considered the minimum density of vertex identifying codes for the infinite triangular ($G_T$), square ($G_S$) and hexagonal ($G_H$) grids.  They showed $\sigma_0(G_T) = 1/4$.  In 1999, Cohen et al. \cite{cohen1999} proved $\sigma_0(G_S) \leq 7/20$, and, in 2005, Ben-Haim and Litsyn \cite{benhaim} completed the proof by showing $\sigma_0(G_S) \geq 7/20$.

We concern ourselves in this paper with $\sigma_0(G_H)$.  In 1998, Karpovsky et al. \cite{karpovsky} showed $\sigma_0(G_H) \geq 2/5 = 0.4$.  In 2000, Cohen et al. \cite{cohen2000} improved this result to $\sigma_0(G_H) \geq 16/39 \approx 0.410256$ and constructed two codes with a density of $3/7 \approx 0.428571$ implying $\sigma_0(G_H) \leq 3/7$.  In 2009, Cranston and Yu \cite{12/29} proved $\sigma_0(G_H) \geq 12/29 \approx 0.413793$.   For other results on identifying codes for the hexagonal grid, see \cite{martin, stanton}.

In this paper, we present three new codes with a density of $ 3/7$ and prove $\sigma_0(G_H) \geq 5/12 \approx 0.416667$.  In conclusion, it is now known that $5/12 \leq \sigma_0(G_H) \leq 3/7$.


\begin{figure}[h]
\begin{center}
\subfloat{\label{chaseconstruction}
\includegraphics[width=0.322\textwidth]{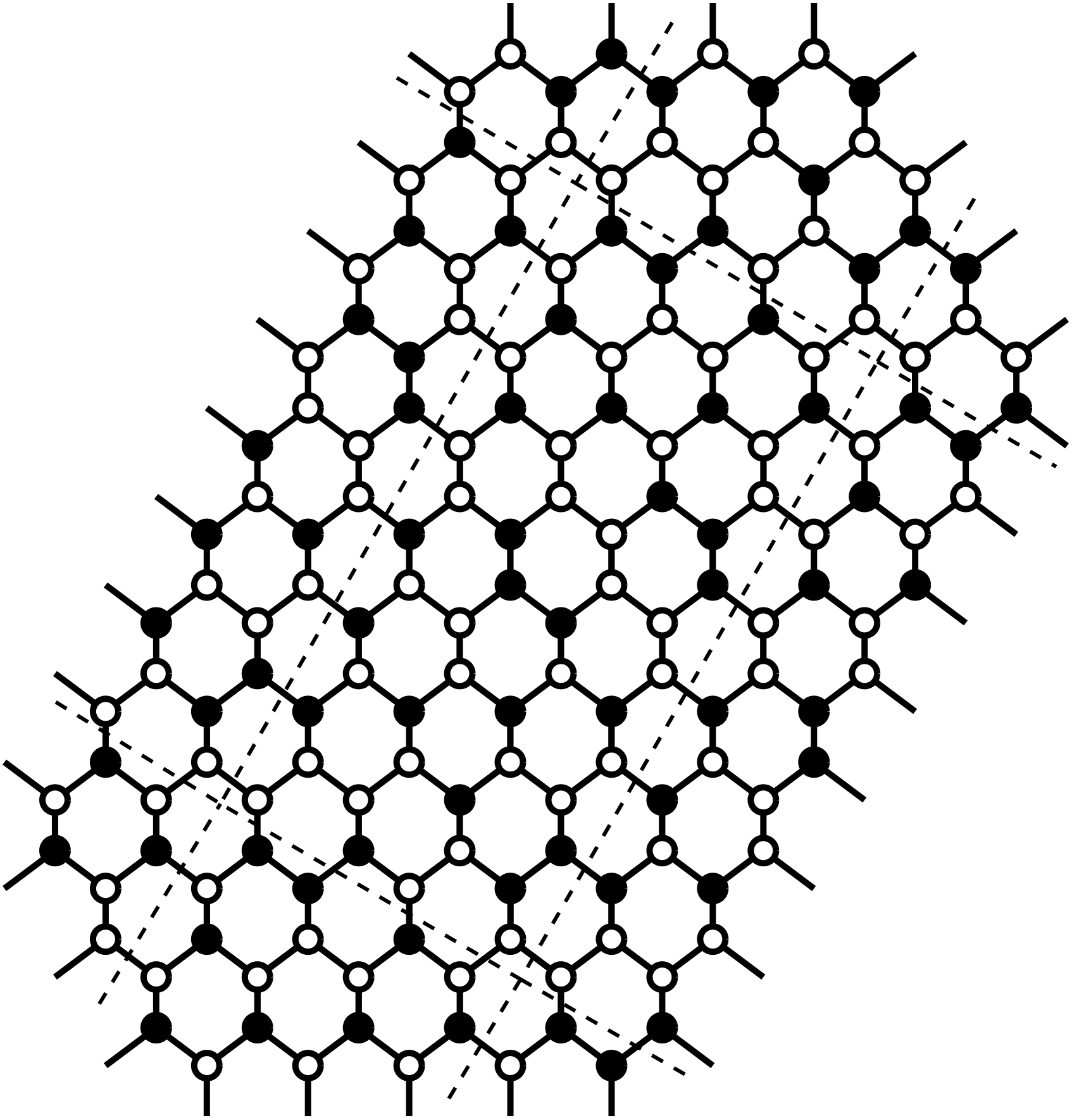}}
\subfloat{\label{chaseconstruction2}\includegraphics[width=0.368\textwidth]{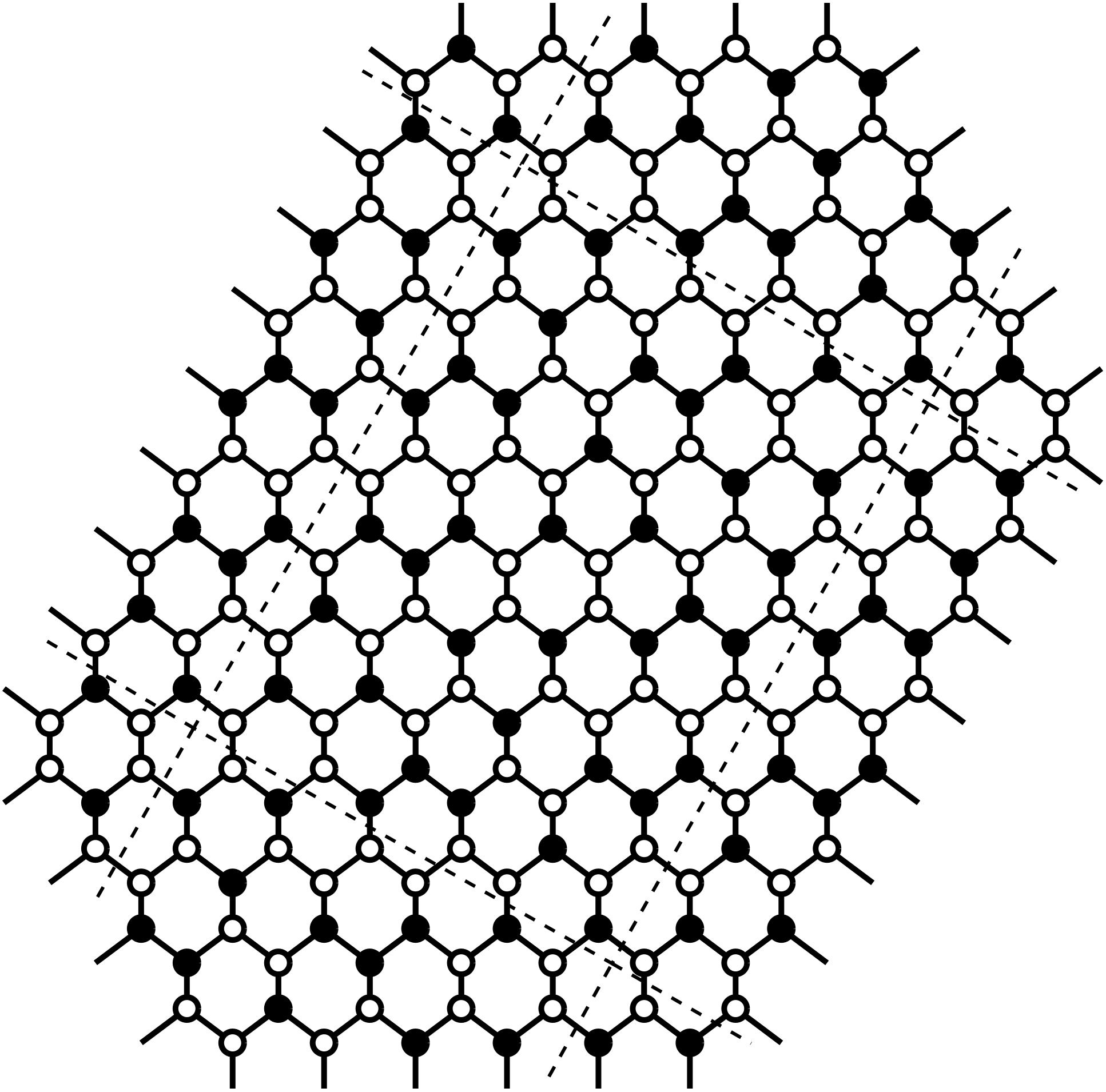}}
\subfloat{\label{jeffconstruction}\includegraphics[width=0.307\textwidth]{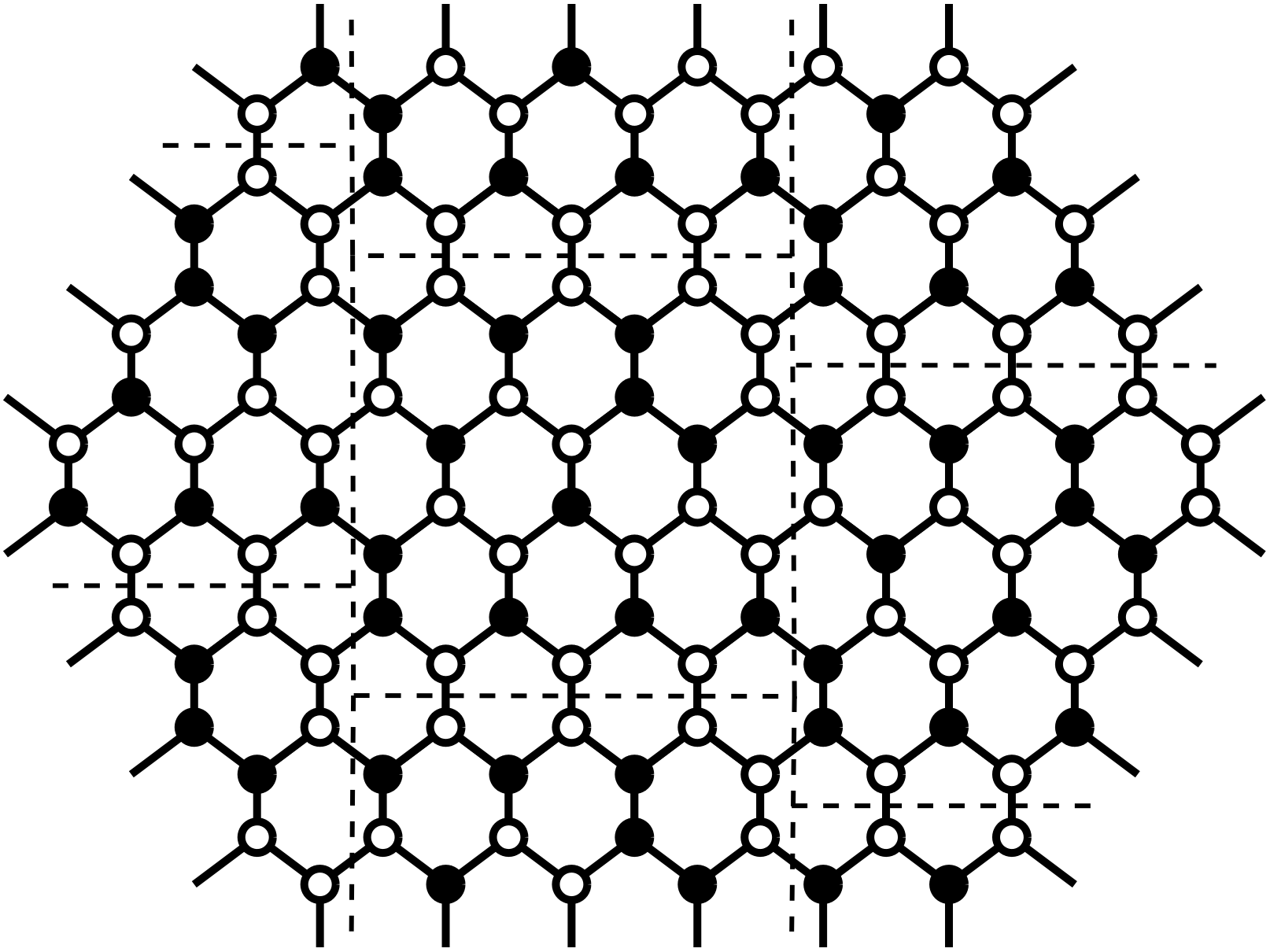}}
\end{center}
\caption{Three new codes with a density of $3/7$.  The solid vertices are in the code.}
\label{3codes}
\end{figure}

Suppose $\beta$ is an upper bound on $\sigma_0(G_H)$.  To prove this, we need only show the existence of a code, $D$, with $\sigma(D,G_H) \leq \beta$.  When constructing such codes, we usually look for tiling patterns.  Since the pattern repeats ad infinitum, the density of one tile is the density of the whole graph.  Figure \ref{3codes} shows three new codes for the infinite hexagonal grid with a density of $3/7$.

\begin{theorem}
\label{theorem}
The minimum density of a vertex identifying code for the infinite hexagonal grid is greater than or equal to $5/12$.
\end{theorem}

To prove Theorem \ref{theorem}, we employ the discharging method.  Let $D$ be an arbitrary vertex identifying code for $G_H$.  We assign 1 ``charge" to each vertex in $D$ which we then redistribute so that every vertex in $G_H$ retains at least $5/12$ charge.  The charge is redistributed in accordance with a set of ``Discharging Rules".  Since $D$ was chosen arbitrarily, we then conclude that $5/12$ is a lower bound on $\sigma_0(G_H)$.

As the proof of Theorem \ref{theorem} is rather lengthy, we include a sketch of the proof in Section \ref{sketch}.  In Section \ref{general}, we introduce several properties of vertex identifying codes for $G_H$ which we will reference throughout the paper.  Section \ref{terminology} is devoted to terminology and notations; the vast majority of relevant notions are defined here.  In Section \ref{lemmas}, we state several lemmas concerning the structure of vertex identifying codes for $G_H$.  However, we defer the proofs of these lemmas to Section \ref{deferred proofs}.  The main result of this paper, Theorem \ref{theorem}, is proved in Section \ref{mainresult}.  

For the rest of the paper, if not explicitly stated, $D$ is to be interpreted as a vertex identifying code for the infinite hexagonal grid.

\section{Sketch of the Proof}
\label{sketch}

As mentioned in the introduction, our proof of Theorem \ref{theorem} makes use of the discharging method.  We assign 1 charge to each vertex in $D$ and then redistribute this charge so that each vertex in $G_H$ retains at least $5/12$. 
To design the proper discharging rules, we start with the following (Rule 1 in Section 6):
\begin{center}
If a vertex, $v$, is not in $D$ and has $k$ neighbors in $D$, then $v$ receives $\frac{5}{12k}$ from each of these neighbors.
\end{center}
We can easily verify that Rule 1 suffices to allow each vertex in $G_H \setminus D$ to retain $5/12$ charge (Claim \ref{notinD claim}).  As a result, the remaining discharging rules are concerned exclusively with vertices in $D$.  Now, any vertex, $v$, in $D$ with a neighbor in $G_H \setminus D$ loses charge by Rule 1.  
We show in Section~\ref{mainresult} that only one type of vertex loses too much by Rule 1; we call such a vertex a {\em poor $1$-cluster} (Definition~\ref{poor1 def}).  Consequently, we must find charge to send to poor 1-clusters from nearby vertices.
We find that it is helpful to consider a cluster (Definition~\ref{cluster}) as a single entity.  Thus we first need to determine the surplus charge each cluster may have after Rule 1.   


We observe that some $1$-clusters may have surplus charge and that their surplus differs according to the neighbors they may have; for this reason we define \emph{non-poor $1$-clusters} (Definition \ref{nonpoor1 def}) and \emph{one-third vertices} (Definition \ref{onethird def}).  In Lemmas~\ref{nonpoor lem1}-\ref{nonpoor lem3}, we determine how many poor $1$-clusters can lie in the neighborhood of a non-poor $1$-cluster, and then in Rules 2, 3d and 3e, we design the appropriate discharging rules to distribute the surplus charge.    In Claim~\ref{nonpoor1 claim}, we show that non-poor $1$-clusters ultimately retain a charge of at least $5/12$. 

For \threeplus s, the situation is more complicated.   We first see a difference of surplus charge according to the distribution of vertices at distance-2 from a given \threeplus; for this reason we define \emph{open/closed $k$-clusters} (Definition \ref{open-closed}), \emph{crowded/uncrowded $k$-clusters} (Definition~\ref{crowded}) and the $P$-function (Definition \ref{p-function}).  These definitions allow us to distinguish among \threeplus s with varying amounts of surplus charge.  We will see in Section \ref{mainresult} that for very large $k$, a $k$-cluster can always afford to send charge to all nearby poor $1$-clusters.  Consequently, we are mostly concerned with $k$-clusters with $3\le k\le 6$.  In Lemmas \ref{closed3 lem}-\ref{k+8}, we determine the number of poor $1$-clusters that can lie in the neighborhood of a given $k$-cluster.  Discharging Rules 3a-3c are designed in accordance with these lemmas to send charge from $3^+$-clusters to poor 1-clusters lying in a distance-2 or distance-3 neighborhood.

Now, some poor 1-clusters do not lie in a neighborhood that receives charge by Rule 3.  We call these \emph{very poor $1$-clusters} (Definition \ref{vp def}), and we distinguish between two orientations: symmetric and asymmetric (Definition \ref{orientations}).  In Lemmas \ref{vpsym lem}, \ref{vp lem} and \ref{vpshoulders lem} we scan the neighborhood of a very poor 1-cluster for clusters with charge available for redistribution after Rule 3.  
Crucially, we find in Lemma \ref{vp lem} that if there is no other way to squeeze charge for a given very poor 1-cluster from a single nearby cluster, there must be \emph{type-$1$ paired $3$-clusters} or \emph{type-$2$ paired $3$-clusters} (Definition \ref{paired def}) in the neighborhood.
These are structures which tend to form in the extended neighborhood of an asymmetric very poor 1-cluster and which always have extra charge after Rule 3.  In order to reserve this extra charge for very poor $1$-clusters, several discharging rules make exceptions for type-$1$ and type-$2$ paired $3$-clusters.  That this creates no new deficiency of charge is proved in Section \ref{mainresult}.  We prove some properties of type-1 and type-2 paired 3-clusters in Lemmas \ref{type1paired lem} and \ref{type2paired lem}.  Discharging Rules 4-7 are designed in accordance with the above-mentioned lemmas to send charge to very poor 1-clusters.

On an additional note,  the structure of type-1 and type-2 paired $3$-clusters is very specific, and this forces us to introduce some very specific notions (for example, Definitions 4.5 and 4.6).  This is done so that our analysis can penetrate to the properties of individual vertices.  As a result, the proofing process is somewhat tedious though more or less straightforward.



\section{General Structural Properties}
\label{general}

\begin{definition}\label{cluster}
A component of the subgraph induced by $D$ is called a {\bf cluster}.  A cluster containing $k$ vertices is called a {\bf $\bf k$-cluster}; a cluster containing $k$ or more vertices is called a {\bf ${\bf k^+}$-cluster}.
Let $D_k$ be the set of all vertices in $k$-clusters; and let $\mathcal{K}_k$ be the set of all $k$-clusters.  Let $d_C(v)$ be the degree of a vertex, $v$, in a \threeplus, $C$; and let $\Delta(C) = \max \{d_C(v) : v \in C \}$.
\end{definition}


\begin{proposition}
\label{2clusters prop}

There exist no $2$-clusters.

\end{proposition}

\begin{proof}

Suppose by contradiction that there exists a 2-cluster, $C$, and let $V(C) = \{v,w\}$.  Then, $N[v] \cap D = \{v,w\}$ and $N[w] \cap D = \{v,w\}$.  Now, if $N[v] \cap D = N[w] \cap D$, then $v=w$ (Definition \ref{VIC def}), which is a contradiction.
\end{proof}

\begin{corollary}
\label{not in 3+}

If a vertex, $v$, is not in a \threeplus, then either $v$ is not in $D$ or $v$ is a $1$-cluster.

\end{corollary}

\begin{proposition}
\label{force3cluster prop}

If a vertex not in $D$ has $2$ adjacent vertices not in $D$, then the remaining adjacent vertex is in a \threeplus.

\end{proposition}

\begin{proof}

Consider a vertex, $v$, such that that $N_1(v) = \{ a, b, c\}$ and let $a,b,v \not \in D$.  Suppose by contradiction that $c \not \in D_{3^+}$.  Then, $c \not \in D$ or $c \in D_1$ (Corollary \ref{not in 3+}).  If $c \not \in D$, then $N[v] \cap D = \emptyset$, which is a contradiction (Definition \ref{VIC def}).  If $c \in D_1$, then $N[v] \cap D = N[c] \cap D = \{ c \}$; therefore, $c = v$ (Definition \ref{VIC def}), which is a contradiction.
\end{proof}


\begin{proposition}
\label{1clusters prop}

Each of the vertices adjacent to a $1$-cluster, $v$, has at least one adjacent vertex in $D \setminus \{v\}$.

\end{proposition}

\begin{proof}

Let $v \in D_1$, and let $u$ be an adjacent vertex.  Suppose by contradiction that $u$ has no adjacent vertices in $D \setminus \{v\}$.  Then, $v \in D_{3^+}$ (Proposition \ref{force3cluster prop}), which is a contradiction.
\end{proof}

\begin{proposition}
\label{leaves prop}

Each leaf of a \threeplus, $C$, has at least one distance-$2$ vertex in $D \setminus C$.

\end{proposition}

\begin{proof}

Let $v$ be a leaf of a \threeplus, $C$.  Then, exactly 2 of the vertices adjacent to $v$ are not in $D$; let $u$ and $w$ be these vertices.  Suppose by contradiction that $v$ has no distance-2 vertices in $D \setminus C$.  Then, $N[u] \cap D = N[w] \cap D = \{v\}$; therefore, $u=w$ (Definition \ref{VIC def}), which is a contradiction.
\end{proof}




\section{Terminology and Notations}
\label{terminology}

We introduce the following convention which we will use throughout the paper.  Let $G$ be a graph, and suppose $D \subseteq V(G)$ is a vertex identifying code for $G$.  In the figures, we use a solid vertex to denote that a vertex is in $D$, and we use a hollow vertex to denote that a vertex is not in $D$.  The status of all other vertices is undetermined.  In Figure \ref{VICconvention}, for instance, $u \in D$ and $v \not \in D$, while the status of $w$ is undetermined.


\begin{figure}[h]
\begin{center}
\includegraphics[width=0.15\textwidth]{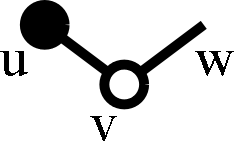}
\end{center}
\vspace{-15pt}
\caption{}
\label{VICconvention}
\end{figure}

\begin{definition}
\label{poor1 def}
A 1-cluster with exactly 3 distance-2 vertices in $D$ is called a {\bf poor $1$-cluster}.  Let $D_1^p$ be the set of all poor 1-clusters.  
\end{definition}

\begin{corollary}
\label{poor1 cor}
Each of the neighbors of a poor $1$-cluster, $v$, has exactly one neighbor in $D \setminus \{v\}$.
\end{corollary}

\begin{proof}
This follows from Proposition \ref{1clusters prop} and Definition \ref{poor1 def}.
\end{proof}

\begin{definition}
\label{open-closed}
For $k \geq 3$, let $C$ be a $k$-cluster with $\Delta(C) = 2$.  If none of the non-leaf vertices of $C$ has a distance-2 vertex in $D \setminus C$, then $C$ is an {\bf open $k$-cluster}.  If at least one of the non-leaf vertices of $C$ has a distance-2 vertex, $v$, in $D \setminus C$, then $C$ is a {\bf closed $k$-cluster} and $v$ {\bf closes} $C$.
Let $D_k^o$ be the set of all vertices in open $k$-clusters and $D_k^c$ the set of all vertices in closed $k$-clusters; let $\mathcal{K}_k^o$ be the set of all open $k$-clusters and $\mathcal{K}_k^c$ the set of all closed $k$-clusters.
\end{definition}

\begin{definition}
\label{crowded}
If an open $k$-cluster, $C$, has exactly 2 distance-2 vertices in $D$, both of which are poor 1-clusters, then $C$ is {\bf uncrowded}.  Otherwise, $C$ is {\bf crowded}.
\end{definition}

\begin{definition}
\label{p-function}
For a given cluster, $C$, let $P(C) = \sum\limits_{v \in C} | N_2(v) \cap D \setminus C | $.
\end{definition}

%


\begin{definition}
\label{3cluster def}
Let $C$ be the $3$-cluster shown in Figure \ref{3-cluster}.  Vertices $a$ and $b$ are in the {\bf head positions} of $C$; $c$ and $e$ are in the {\bf shoulder positions}; $f$ and $g$ are in the {\bf arm positions}; $h$ and $m$ are in the {\bf hand positions}; $i$ and $k$ are in the {\bf foot positions}; $j$ is in the {\bf tail position}; and $n$ and $q$ are in the {\bf fin positions}.
If $q$ is not in $D$, then $b,d,e,g,k$ and $m$ are on the {\bf finless side} of $C$.  If $d$ is in $D$, then $b,e,g,k,m$ and $q$ are on the {\bf closed side} of $C$.
\end{definition}

\begin{figure}[h]
\setcaptionwidth{0.2\textwidth}
\begin{center}
\subfloat[3-Cluster]{\label{3-cluster}\includegraphics[width=.25\textwidth]{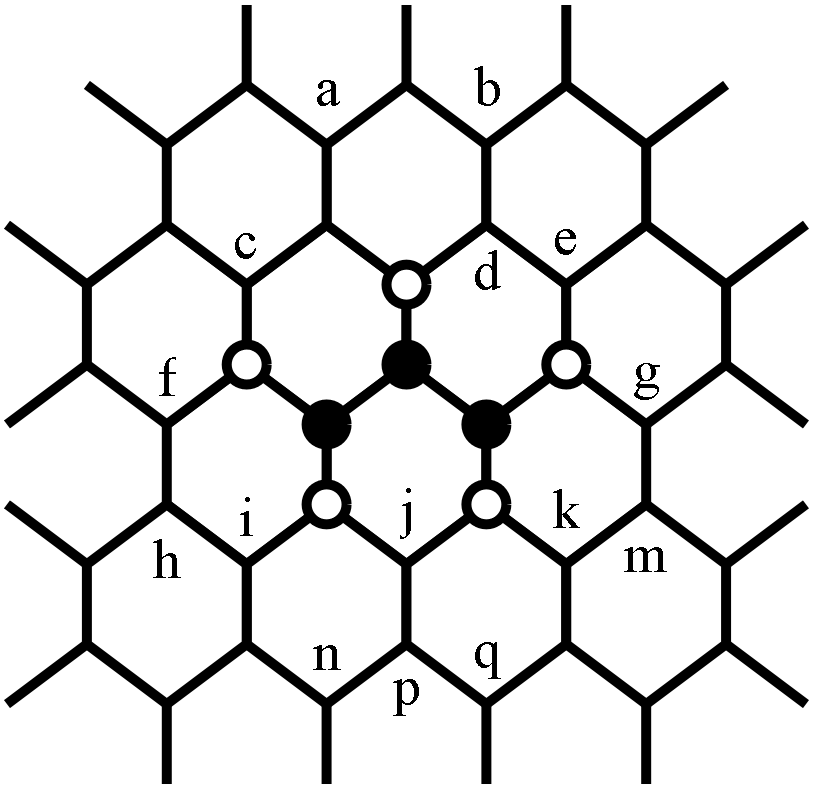}}
\hspace{1 cm}
\subfloat[Linear 4-Cluster]{\label{linear4}\includegraphics[width=0.275\textwidth]{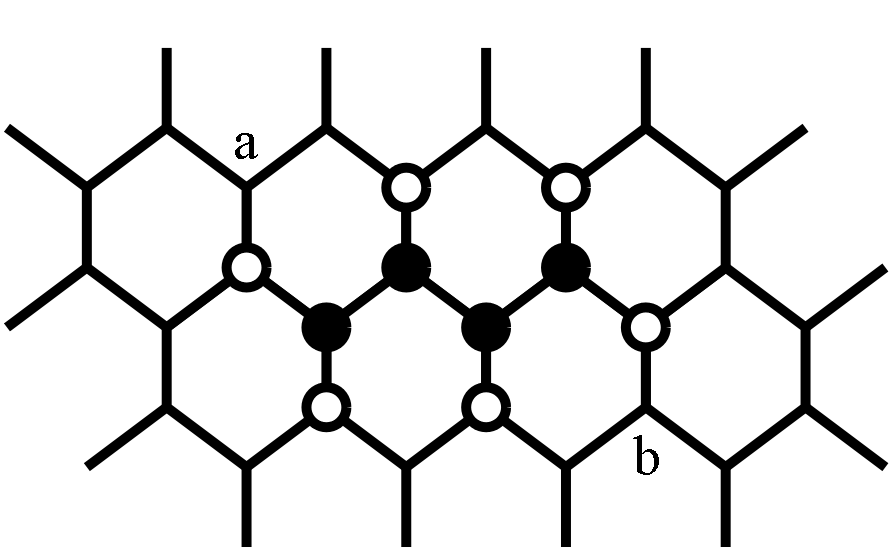}}
\hspace{1 cm}
\subfloat[Curved 4-Cluster]{\label{curved4}\includegraphics[width=0.2\textwidth]{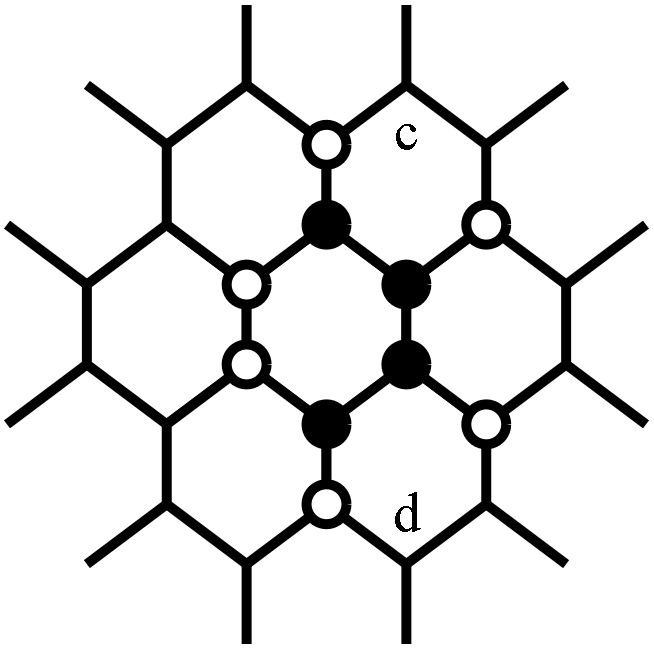}}
\end{center}
\caption{}
\label{linear and curved 4}
\end{figure}

\begin{definition}
\label{4cluster def}
Let $C$ be a 4-cluster with $\Delta(C) = 2$.  If the leaves of $C$ do not lie on the same 6-cycle, then $C$ is a {\bf linear $4$-cluster}.  Otherwise, $C$ is a {\bf curved $4$-cluster}.  Let $C_1$ be the linear 4-cluster shown in Figure \ref{linear4}.  Vertices $a$ and $b$ are in the {\bf one-turn positions} of $C_1$.  Let $C_2$ be the curved 4-cluster shown in Figure \ref{curved4}.  Vertices $c$ and $d$ are in the {\bf backwards positions} of $C_2$.

\end{definition}

\begin{definition}

A vertex, $v$, is {\bf distance-$k$ from a cluster}, $C$, if $k$ is the minimum distance from $v$ to any of the vertices of $C$.  If $k \leq \ell$, then $v$ is {\bf within distance-$\ell$} of $C$.
If a vertex, $v$, is within distance-3 of a cluster, $C$, then $v$ is {\bf nearby} $C$.

\end{definition}

\begin{figure}[h]
\setcaptionwidth{0.26\textwidth}
\begin{center}
\subfloat[Paired 3-Clusters]{\label{paired}\includegraphics[width=0.227\textwidth]{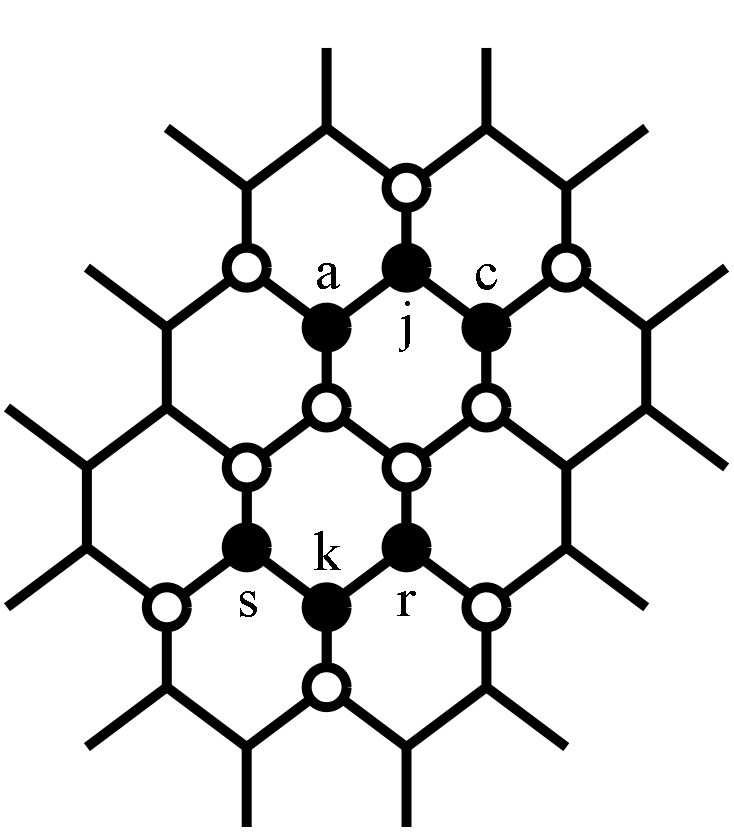}}
\hspace{0.5 cm}
\subfloat[Type-1 Paired 3-Clusters]{\label{type-1 paired}\includegraphics[width=0.217\textwidth]{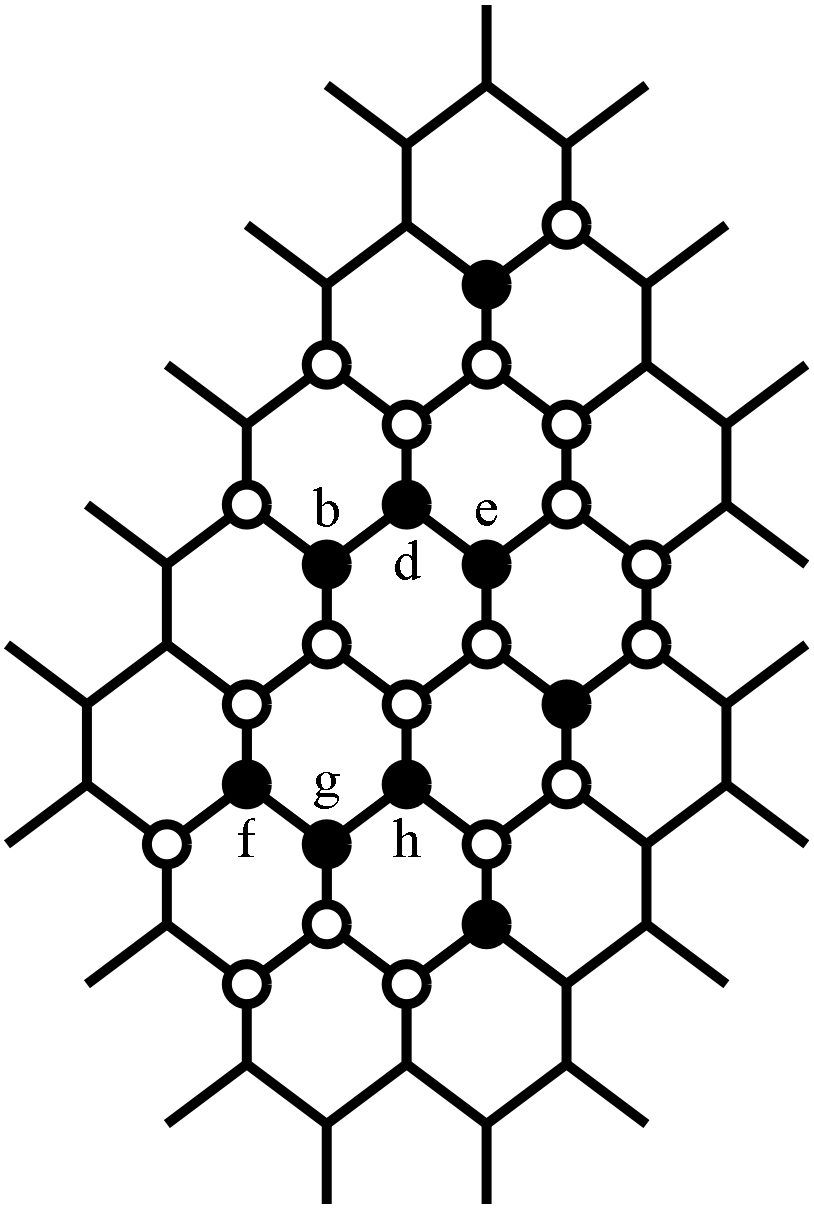}}
\hspace{0.5 cm}
\subfloat[Type-2 Paired 3-Clusters]{\label{type-2 paired}\includegraphics[width=0.227\textwidth]{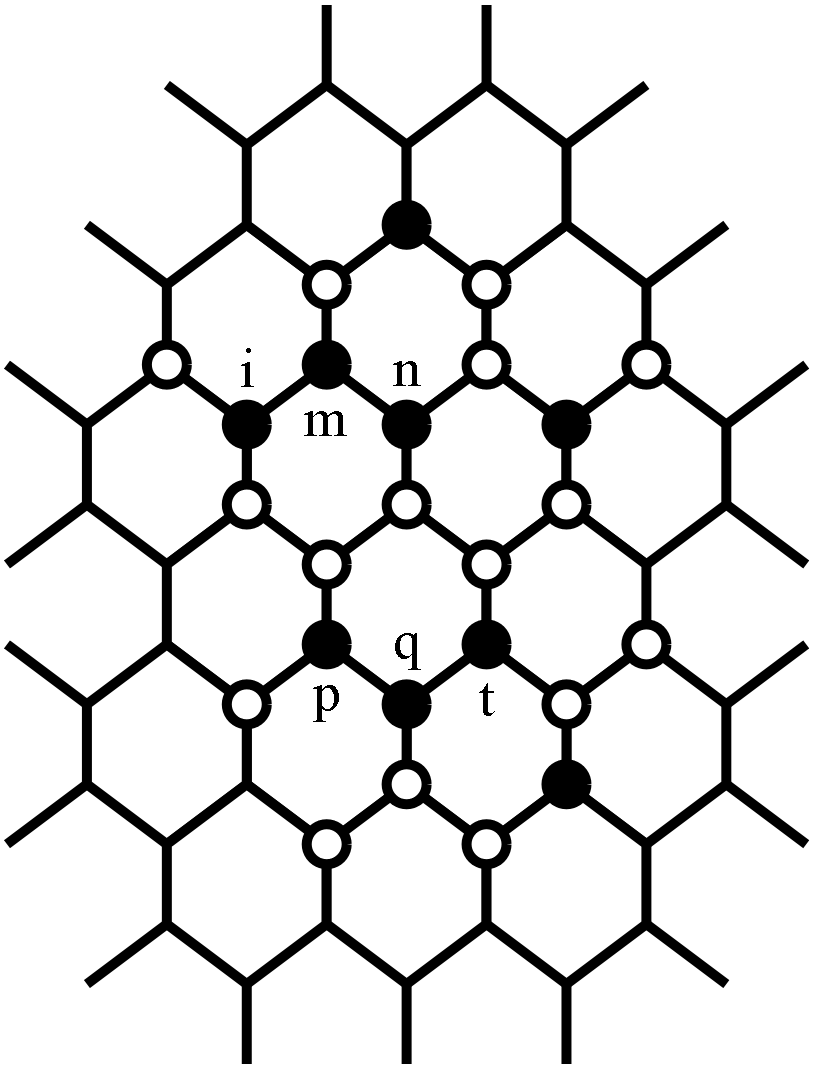}}
\end{center}
\caption{}
\label{type 1 and type 2 paired}
\end{figure}

\begin{definition}
\label{paired def}
Let $C_1$ be the 3-cluster described by $a$, $j$ and $c$ in Figure \ref{paired}, and let $C_2$ be the 3-cluster described by $s$, $k$ and $r$.  Then, $C_1$ and $C_2$ are {\bf paired 3-clusters}.
Let $C_3$ be the 3-cluster described by $b$, $d$ and $e$ in Figure \ref{type-1 paired}, and let $C_4$ be the 3-cluster described by $f$, $g$ and $h$.  Then, $C_3$ and $C_4$ are {\bf type-1 paired}, and $C_3$ is {\bf type-1 paired on top}.  
Let $C_5$ be the 3-cluster described by $i$, $m$ and $n$ in Figure \ref{type-2 paired}, and let $C_6$ be the 3-cluster described by $p$, $q$ and $t$.  Then, $C_5$ and $C_6$ are {\bf type-2 paired}.
\end{definition}

\begin{corollary}
\label{paired cor}
If a $3$-cluster, $C$, is type-$1$ paired, then $C$ is not type-$2$ paired, and vice versa.
\end{corollary}

\begin{definition}
\label{stealable def}
A poor 1-cluster, $v$, is {\bf stealable} if $v$ is distance-3 from a \fourpluss and distance-2 from a \threeplus, $C$, such that if $C$ is an open 3-cluster, then
\begin{itemize}
\item[(i)] $v$ is not in a shoulder position;
\item[(ii)] if $v$ is in an arm position, then $C$ is neither type-1 nor type-2 paired.
\end{itemize}
\end{definition}

\begin{definition}
\label{nonpoor1 def}
A 1-cluster that is not poor is called a {\bf non-poor 1-cluster}.    Let $D_1^{np}$ be the set of all non-poor 1-clusters.  If 3 non-poor 1-clusters, $u$, $v$ and $w$, are adjacent to the same one-third vertex, then $u$, $v$ and $w$ are referred to as a {\bf group of non-poor $1$-clusters}.  A vertex, $v$, is {\bf distance-$k$ from a group of non-poor $1$-clusters}, $H$, if $v$ is distance-$k$ from any of the 1-clusters in $H$.
\end{definition}

\begin{definition}
\label{onethird def}
If a vertex, $v$, is not in $D$ and has 3 neighbors in $D$, then $v$ is called a {\bf one-third vertex}.
\end{definition}

\begin{definition}
\label{vp def}

If a poor 1-cluster, $v$, is neither distance-2 from a $3^+$-cluster or a non-poor 1-cluster nor distance-3 from a closed 3-cluster or $4^+$-cluster, then $v$ is called a {\bf very poor 1-cluster}.  Let $D_1^{vp}$ be the set of all very poor 1-clusters.

\end{definition}

\begin{figure}[h]
\setcaptionwidth{0.18\textwidth}
\begin{center}
\subfloat[Symmetric Orientation]{\label{vpsym}\includegraphics[width=0.13\textwidth]{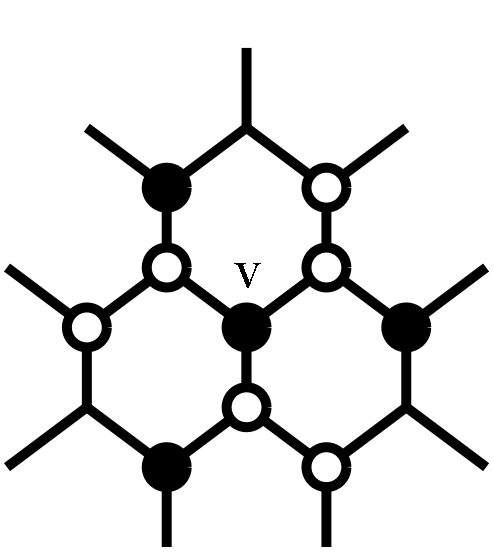}}
\hspace{3.3 cm}
\subfloat[Asymmetric Orientation]{\label{vpasym}\includegraphics[width=0.13\textwidth]{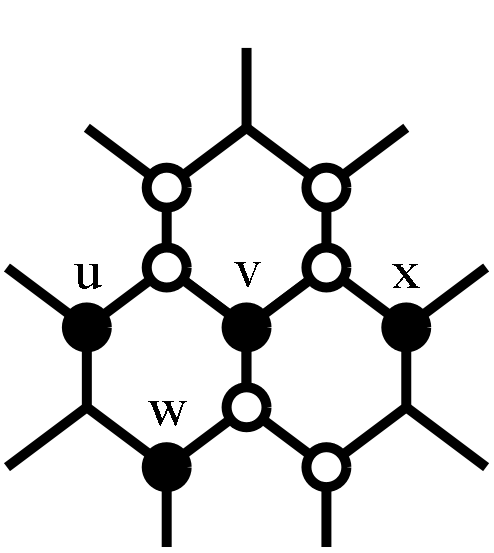}}
\end{center}
\caption{}
\label{verypoororientations}
\end{figure}

\begin{definition}
\label{orientations}

If a very poor 1-cluster, $v$, has 3 distance-2 vertices in $D$ which are all distance-4 from each other, then $v$ is in a {\bf symmetric orientation} (see Figure \ref{vpsym}).  A very poor 1-cluster which is not in a symmetric orientation is in an {\bf asymmetric orientation} (see Figure \ref{vpasym}).  The vertex $u$ in Figure \ref{vpasym} is in the {\bf u-position} of $v$, the vertex $w$ is in the {\bf w-position} of $v$, and the vertex $x$ is in the {\bf x-position} of $v$.

\end{definition}

\section{Structural Lemmas}
\label{lemmas}

In this section, we state several lemmas concerning the structure of a vertex identifying code for the infinite hexagonal grid.  As the primary purpose of these lemmas is to abridge the proof of Theorem \ref{theorem}, we defer the proofs of all lemmas to page \pageref{deferred proofs}.  Additionally, we defer the proofs of Proposition \ref{nonpoor prop1} and Corollaries \ref{vpasym cor2} and \ref{vpasym cor3}.

\begin{proposition}
\label{nonpoor prop1}

If a poor $1$-cluster, $v$, is distance-$2$ from exactly one of the $1$-clusters in a group of non-poor $1$-clusters, then $v$ is distance-$2$ from an open $3$-cluster or within distance-$3$ of a closed $3$-cluster or \fourplus.

\end{proposition}

\begin{lemma}
\label{nonpoor lem1}

Consider a group of non-poor $1$-clusters, $H$.  There exist at most $2$ poor $1$-clusters which are distance-$2$ from $H$ and neither distance-$2$ from an open $3$-cluster nor within distance-$3$ of a closed $3$-cluster or \fourplus.

\end{lemma}

\begin{lemma}
\label{nonpoor lem2}

Let $v$ be a one-third vertex, and let $v$ have exactly $2$ adjacent $1$-clusters, $c$ and $d$.  Each of $c$ and $d$ has at most one distance-$2$ poor $1$-cluster that is neither distance-$2$ from an open $3$-cluster nor within distance-$3$ of a closed $3$-cluster or \fourplus.

\end{lemma}

\begin{lemma}
\label{nonpoor lem3}

If a one-third vertex has exactly one adjacent $1$-cluster, $d$, then $d$ has at most $2$ distance-$2$ poor $1$-clusters.

\end{lemma}

\begin{lemma}
\label{finless lem}

A $3$-cluster has at most one finless side.

\end{lemma}

\begin{lemma}
\label{closed3 lem}

Let $C_1$ be a closed $3$-cluster with $P(C_1) = 3$.


\begin{itemize}

\item[(i)] $C_1$ has at most $8$ nearby poor 1-clusters.  If $C_1$ has $8$ such clusters, at least one of the poor $1$-clusters at distance-$3$, $v$, is distance-$2$ from another \threeplus, $C_2$, such that

\begin{itemize}

\item[(a)] if $C_2$ is an open $3$-cluster, then $v$ is not in a shoulder position;

\item[(b)] if $C_2$ is an open $3$-cluster and $v$ is in an arm position, then $C_2$ is not type-$1$ paired; if $C_2$ is type-2 paired, then $C_1$ is type-$2$ paired with $C_2$.

\end{itemize}

\item[(ii)] If neither the shoulder positions nor the tail position are in $D$, then $C_1$ has at most $5$ nearby poor $1$-clusters.

\end{itemize}

\end{lemma}

\begin{lemma}
\label{closed3 lem2}

Let $C$ be a closed $3$-cluster with $P(C) = 4$.  If $C$ is adjacent to a one-third vertex, then $C$ has at most $8$ nearby poor $1$-clusters.  Furthermore, if an arm position or a foot position of $C$ is a poor $1$-cluster, then $C$ has at most $7$ nearby poor $1$-clusters.  If $2$ arm or foot positions are poor $1$-clusters, then $C$ has at most $6$ nearby poor $1$-clusters.

\end{lemma}

\begin{lemma}
\label{linearopen4 lem}

Let $C_1$ be a linear open $4$-cluster with $P(C_1) = 2$.

\begin{itemize}

\item[(i)] $C_1$ has at most $8$ nearby poor $1$-clusters.  Furthermore, if $C_1$ has $8$ such clusters, then at least $2$ of the distance-$3$ poor $1$-clusters are stealable; if $C_1$ has $7$ such clusters, then at least one is stealable.

\item[(ii)] If one one-turn position is not in $D$, then $C_1$ has at most $6$ nearby poor $1$-clusters.  If $C_1$ has exactly $6$ such clusters, then at least one of the distance-$3$ poor $1$-clusters is stealable.

\item[(iii)] If neither one-turn position is in $D$, then $C_1$ has at most $4$ nearby poor $1$-clusters.

\end{itemize}

\end{lemma}

\begin{lemma}
\label{linear4 lem2}

Let $C$ be a linear $4$-cluster with $P(C) = 3$.

\begin{itemize}

\item[(i)] If $C$ is adjacent to no one-third vertices, then $C$ has at most $9$ nearby poor $1$-clusters.

\item[(ii)] If $C$ is adjacent a one-third vertex, then $C$ has at most $6$ nearby poor $1$-clusters.

\end{itemize}

\end{lemma}

\begin{lemma}
\label{curvedopen4 lem}

Let $C_1$ be a curved open $4$-cluster with $P(C_1) = 2$.

\begin{itemize}

\item[(i)] $C_1$ has at most $8$ nearby poor $1$-clusters.  Furthermore, if $C_1$ has $8$ such clusters, then at least $2$ of the distance-$3$ poor $1$-clusters are stealable; if $C_1$ has $7$ such clusters, then at least one is stealable.

\item[(ii)] If one backwards position is not in $D$, then $C_1$ has at most $6$ nearby poor $1$-clusters.  If $C_1$ has exactly $6$ such $1$-clusters, then at least one of the distance-$3$ poor $1$-clusters is stealable.

\item[(iii)] If neither backwards position is in $D$, then $C_1$ has at most $2$ nearby poor $1$-clusters.

\end{itemize}

\end{lemma}

\begin{lemma}
\label{curved4 lem2}

Let $C$ be a curved $4$-cluster with $P(C) = 3$.

\begin{itemize}

\item[(i)] If $C$ is adjacent to no one-third vertices, then $C$ has at most $11$ nearby poor $1$-clusters.  Furthermore, if $C$ has $k$ backwards positions not in $D$, then $C$ has at most $11-k$ nearby poor $1$-clusters.

\item[(ii)] If $C$ is adjacent to a one-third vertex, then $C$ has at most $6$ nearby poor $1$-clusters.

\end{itemize}

\end{lemma}

\begin{lemma}
\label{5cluster lem}

An open $5$-cluster, $C$, has at most $9$ nearby poor $1$-clusters.  If $C$ has exactly $9$ such $1$-clusters, then at least one of the distance-$3$ poor $1$-clusters is stealable.

\end{lemma}

\begin{lemma}
\label{4star lem}

If a $4$-cluster, $C$, has one degree-$3$ vertex and $P(C) = 3$, then $C$ has at most $8$ nearby poor $1$-clusters.

\end{lemma}

\begin{lemma}
\label{5star lem}

If a $5$-cluster, $C$, has one degree-$3$ vertex, then $C$ has at most $12$ nearby poor $1$-clusters.

\end{lemma}

\begin{lemma}
\label{6cluster lem}

An open $6$-cluster, $C$,  with $\Delta(C) = 2$ has at most $10$ nearby poor $1$-clusters.

\end{lemma}

Lemma \ref{k+8} was proved by Cranston and Yu \cite[p. 14]{12/29} in 2009.  We state it here without proof.

\begin{lemma}
\label{k+8}

For $k \geq 3$, a $k$-cluster has at most $k+8$ nearby clusters.  

\end{lemma}

\begin{lemma}
\label{vpsym lem}

Let $v$ be a very poor $1$-cluster in a symmetric orientation, and let $a$, $b$ and $c$ be the vertices in $D$ at distance-$2$ from $v$.  There exist $3$ open $3$-clusters, $C_1$, $C_2$ and $C_3$, such that $v$ is in a head position of all $3$ and exactly one of $a$, $b$ and $c$ is in a shoulder position of each of $C_1$, $C_2$ and $C_3$.  Furthermore, if each of $C_1$, $C_2$ and $C_3$ is uncrowded, then each of $a$, $b$ and $c$ is distance-$3$ from a closed $3$-cluster or \fourplus.

\end{lemma}

Corollaries \ref{vpsym cor} and \ref{vpsym cor2} follow directly from the proof of Lemma \ref{vpsym lem}.

\begin{corollary}
\label{vpsym cor}

None of the open $3$-clusters of which a very poor $1$-cluster in a symmetric orientation is in a head position is type-$1$ paired on top.

\end{corollary}

\begin{corollary}
\label{vpsym cor2}

Each of the vertices in $D$ at distance-$2$ from a very poor $1$-cluster in a symmetric orientation is distance-$2$ from no other very poor $1$-cluster.

\end{corollary}

\begin{lemma}
\label{vp lem}

Let $v$ be a very poor $1$-cluster, and let $u$, $w$ and $x$ be in the $u$-position, $w$-position and $x$-position, respectively, of $v$.  There exists an open $3$-cluster, $C_0$, such that $v$ and $x$ are in the head positions and $w$ is in a shoulder position of $C_0$; and one of the following holds:

\begin{itemize}

\item[(i)] $C_0$ is crowded.

\item[(ii)] There exists a closed $3$-cluster or \fourpluss at distance-$3$ from $w$.

\item[(iii)] There exists an open $3$-cluster, $C$, such that the tail position of $C$ is in $D$ and $u$, $v$ or $w$ is in the hand position on the finless side of $C$.

\item[(iv)] There exists a leaf, $\ell$, of a \fourplus, $C$, at distance-$2$ from $u$ such that $u$ is the only vertex in $D \setminus C$ at distance-$2$ from $\ell$.  If $C$ is a linear $4$-cluster, then $C$ has at most one one-turn position.  If $C$ is a curved $4$-cluster, then $C$ has at most one backwards position in $D$.  

\item[(v)] There exists a leaf, $\ell$, of a closed $3$-cluster, $C$, at distance-$2$ from $u$ such that $u$ is the only vertex in $D \setminus C$ at distance-$2$ from $\ell$ and $u$ is in a foot or arm position.  Either $C$ is type-$2$ paired and $u$ is in the arm position on the closed side of $C$, or $C$ has at most $6$ nearby poor $1$-clusters.

\item[(vi)] There exists an open $3$-cluster, $C$, such that $v$ or $w$ is in a hand position of $C$ and the hand and arm positions on the opposite side of $C$ are both in $D$.

\item[(vii)] There exists an open $3$-cluster, $C$, such that $u$ is in a foot position and $C$ is type-$1$ paired on top.

\end{itemize}

\end{lemma}

Corollary \ref{vpasym cor1} follows directly from the proof of Lemma \ref{vp lem}.

\begin{corollary}
\label{vpasym cor1}

Let $v$ be a very poor $1$-cluster in an asymmetric orientation, and let $w$ be in the $w$-position of $v$.  The open $3$-cluster of which $v$ is in a head position and $w$ is in a shoulder position is not type-$1$ paired on top.

\end{corollary}

\begin{corollary}
\label{vpasym cor2}

Let $u$ and $w$ be vertices in the $u$-position and $w$-position, respectively, of a very poor $1$-cluster, $v$.  Neither $u$ nor $w$ is distance-$2$ from any very poor $1$-cluster other than $v$.

\end{corollary}

\begin{corollary}
\label{vpasym cor3}

Consider a very poor $1$-cluster, $v$, in an asymmetric orientation, and let $x$ be in the $x$-position of $v$.  If $x$ is a very poor $1$-cluster, then $x$ is in an asymmetric orientation and $v$ is in the $x$-position of $x$.

\end{corollary}

\begin{lemma}
\label{vpshoulders lem}

If a very poor $1$-cluster is in a head position of an open $3$-cluster, $C$, then both shoulder positions of $C$ are in $D$.

\end{lemma}

\begin{lemma}
\label{type1paired lem}

If a poor $1$-cluster, $v$, is in a shoulder or arm position of an open $3$-cluster that is type-$1$ paired on top, then $v$ is nearby another \threeplus, $C_2$, such that if $C_2$ is an open $3$-cluster then $v$ is distance-$2$ from $C_2$ but not in an arm position and $C_2$ is not type-$1$ paired on top.
\end{lemma}

\begin{lemma}
\label{type2paired lem}

Let $C_1$ be a closed $3$-cluster that is type-$2$ paired with the open $3$-cluster, $C_2$.  If $C_1$ has $7$ nearby poor $1$-clusters, then the arm position, $n$, of $C_2$ is in $D$ and the hand position, $k$, on the same side is not in $D$.  Furthermore, if $n$ is a poor $1$-cluster, then $n$ is nearby a third \threeplus, $C_3$, such that if $C_3$ is an open $3$-cluster then $n$ is distance-$2$ from $C_3$ but not in an arm position and $C_3$ is not type-$1$ paired on top.

\end{lemma}

\section{Proof of Theorem \ref{theorem}}
\label{mainresult}


We employ the discharging method.  Suppose each vertex in $D$ has 1 charge.  We redistribute this charge so that each vertex in $G_H$ has at least $\frac{5}{12}$ charge.  Below are the discharging rules:

\begin{enumerate}

\item \label{one} If a vertex, $v$, is not in $D$ and has $k$ neighbors in $D$, then $v$ receives $\frac{5}{12k}$ from each of these neighbors.

\item \label{two} Let $v_{\frac{1}{3}}$ be a one-third vertex, and let $a$, $b$ and $c$ be the vertices adjacent to $v_{\frac{1}{3}}$.

\begin{enumerate}

\item \label{twoa} If $a$ and $b$ are 1-clusters and $c$ is in a \threeplus, $C$, then each of $a$ and $b$ receives $\frac{1}{72}$ from $C$.

\item \label{twob} If $a$ is a 1-cluster and $b$ and $c$ are in \threeplus s, $C_1$ and $C_2$, then $a$ receives $\frac{1}{36}$ from each of $C_1$ and $C_2$.  If $C_1=C_2$, then $a$ receives $2 \cdot \frac{1}{36} = \frac{1}{18}$ from $C_1$.

\end{enumerate}

\item \label{three} Let $v$ be a poor 1-cluster.

\begin{enumerate}

\item \label{threea} If $v$ is distance-2 from a closed 3-cluster or \fourplus, $C$, then $v$ receives $\frac{1}{24}$ from $C$.  

\item \label{threeb} If $v$ is distance-2 from an open 3-cluster, $C$, and has not received charge by previous rules, then $v$ receives $\frac{1}{24}$ from $C$ unless (a) $C$ is type-1 paired on top and $v$ is in a shoulder or arm position, or (b) $C$ is type-2 paired and $v$ is in an arm position.

\item \label{threec} If $v$ is distance-3 from a closed 3-cluster or \fourplus, $C$, and has not received charge by previous rules, then $v$ receives $\frac{1}{24}$ from $C$ unless $C$ is a closed 3-cluster and $v$ is in the arm position of an open 3-cluster that is type-2 paired with $C$.

\item \label{threed} Let $h$ be a non-poor 1-cluster that is not in a group of non-poor 1-clusters.  If $v$ is distance-2 from $h$ and has not received charge by previous rules, then $v$ receives $\frac{1}{24}$ from $h$.

\item \label{threee} If $v$ is distance-2 from a group of non-poor 1-clusters, $H$, and has not received charge by previous rules, then $v$ receives $\frac{1}{24}$ from $H$.

\end{enumerate}

\item \label{four} If a closed 3-cluster, $C_1$, and an open 3-cluster, $C_2$, are type-2 paired and the arm position of $C_2$ is in $D$ and the hand position on the same side is not in $D$, then $C_1$ receives $\frac{1}{24}$ from $C_2$.

\item \label{five} If a very poor 1-cluster, $v$, is in a head position of a crowded open 3-cluster, $C_0$, then $v$ receives $\frac{1}{24}$ from $C_0$.

\item \label{six} Let $v$ be a very poor 1-cluster in a symmetric orientation, and let $a$ be a distance-2 poor 1-cluster.  If $a$ is in a shoulder position of an open 3-cluster and distance-3 from a closed 3-cluster or \fourplus, $C$, then $a$ receives $\frac{1}{24}$ from $C$ in addition to any charge received by previous rules and $v$ receives $\frac{1}{24}$ from $a$.

\item \label{seven} Let $v$ be a very poor 1-cluster in an asymmetric orientation, and let $u$, and $w$ be poor 1-clusters in the $u$-position and $w$-position, respectively, of $v$.  The following applies only if $v$ does not receive charge by Discharging Rule 5.

\begin{enumerate}

\item \label{sevena} If $w$ is distance-3 from a closed 3-cluster or \fourplus, $C$, then $w$ receives $\frac{1}{24}$ from $C$ in addition to any charge received by previous rules and $v$ receives $\frac{1}{24}$ from $w$.

\item \label{sevenb} Let $C$ be an open 3-cluster, and let the tail position of $C$ be in $D$.  If $u$, $v$ or $w$ is in the hand position on the finless side of $C$, then $u$, $v$ or $w$, respectively, receives $\frac{1}{24}$ from $C$ in addition to any charge received by previous rules.  If $u$ or $w$ receives this charge, then $v$ receives $\frac{1}{24}$ from $u$ or $w$, respectively.

\item \label{sevenc} If $u$ distance-2 from a leaf, $\ell$, of a type-2 paired closed 3-cluster, a closed 3-cluster with at most 6 nearby poor 1-clusters or a \fourplus, $C$, such that $u$ is not in a shoulder or tail position, a one-turn position or a backwards position of a closed 3-cluster, a linear 4-cluster or a curved 4-cluster, respectively, and $u$ is the only vertex in $D \setminus C$ at distance-2 from $\ell$, then $u$ receives $\frac{1}{24}$ from $C$ in addition to any charge received by previous rules and $v$ receives $\frac{1}{24}$ from $u$.

\item \label{sevend} Let $C$ be an open 3-cluster, and let the hand and arm positions on one side of $C$ be in $D$.  If $v$ or $w$ is in the hand position on the other side of $C$, then $v$ or $w$, respectively, receives $\frac{1}{24}$ from $C$ in addition to any charge received by previous rules.  If $w$ receives this charge, then $v$ receives $\frac{1}{24}$ from $w$.

\item \label{sevene} Let $C$ be an open 3-cluster that is type-1 paired on top.  If $u$ is in the foot position of $C$, then $u$ receives $\frac{1}{24}$ from $C$ in addition to any charge received by previous rules and $v$ receives $\frac{1}{24}$ from $u$.

\end{enumerate}

\end{enumerate}

Now we verify that the above discharging rules allow each vertex in $G_H$ to retain at least $\frac{5}{12}$ charge.  For a given vertex, $v$, let $f(v)$ be the final charge of $v$ and let $f_n(v)$ be the charge of $v$ after Discharging Rule $n$.  And for a given $k$-cluster, $C$, where $k \geq 3$, let $f(C)$ be the final charge of $C$ and let $f_n(C)$ be the charge of $C$ after Discharging Rule $n$; note that $f(C) \geq \frac{5k}{12}$ immediately implies that each vertex in $C$ can retain at least $\frac{5}{12}$ charge.

If $v \in V(G_H)$, then $v \not \in D$ or $v \in D_1$ or $v \in D_3$ or $v \in D_{4^+}$.  We consider vertices not in $D$ in Claim \ref{notinD claim}.  We partition $D_1$ such that $D_1 = (\poorone \setminus D_1^{vp}) \cup D_1^{np} \cup D_1^{vp}$, and we consider each case separately in Claims \ref{poor1 claim}, \ref{nonpoor1 claim} and \ref{vp claim}, respectively.  Rather than considering individual vertices in $D_3$, we consider $\cK_3$.  We partition $\cK_3$ such that $\cK_3 = \cK_3^o \cup \cK_3^c$, and we consider each case separately in Claims \ref{open3 claim} and \ref{closed3 claim}, respectively.  We defer our discussion of \fourplus s until after Claim \ref{closed3 claim}.

\begin{claim}
\label{notinD claim}

If a vertex, $v$, is not in $D$, then $f(v) = \frac{5}{12}$.

\end{claim}

\begin{proof}

Let $v \not \in D$, and suppose $v$ has $k$ neighbors in $D$.  Then, by Discharging Rule 1, $v$ receives $\frac{5}{12k}$ from each of these neighbors.  That is, $f(v)=f_1(v)=k \cdot \frac{5}{12k} = \frac{5}{12}$.
\end{proof}


\begin{proposition}
\label{poor1discharging prop}

Any poor $1$-cluster at distance-$2$ from an open $3$-cluster or within distance-$3$ of a closed $3$-cluster or \fourpluss receives charge by Discharging Rules $3\emph{a}$-$3\emph{c}$.

\end{proposition}

\begin{proof}

Let $v \in \poorone$ such that $v$ is distance-$2$ from an open $3$-cluster or within distance-$3$ of a closed $3$-cluster or \fourplus.  Then, by Rules 3a-3c, $v$ receives $\frac{1}{24}$ from a nearby \threepluss except, potentially, in 2 cases.  In the first case, $v$ is in a shoulder or arm position of an open 3-cluster which is type-1 paired on top.  But, by Lemma \ref{type1paired lem}, $v$ is nearby another \threeplus, $C_1$, such that if $C_1$ is an open 3-cluster then $v$ is distance-2 from $C_1$ but not in an arm position and $C_1$ is not type-1 paired on top; therefore, $v$ receives $\frac{1}{24}$ from $C_1$ by Discharging Rules 3a-3c.  In the second case, $v$ is in an arm position of an open 3-cluster which is type-2 paired with a closed 3-cluster.  But, by Lemma \ref{type2paired lem}, $v$ is nearby a third \threeplus, $C_2$, such that if $C_2$ is an open 3-cluster then $v$ is distance-2 from $C_2$ but not in an arm position and $C_2$ is not type-1 paired on top; therefore, $v$ receives $\frac{1}{24}$ from $C_2$ by Discharging Rules 3a-3c.
\end{proof}

\begin{claim}
\label{poor1 claim}

If a poor $1$-cluster, $v$, is not very poor, then $f(v) = \frac{5}{12}$.

\end{claim}

\begin{proof}

Let $v \in \poorone \setminus D_1^{vp}$.  Then, $v$ must send charge to all 3 neighbors, each of which has exactly 2 neighbors in $D$.  That is, $f_1(v) = 1 - 3 \cdot \frac{5}{12 \cdot 2} = \frac{9}{24}$.  But $v$ is not very poor; therefore, $v$ is distance-2 from a \threepluss or non-poor 1-cluster or distance-3 from a closed 3-cluster or \fourplus.  If $v$ is distance-2 from an open 3-cluster or within distance-3 of a closed 3-cluster or \fourplus, then $v$ receives $\frac{1}{24}$ by Rules 3a-3c (Proposition \ref{poor1discharging prop}); if not, then $v$ receives charge from a distance-2 non-poor 1-cluster by Rules 3d-3e.  Thus, we have shown that $v$ will receive $\frac{1}{24}$ from a nearby cluster.  Therefore, $f_3(v) = f_1(v) + \frac{1}{24} = \frac{9}{24} + \frac{1}{24} = \frac{5}{12}$.  If $v$ is distance-2 from a very poor 1-cluster, $w$, then $v$ may need to receive charge from a nearby cluster and send charge to $w$ by Rules 6-7.  If $w$ is in a symmetric orientation, then $v$ is not distance-2 from any very poor 1-cluster other than $w$ (Corollary \ref{vpsym cor2}).  If $w$ is in an asymmetric orientation and $v$ is in the $u$-position or $w$-position of $w$, then $v$ is not distance-2 from any very poor 1-cluster other than $w$ (Corollary \ref{vpasym cor2}).  Thus, if Rules 6-7 require $v$ to receive and send charge, then $v$ must only send charge to one very poor 1-cluster.  Then, if Rule 6 is applicable, $v$ receives $\frac{1}{24}$ and sends $\frac{1}{24}$.  The same is true of Rules 7a-7e.  Therefore, Rules 6-7 have no effect on the final charge of $v$.  Therefore, $f(v) = f_7(v) = f_3(v) = \frac{5}{12}$.
\end{proof}

\begin{claim}
\label{nonpoor1 claim}

For every non-poor $1$-cluster, $v$, $f(v) \geq \frac{5}{12}$.

\end{claim}

\begin{proof}

Let $v \in D_1^{np}$.  Then, $v$ must send charge to all 3 of its neighbors, at least one of which has 3 neighbors in $D$.  Therefore, $f_1(v) \geq 1 - \left(2 \cdot \frac{5}{12 \cdot 2} + \frac{5}{12 \cdot 3} \right) = \frac{4}{9}$.

Suppose $v$ is in a group of non-poor 1-clusters, $H$.  Then, $f_1(H) \geq 3 \cdot f_1(v) = \frac{4}{3}$.  By Discharging Rule 3e, $H$ may need to send charge to distance-2 poor 1-clusters that do not receive charge by Rules 3a-3d; therefore, $H$ must send charge to a distance-2 poor 1-cluster, $u$, only if $u$ is neither distance-2 from an open 3-cluster nor within distance-3 of a closed 3-cluster or \fourpluss (Proposition \ref{poor1discharging prop}).  Then, $H$ must send charge to at most 2 distance-2 poor 1-clusters (Lemma \ref{nonpoor lem1}).  Therefore, $f(H) = f_3(H) \geq f_1(H) - 2 \cdot \frac{1}{24} \geq \frac{4}{3} - \frac{1}{12} = \frac{15}{12} = 3 \cdot \frac{5}{12}$.  Therefore, $f(v) \geq \frac{5}{12}$.

Now, suppose $v$ shares a one-third vertex with a non-poor 1-cluster, $w$, and a \threeplus, $C$.    By Discharging Rule 2a, each of $v$ and $w$ receives $\frac{1}{72}$ from $C$.  And, by Discharging Rule 3d, each of $v$ and $w$ may need to send charge to distance-2 poor 1-clusters that do not receive charge by Rules 3a-3c; therefore, $v$ and $w$ send charge to a poor 1-cluster, $u$, only if $u$ is neither distance-2 from an open 3-cluster nor within distance-3 of a closed 3-cluster or \fourpluss (Proposition \ref{poor1discharging prop}).  Then, each of $v$ and $w$ has at most one distance-2 poor 1-cluster that does not receive charge by Rules 3a-3c (Lemma \ref{nonpoor lem2}).  Therefore, $f(v) = f_3(v) \geq f_2(v) - \frac{1}{24} = \left (f_1(v) + \frac{1}{72} \right ) - \frac{1}{24} \geq \left (\frac{4}{9} + \frac{1}{72} \right ) - \frac{1}{24} = \frac{5}{12}$

Now, suppose $v$ shares a one-third vertex with \threeplus s only.  Then, by Rule 2b, $v$ receives $\frac{1}{18}$ from these \threeplus s.  By Rule 3d, $v$ may need to send charge to distance-2 poor 1-clusters.  However, $v$ has at most 2 distance-2 poor 1-clusters (Lemma \ref{nonpoor lem3}).  Therefore, $f(v) = f_3(v) \geq f_2(v) - 2 \cdot \frac{1}{24} = \left (f_1(v) + \frac{1}{18} \right ) - \frac{1}{12} \geq \left (\frac{4}{9} + \frac{1}{18} \right ) - \frac{1}{12} = \frac{5}{12}$.
\end{proof}

\begin{claim}
\label{vp claim}

For every very poor $1$-cluster, $v$, $f(v) \geq \frac{5}{12}$.

\end{claim}

\begin{proof}

Let $v \in D_1^{vp}$.  We saw above that $f_1(v) = \frac{9}{24}$.  

If $v$ is in a symmetric orientation, then $v$ is in a head position of 3 open 3-clusters, $C_1$, $C_2$ and $C_3$ (Lemma \ref{vpsym lem}).  If any of $C_1$, $C_2$ and $C_3$ is crowded, say $C_1$, then $v$ receives $\frac{1}{24}$ from $C_1$.  Then, $f_5(v) = f_1(v) + \frac{1}{24} = \frac{5}{12}$.  Let $a$, $b$ and $c$ be the vertices in $D$ at distance-2 from $v$.  Since $v \in D_1^{vp}$, we must have $a,b,c \in \poorone$.  However, exactly one of $a$, $b$ and $c$ is in a shoulder position of each of $C_1$, $C_2$ and $C_3$ (Lemma \ref{vpsym lem}); therefore $a,b,c \not \in D_1^{vp}$.  Then, Rules 6-7 do not require $v$ to send any charge; therefore, $f(v) \geq f_5(v) \geq \frac{5}{12}$.
Now, if each of $C_1$, $C_2$ and $C_3$ is uncrowded, then each of $a$, $b$ and $c$ is distance-3 from a closed 3-cluster or \fourpluss (Lemma \ref{vpsym lem}).  Discharging Rule 5 is not applicable, but, by Discharging Rule 6, each of $a$, $b$ and $c$ receives $\frac{1}{24}$ from a distance-3 closed 3-cluster or \fourpluss and sends $\frac{1}{24}$ to $v$.  Rule 7 is not applicable; therefore $f(v) = f_6(v) = f_1(v) + 3 \cdot \frac{1}{24} = \frac{9}{24} + \frac{1}{8} = \frac{1}{2} > \frac{5}{12}$.

Now, suppose $v$ is in an asymmetric orientation.  Let $u$, $w$ and $x$ be in the $u$-position, $w$-position and $x$-position, respectively.  Now, $u,w \not \in D_1^{vp}$ (Corollary \ref{vpasym cor2}); and if $x \in D_1^{vp}$, then $v$ is in the $x$-position of $x$ (Corollary \ref{vpasym cor3}).  Therefore, none of the Discharging Rules requires $v$ to send charge.  Now, by Lemma \ref{vp lem}, $v$ and $x$ are in the head positions of an open 3-cluster, $C_0$, and one of the following holds:

\begin{itemize}

\item $C_0$ is crowded.
In this case, $v$ receives $\frac{1}{24}$ from $C_0$ by Rule 5.

\item There exists a closed 3-cluster or \fourpluss at distance-3 from $w$.
In this case, $v$ receives $\frac{1}{24}$ from $w$ by Rule 7a.

\item There exists an open 3-cluster, $C$, such that the tail position of $C$ is in $D$ and $u$, $v$ or $w$ is in the hand position on the finless side of $C$.  In this case, $v$ receives $\frac{1}{24}$ from $u$, $w$ or $C$ by Rule 7b.

\item There exists a leaf, $\ell$, of a \fourplus, $C$, at distance-2 from $u$ such that $u$ is not in a one-turn position or a backwards position of a linear 4-cluster or a curved 4-cluster, respectively, and $u$ is the only vertex in $D \setminus C$ at distance-2 from $\ell$.  In this case, $v$ receives $\frac{1}{24}$ from $u$ by Rule 7c.

\item There exists a leaf, $\ell$, of a closed 3-cluster, $C$, at distance-2 from $u$ such that $u$ is in a foot or arm position and $u$ is the only vertex in $D \setminus C$ at distance-2 from $\ell$.  Furthermore, either $C$ is type-2 paired or $C$ has at most 6 nearby poor 1-clusters.  In this case, $v$ receives $\frac{1}{24}$ from $u$ by Rule 7c.

\item There exists an open 3-cluster, $C$, such that $v$ or $w$ is in a hand position of $C$ and the hand and arm positions on the other side of $C$ are in $D$.  In this case, $v$ receives $\frac{1}{24}$ from $C$ or $w$ by Rule 7d.

\item There exists an open 3-cluster, $C$, such that $u$ is in a foot position and $C$ is type-1 paired on top.  In this case, $v$ receives $\frac{1}{24}$ from $u$ by Rule 7e.

\end{itemize}
Therefore, $v$ receives at least $\frac{1}{24}$ from a nearby cluster by Discharging Rules 5 and 7.  Therefore, $f(v) = f_7(v) \geq f_1(v) + \frac{1}{24} = \frac{9}{24} + \frac{1}{24} = \frac{5}{12}$.
\end{proof}

\begin{proposition}
\label{open3 prop}

If an open $3$-cluster, $C$, is neither type-$1$ paired nor type-$2$ paired and $f_2(C) \geq \frac{34+ P(C)}{24}$, then $f(C) \geq 3 \cdot \frac{5}{12}$.


\end{proposition}

\begin{proof}

Let $C$ be an open 3-cluster that is neither type-1 nor type-2 paired.  Note that Rule 3b requires $C$ to send at most $P(C) \cdot \frac{1}{24}$, Rule 5 requires $C$ to send at most $\frac{2}{24}$ and Rules 7b and 7d each require $C$ to send at most $\frac{1}{24}$ (Lemma \ref{finless lem}); therefore, $C$ sends at most $\frac{ P(C) + 4}{24}$ by Rules 3-7.  Therefore, if $f_2(C) \geq \frac{34+P(C)}{24}$, then $f(C) \geq \frac{30}{24} = 3 \cdot \frac{5}{12}$.
\end{proof}

\begin{claim}
\label{open3 claim}

For every open $3$-cluster, $C$, $f(C) \geq 3 \cdot \frac{5}{12}$.

\end{claim}

\begin{proof}

Consider an open 3-cluster, $C$.  Then, by Discharging Rule 1, the middle vertex of $C$ must send $\frac{5}{12}$.  Each of the leaf vertices has at least one distance-2 vertex in $D \setminus C$ (Proposition \ref{leaves prop}); therefore, each of the leaf vertices must send at most $\frac{5}{12 \cdot 2} + \frac{5}{12} = \frac{15}{24}$ by Rule 1.  Thus, $f_1(C) \geq 3 - \frac{5}{12} - 2 \cdot \frac{15}{24} = \frac{4}{3}$.

First, suppose $C$ is type-1 paired on top (see Figure \ref{type-1 paired}).  Now, the shoulder position of $C$ is in $D$ or the arm position is in $D$ (Proposition \ref{leaves prop}).  Suppose exactly one is in $D$, and let $v$ be this vertex.  Then, $C$ has exactly 2 distance-2 vertices in $D$; therefore, $f_1(C) = \frac{4}{3}$ and Rule 2 does not apply.  By Rule 3b, $C$ does not send charge to $v$ even if $v \in \poorone$, but $C$ may need to send charge to the poor 1-cluster, $u$, in the foot position.  Therefore, $f_3(C) \geq f_1(C) - \frac{1}{24} = \frac{31}{24}$.  Now, $C$ is not type-2 paired (Corollary \ref{paired cor}); therefore, Rule 4 does not apply.  Since at least one of the shoulder positions of $C$ is not in $D$, Rule 5 does not apply (Lemma \ref{vpshoulders lem}).  Since the tail position of a paired 3-cluster is not in $D$, Rule 7b does not apply.  At least one of the hand positions of $C$ is not in $D$; therefore, Rule 7d does not apply.  By Rule 7e, $C$ may need to send $\frac{1}{24}$ to $u$; therefore, $f(C) = f_7(C) \geq f_3(C) - \frac{1}{24} \geq \frac{31}{24} - \frac{1}{24} = 3 \cdot \frac{5}{12}$.

Now, suppose $C$ is type-1 paired on top and both the shoulder position and the arm position of $C$ are in $D$.  Then, $f_1(C) = 3 - 3 \cdot \frac{5}{12} - \frac{5}{12 \cdot 2} - \frac{5}{12 \cdot 3} = \frac{101}{72}$.  Let $v$ and $w$ be the shoulder and arm positions of $C$, respectively.  If $v$ and $w$ are 1-clusters, then $C$ sends $\frac{1}{72}$ to both by Discharging Rule 2a.  If exactly one of $v$ and $w$ is a 1-cluster, say $v$, then $C$ sends $\frac{1}{36}$ to $v$ by Rule 2b.  In both cases, $C$ sends a total of $\frac{1}{36}$; thus, $f_2(v) \geq f_1(v) - \frac{1}{36} = \frac{33}{24}$.  By Rule 3b, $C$ may need to send $\frac{1}{24}$ to the poor 1-cluster, $u$, in the foot position; therefore, $f_3 \geq \frac{32}{24}$.  Again, Rules 4-7d do not apply.  By Rule 7e, $C$ may need to send $\frac{1}{24}$ to $u$; therefore, $f(C) = f_7(C) \geq f_3(C) - \frac{1}{24} \geq \frac{32}{24} - \frac{1}{24} = \frac{31}{24} > 3 \cdot \frac{5}{12}$.

Suppose $C$ is type-2 paired (see Figure \ref{type-2 paired}).  One of the shoulder positions of $C$ is in $D$.  On the other side of $C$, the shoulder position is in $D$ or the arm position is in $D$ (Proposition \ref{leaves prop}).  Let $v$ and $w$ be the shoulder and arm positions of $C$, respectively.  Then, $w \in D_1^{np}$ or $w \in \poorone$ or $w \in D_{3^+}$ or $w \not \in D$.  If $w \in D_1^{np}$ and $v \not \in D$, then $f_1(C) = \frac{32}{24}$.  Now, $C$ is not type-1 paired (Corollary \ref{paired cor}); therefore, $C$ may need to send charge to the shoulder position on the side opposite $w$ by Rule 3b.  Then, $f_3(C) \geq \frac{31}{24}$.  By Rule 4, $C$ sends $\frac{1}{24}$ to the closed 3-cluster with which $C$ is type-2 paired.  Then, $f_4(C) \geq \frac{30}{24}$.  Rule 5 does not apply (Lemma \ref{vpshoulders lem}).  And Rules 7b, 7c and 7e do not apply.  Therefore, $f(C) \geq \frac{30}{24} = 3 \cdot \frac{5}{12}$.  If $w \in D_1^{np}$ and $v \in D$, then $f_1(C) = \frac{101}{72}$.  By Rule 2, $C$ may need to send $\frac{1}{36}$ to distance-2 non-poor 1-clusters.  Then, $f_2(C) \geq \frac{33}{24}$.  By Rule 3b, $C$ sends at most $\frac{1}{24}$.  By Rule 4, $C$ sends $\frac{1}{24}$.  Since $v \in D_1^{np} \cup D_{3^+}$, at least one of the head positions of $C$ is not a very poor 1-cluster; therefore, $C$ sends at most $\frac{1}{24}$ by Rule 5.  Rules 6-7 do not apply.  Therefore, $f(C) \geq \frac{30}{24} = 3 \cdot \frac{5}{12}$.
Now, if $w \in \poorone$, then $v \not \in D$ (Corollary \ref{poor1 cor}) and $f_1(C) = \frac{32}{24}$.  Then, $C$ sends at most $\frac{1}{24}$ by Rule 3b, and $C$ sends $\frac{1}{24}$ by Rule 4.  Rules 5-7 do not apply.  Therefore, $f(C) \geq \frac{30}{24} = 3 \cdot \frac{5}{12}$.  If $w \in D_{3^+}$ and $v \not \in D$, then $f_1(C) = \frac{32}{24}$.  By Rule 3b, $C$ sends at most $\frac{1}{24}$.  If the hand position adjacent to $w$ is in $D$, then $C$ does not send charge by Rule 4 but $C$ may need to send charge by Rule 7d.  If the hand position adjacent to $w$ is not in $D$, then $C$ must send charge by Rule 4 but not by Rule 7d.  Therefore, $C$ sends charge by at most one of Rules 4 and 7d.  No other rules apply.  Therefore, $f(C) \geq \frac{30}{24} = 3 \cdot \frac{5}{12}$.  If $w \in D_{3^+}$ and $v \in D$, then $f_1(C) = \frac{101}{72}$.  By Rule 2, $C$ sends at most $\frac{1}{36}$ to $v$.  Then, $f_2(C) \geq \frac{33}{24}$.  By Rule 3b, $C$ sends at most $\frac{1}{24}$.  Therefore, $f_3(C) \geq \frac{32}{24}$.  Now, $v \in D_1^{np} \cup D_{3^+}$ and $v$ is distance-2 from a head position of $C$; therefore, at most one of the head positions of $C$ is a very poor 1-cluster.  Then, $C$ sends at most $\frac{1}{24}$ by Rule 5.  Again, $C$ sends charge by at most one of Rules 4 and 7d.  Therefore, $f(C) \geq \frac{30}{24} = 3 \cdot \frac{5}{12}$.  If $w \not \in D$, then $v \in D$ (Proposition \ref{leaves prop}).  Then, $f_1(C) = \frac{32}{24}$ and both shoulder positions of $C$ are in $D$.  If both are poor 1-clusters, then $C$ sends at most $\frac{2}{24}$ by Rule 3b and sends no charge by other rules.  Therefore, $f(C) \geq 3 \cdot \frac{5}{12}$.  If one is not a poor 1-cluster, then at most one of the head positions of $C$ is a very poor 1-cluster.  Therefore, $C$ sends at most $\frac{1}{24}$ by Rule 3b and at most $\frac{1}{24}$ by Rule 5, and $C$ sends no charge by other rules.  Therefore, $f(C) \geq \frac{30}{24}$.  If neither is a poor 1-cluster, then $C$ does not send charge by any rule; therefore, $f(C) = \frac{32}{24} > 3 \cdot \frac{5}{12}$.

Suppose $C$ is neither type-1 paired nor type-2 paired and $P(C) = 2$.  Then, 4 and 7e do not apply and $f_1(C) = \frac{32}{24}$.  Since $P(C) = 2$, there exists no one-third vertex adjacent to $C$; therefore, Rule 2 does not apply.  Now, $C$ may need to send charge by Rules 3b, 5, 7b, and 7d.  If $C$ must send charge by Rule 5, then both shoulder positions are in $D$ (Lemma \ref{vpshoulders lem}); therefore, the arm positions and the tail position of $C$ are not in $D$ and, hence, Rules 7b and 7d do not apply.  If $C$ must send charge by Rule 7b, then the tail position of $C$ is in $D$; therefore, the arm and shoulder positions of $C$ are not in $D$ and, hence, Rules 5 and 7d do not apply.  If $C$ must send charge by Rule 7d, then an arm position of $C$ is in $D$.  Since each leaf must have at least one distance-2 vertex in $D$ (Proposition \ref{leaves prop}) and $P(C) = 2$, the tail position of $C$ is not in $D$ and at least one of the shoulder positions is not in $D$; therefore, Rules 5 and 7b do not apply.  Thus, we have shown that $C$ sends charge by at most one of Rules 5, 7b and 7d.  First, suppose $C$ sends by none of Rules 5, 7b and 7d.  Then, $C$ is uncrowded and the tail position of $C$ is not in $D$; therefore, $C$ has exactly 2 distance-2 poor 1-clusters.  Then, $C$ sends $\frac{2}{24}$ by Rule 3b and no charge by any other rule; therefore, $f(C) = \frac{30}{24}$.  Suppose $C$ sends charge by Rule 5.  Then, both shoulder positions of $C$ are in $D$ (Lemma \ref{vpshoulders lem}), and $C$ is crowded; therefore, one of the shoulder positions is not a poor 1-cluster.  But then one of the head positions of $C$ is distance-2 from a non-poor 1-cluster or \threeplus; therefore, there exists only one very poor 1-cluster in a head position of $C$.  Then, $C$ sends $\frac{1}{24}$ by 3b and $\frac{1}{24}$ by Rule 5; therefore, $f(C) = \frac{30}{24}$.  Now suppose $C$ sends charge by Rule 7b.  Then, the tail position of $C$ is in $D$.  Since $P(C) = 2$, the tail position is the only vertex in $D$ at distance-2 from $C$; therefore, there exists at most one distance-2 poor 1-cluster.  Then, $f_3(C) \geq \frac{31}{24}$ and Rule 7b requires $C$ to send at most $\frac{1}{24}$; therefore, $f(C) \geq \frac{30}{24}$.  Finally, suppose $C$ sends charge by Rule 7d.  Then, the hand and arm positions on one side of $C$ are in $D$; therefore, at least one of the vertices in $D$ at distance-2 from $C$ is not a poor 1-cluster.  Then, $f_3(C) \geq \frac{31}{24}$ and, by Rule 7d, $C$ sends at most $\frac{1}{24}$; therefore, $f(C) \geq \frac{30}{24}$.

Suppose $C$ is neither type-1 paired nor type-2 paired and $P(C) = 3$.  If $C$ is adjacent to no one-third vertices, then $f_1(C) = 3 - 3 \cdot \frac{5}{24} - 2 \cdot \frac{5}{12} = \frac{37}{24}$; therefore, $f(C) \geq \frac{30}{24}$ (Proposition \ref{open3 prop}).  If $C$ is adjacent to a one-third vertex, then $f_1(C) = \frac{101}{72}$ and $f_2(C) \geq \frac{99}{72} = \frac{33}{24}$.  Since $C$ is adjacent to a one-third vertex and $P(C) = 3$, there exists at most one distance-2 poor 1-cluster; therefore, $f_3(C) \geq \frac{32}{24}$.  First, suppose the tail position of $C$ is in $D$.  Then, a foot position is also in $D$ and no other distance-2 vertices are in $D$; therefore, Rules 5 and 7d do not apply.  By Rule 7b, $C$ sends at most $\frac{1}{24}$; therefore, $f(C) \geq \frac{31}{24}$.  Now, suppose the tail position of $C$ is not in $D$; therefore, Rule 7b does not apply.  If $C$ does not send charge by Rule 5, then $C$ sends at most by Rule 7d and $f(C) \geq \frac{31}{24}$.  If $C$ sends charge by Rule 5, then both shoulder positions of $C$ are in $D$ (Lemma \ref{vpshoulders lem}).  However, $C$ is adjacent to a one-third vertex; therefore, one of the shoulder positions is not a poor 1-cluster.  Then, at most one of the head positions is a very poor 1-cluster.  Therefore, $C$ sends at most $\frac{1}{24}$ by Rule 5 and at most $\frac{1}{24}$ by Rule 7d.  Therefore, $f(C) \geq \frac{30}{24}$.

Suppose $C$ is neither type-1 paired nor type-2 paired and $P(C) = 4$.  First, suppose $C$ is adjacent to no one-third vertices.  Then, $f_1(C) = \frac{42}{24}$; therefore, $f(C) \geq \frac{30}{24}$ (Proposition \ref{open3 prop}).  Now, suppose $C$ is adjacent to exactly one one-third vertex.  Then, $f_1(C) = \frac{116}{72}$.  By Rule 2, $C$ sends at most $\frac{1}{36}$; therefore, $f_2(C) = \frac{114}{72} = \frac{38}{24}$.  Therefore, $f(C) \geq \frac{30}{24}$ (Proposition \ref{open3 prop}).  Now, suppose $C$ is adjacent to exactly 2 one-third vertices.  Then, $f_1(C) = \frac{53}{36}$.  By Rule 2, $C$ sends at most $2 \cdot \frac{1}{36}$; therefore, $f_2(C) \geq \frac{51}{36} = \frac{34}{24}$.  Since $C$ is adjacent to 2 one-third vertices and $P(C) = 4$, there exist no distance-2 poor 1-clusters; therefore, $f_3(C) \geq f_2(C) \geq \frac{34}{24}$.  In total, Rules 5-7 require $C$ to send at most $\frac{4}{24}$; therefore, $f(C) \geq \frac{30}{24}$.

Suppose $C$ is neither type-1 paired nor type-2 paired and $P(C) \geq 5$.  Then, $f_2(C) \geq \frac{39}{24}$.  Now, an open 3-cluster has at most 4 distance-2 poor 1-clusters.  Therefore, $C$ sends at most $\frac{4}{24}$ by Rule 3b.  By Rules 5-7, $C$ sends at most $\frac{4}{24}$.  Therefore, $f(C) \geq \frac{31}{24}$.
\end{proof}

\begin{proposition}
\label{dist3 prop}

A closed $3$-cluster or \fourpluss sends charge to a distance-$3$ poor $1$-cluster by at most one of Rules $3$\emph c, $6$ and $7$\emph a.

\end{proposition}

\begin{proof}

Let $C_1$ be closed 3-cluster or \fourplus.  If a poor 1-cluster, $v$, is distance-2 from a very poor 1-cluster in a symmetric orientation, then $v$ is distance-2 from exactly one very poor 1-cluster.  If $v$ is in the $u$-position or $w$-position of a very poor 1-cluster in an asymmetric orientation, then $v$ is distance-2 from exactly one very poor 1-cluster.  Therefore, $C_1$ sends charge to a poor 1-cluster by at most one of Rules 6 and 7a.  If $C_1$ sends charge to a poor 1-cluster, $a$, by Rule 6, then $a$ is distance-2 from a very poor 1-cluster in a symmetric orientation and in a shoulder position of an open 3-cluster, $C_2$ (Lemma \ref{vpsym lem}).  Now, $C_2$ is not type-1 paired on top (Corollary \ref{vpsym cor}); therefore, $a$ receives charge from $C_2$ by Rule 3b and not from $C_1$ by Rule 3c.  Therefore, $C_1$ sends charge to $a$ by at most one of Rules 3c and 6.  If $C_1$ sends charge to a poor 1-cluster, $w$, by Rule 7a, then $w$ is in the $w$-position of a very poor 1-cluster in an asymmetric orientation; therefore, $w$ is in the shoulder position of an open 3-cluster, $C_0$ (Lemma \ref{vp lem}).  Now, $C_0$ is not type-1 paired on top (Corollary \ref{vpasym cor1}); therefore, $w$ receives charge from $C_0$ by Rule 3b and not from $C_1$ by Rule 3c.  Therefore, $C_1$ sends charge to $w$ by at most one of Rules 3c and 7a. 
\end{proof}


\begin{proposition}
\label{c3 prop}

If $C$ is a closed $3$-cluster and $f_2(C) \geq \frac{43}{24}$, then $f(C) \geq 3 \cdot \frac{5}{12}$.

\end{proposition}

\begin{proof}

By Rules 3c, 6 and 7a, $C$ sends at most $\frac{1}{24}$ to each distance-3 poor 1-cluster (Proposition \ref{dist3 prop}).  By Rule 3a, $C$ sends at most $\frac{1}{24}$ to each distance-2 poor 1-cluster.  Now, $C$ has at most 11 nearby poor 1-clusters (Lemma \ref{k+8}); therefore, by Rules 3-7a, $C$ sends at most $\frac{11}{24}$.  By Rule 7c, $C$ sends at most $\frac{2}{24}$ to distance-2 poor 1-clusters.  Therefore, $C$ sends at most $\frac{13}{24}$ by Rules 3-7.
\end{proof}

\begin{corollary}
\label{c3 prop cor}

A closed $3$-cluster sends at most $\frac{11}{24}$ by Rules $3$-$7$\emph a and at most $\frac{2}{24}$ by Rule $7$\emph c.

\end{corollary}

\begin{claim}
\label{closed3 claim}

For every closed $3$-cluster, $C$, $f(C) \geq 3 \cdot \frac{5}{12}$.

\end{claim}

\begin{proof}

Consider a closed 3-cluster, $C_1$, and let $P(C_1) = 3$.  Then, $f_1(C_1) = \frac{37}{24}$.  By Lemma \ref{closed3 lem}, $C_1$ has at most 8 nearby poor 1-clusters; however, if $C_1$ has 8 such clusters, at least one of the poor 1-clusters at distance-3, $v$, is distance-2 from another \threeplus, $C_2$, such that

\begin{itemize}

\item[(a)] if $C_2$ is an open 3-cluster, then $v$ is not in a shoulder position;

\item[(b)] if $C_2$ is an open 3-cluster and $v$ is in an arm position, then $C_2$ is not type-1 paired; if $C_2$ is type-2 paired, then $C_1$ is type-2 paired with $C_2$.  

\end{itemize}
Therefore, $v$ receives charge from $C_2$ by Rules 3a-3b and not from $C_1$ by Rule 3c; additionally, $C_1$ does not send charge to $v$ by Rules 6 and 7a.  Then, $C_1$ sends at most $\frac{7}{24}$ by Rules 3, 6 and 7a (Proposition \ref{dist3 prop}).  If $C_1$ sends charge by Rule 7c, then $C_1$ has at most 6 nearby poor 1-clusters or $C_1$ is type-2 paired.  A poor 1-cluster, $u$, receives charge by Rule 7c only if $u$ is the only vertex in $D \setminus C_1$ at distance-2 from a leaf of $C_1$ and $u$ is not in a shoulder or tail position; therefore, $C_1$ sends at most $\frac{2}{24}$ by Rule 7c.  If $C_1$ sends $\frac{2}{24}$ by Rule 7c, then the shoulder positions and the tail position of $C_1$ are not in $D$; therefore, $C_1$ has at most 5 nearby poor 1-clusters (Lemma \ref{closed3 lem}).  Then, $C$ sends at most $\frac{5}{24}$ by Rules 3, 6 and 7a (Proposition \ref{dist3 prop}) and $\frac{2}{24}$ by Rule 7c, and $f(C_1) \geq \frac{30}{24}$.  If $C_1$ sends $\frac{1}{24}$ by Rule 7c and $C_1$ has at most 6 nearby poor 1-clusters, then $C_1$ sends at most $\frac{6}{24}$ by Rules 3, 6 and 7a (Proposition \ref{dist3 prop}) and $\frac{1}{24}$ by Rule 7c, and $f(C_1) \geq \frac{30}{24}$.  If $C_1$ sends $\frac{1}{24}$ by Rule 7c and $C_1$ is type-2 paired with the open 3-cluster, $C_2$, then the argument is identical to the previous case unless $C_1$ has 7 nearby poor 1-clusters.  In this case, the arm position of $C_2$ is in $D$ and the hand position on the same side is not in $D$ (Lemma \ref{type2paired lem}); therefore, $C_1$ receives $\frac{1}{24}$ from $C_2$ by Rule 4.  Then, $C_1$ sends at most $\frac{7}{24}$ by Rules 3, 6 and 7a (Proposition \ref{dist3 prop}) and $\frac{1}{24}$ by Rule 7c; however, $C_1$ receives $\frac{1}{24}$ by Rule 4 and, therefore, $f(C_1) \geq \frac{30}{24}$.


Consider a closed 3-cluster, $C$, and let $P(C) = 4$.  Then, either $C$ is adjacent to no one-third vertices or $C$ is adjacent to exactly one one-third vertex.  In the former case, $f_1(C) = \frac{42}{24}$; and since $C$ is not adjacent to any one-third vertices, $f_2(C) = \frac{42}{24}$.  Now, $C$ sends at most $\frac{11}{24}$ by Rules 3-7a and at most $\frac{2}{24}$ by Rule 7c (Corollary \ref{c3 prop cor}).  Since $C$ is adjacent to no one-third vertices, $C$ is closed by a single vertex; therefore, since $P(C) = 4$, one of the leaves of $C$ is distance-2 from 2 vertices in $D \setminus C$; therefore, $C$ sends at most $\frac{1}{24}$ by Rule 7c.  Therefore, $f(C) \geq \frac{30}{24}$.  Now suppose $C$ is adjacent to a one-third vertex.  Then, $f_2(C) \geq \frac{38}{24}$.  Now, $C$ has at most 8 nearby poor 1-clusters (Lemma \ref{closed3 lem2}); therefore, $C$ sends at most $\frac{8}{24}$ by Rules 3-7a (Proposition \ref{dist3 prop}).  Therefore, if $C$ sends no charge by Rule 7c, then $f(C) \geq \frac{30}{24}$.  If $C$ sends $\frac{1}{24}$ by Rule 7c, then an arm position or foot position of $C$ is a poor 1-cluster.  But then $C$ has at most 7 nearby poor 1-clusters (Lemma \ref{closed3 lem2}) and $f_{7a} \geq \frac{31}{24}$; therefore, $f(C) \geq \frac{30}{24}$.  If $C$ sends $\frac{2}{24}$ by Rule 7c, then 2 arm or foot positions are poor 1-clusters.  But then $C$ has at most 6 nearby poor 1-clusters (Lemma \ref{closed3 lem2}) and $f_{7a} \geq \frac{32}{24}$; therefore, $f(C) \geq \frac{30}{24}$.  

Let $P(C) = 5$.  Then, either $C$ is adjacent to no one-third vertices, one one-third vertex or 2 one-third vertices.  In the first case, $f_2(C) = \frac{47}{24}$; therefore, $f(C) \geq \frac{30}{24}$ (Proposition \ref{c3 prop}).  In the second case, $f_2(C) \geq \frac{43}{24}$; therefore, $f(C) \geq \frac{30}{24}$ (Proposition \ref{c3 prop}).  In the last case, $f_2(C) \geq \frac{39}{24}$.  Now, $C$ has at most 11 nearby clusters (Lemma \ref{k+8}).  Since $C$ is adjacent to 2 one-third vertices, at least 3 of these clusters are not poor 1-clusters; additionally, at least one of the leaves of $C$ is distance-2 from more than one vertex in $D \setminus C$.  Therefore, $C$ sends at most $\frac{8}{24}$ by Rules 3-7a (Proposition \ref{dist3 prop}) and at most $\frac{1}{24}$ by Rule 7c; therefore, $f(C) \geq \frac{30}{24}$.

Let $P(C) \geq 6$.  If $f_2(C) \geq \frac{43}{24}$, then $f(C) \geq \frac{30}{24}$ (Proposition \ref{c3 prop}).  If $f_2(C) < \frac{43}{24}$, then $C$ is adjacent to at least 3 one-third vertices.  Therefore, at least 5 of the clusters nearby $C$ are not poor 1-clusters.  Therefore, $C$ has at most 6 nearby poor 1-clusters (Lemma \ref{k+8}).  Then, $C$ sends at most $\frac{6}{24}$ by Rules 3-7a (Proposition \ref{dist3 prop}).  By Rule 7c, $C$ sends at most $\frac{2}{24}$.  Therefore, $C$ sends at most $\frac{8}{24}$.  But if $P(C) \geq 6$, then $f_2(C) \geq \frac{40}{24}$; therefore, $f(C) \geq \frac{32}{24}$.
\end{proof}

Now we begin our discussion of \fourplus s.  For $k \geq 4$, let $C$ be a $k$-cluster, and let $v$ be a vertex in $C$.  Then, $d_C(v) \in \{ 1,2,3 \}$.  Let $$\alpha_i = | \{ v \in C : d_C(v) = i \} |$$
Now, $C$ has at most $k+8$ nearby poor 1-clusters (Lemma \ref{k+8}); therefore, $C$ sends at most $\frac{k+8}{24}$ by Rules 3-7a (Proposition \ref{dist3 prop}).  By Rule 7c, $C$ sends at most $\frac{1}{24}$ for each leaf of $C$.  Now, the number of leaves of $C$ is $\alpha_1$; therefore, $C$ sends at most $\frac{1}{24} \alpha_1$ by Rule 7c.  Rules 1 and 2 are the only others by which $C$ may need to send charge; therefore, $f(C) \geq f_2(C) - \frac{1}{24} [(k+8) + \alpha_1]$.  Now, $f_2(C)$ is minimal when $P(C) = 2$; that is, $f_2(C) \geq k - (\frac{5}{12} + \frac{5}{24}) \alpha_1 - \frac{5}{12} \alpha_2$.  Let $F(C) = f(C) - \frac{5}{12} k$.  Then, $$F(C) \geq \left[k -  \left(\frac{5}{12} + \frac{5}{24}  \right) \alpha_1 - \frac{5}{12} \alpha_2 \right] - \frac{1}{24} \left[ \left(k+8 \right) + \alpha_1 \right] -\frac{5}{12} k$$
Now, $k = \alpha_1 + \alpha_2 + \alpha_3$.  Then, substituting and simplifying, $$F(C) \geq \frac{1}{24} \left( -3 \alpha_1 + 3 \alpha_2 + 13 \alpha_3 - 8 \right)$$
Now, $\Delta(C) = 3$; therefore, $\alpha_1 \leq \alpha_3 + 2$.  Then, 
\begin{equation} 
\label{F(C)}
F(C) \geq \frac{1}{24} \left[ -3 \left( \alpha_3 +2 \right) + 3 \alpha_2 + 13 \alpha_3 - 8 \right] = \frac{1}{24} \left( 3 \alpha_2 + 10 \alpha_3 - 14 \right)
\end{equation}
Now, $F(C) < 0$ if, and only if, $f(C) < \frac{5}{12} k$.  Let $A = \{ (0,0), (1,0), (2,0), (3,0), (4,0), (0,1), (1,1) \}$.  If $F(C) < 0$, then $(\alpha_2, \alpha_3) \in A$.  That is, for all $(\alpha_2, \alpha_3) \not \in A$, Equation \ref{F(C)} implies $f(C) \geq \frac{5}{12} k$.  Therefore, we have only left to consider the cases in which $(\alpha_2, \alpha_3) \in A$.

If $(\alpha_2, \alpha_3) \in \{ (0,0), (1,0) \}$, then $C \in \cK_1 \cup \cK_3$.  But we assumed $C \in \cK_{4^+}$; therefore, we need not consider this case.  If $(\alpha_2, \alpha_3) \in \{ (2,0), (0,1) \}$, then $C \in \cK_4$; we consider this case in Claim \ref{4cluster claim}.  If $(\alpha_2, \alpha_3) \in \{ (3,0), (1,1) \}$, then $C \in \cK_5$; we consider this case in Claim \ref{5cluster claim}.  Finally, if $(\alpha_2, \alpha_3) = (4,0)$, then $C \in \{L \in \cK_6 : \Delta(L) = 2 \}$; we consider this case in Claim \ref{6cluster claim}.


\begin{claim}
\label{4cluster claim}

For every $4$-cluster, $C$, $f(C) \geq 4 \cdot \frac{5}{12}$.

\end{claim}

\begin{proof}

First, consider a linear 4-cluster, $C$, and let $P(C) = 2$.  Then, $f_1(C) = \frac{46}{24}$.  Rule 2 does not apply; therefore, $f_2(C) = \frac{46}{24}$.  First suppose $C$ sends no charge by Rule 7c.  Then, $C$ has at most 8 nearby poor 1-clusters; however, if $C$ has $k$ nearby poor 1-clusters, where $k > 6$, then $k - 6$ of the distance-3 poor 1-clusters are stealable (Lemma \ref{linearopen4 lem}) -- that is, $k-6$ of the nearby poor 1-clusters will receive charge from other \threeplus s by Rules 3a-3b and not from $C$ by Rules 3c, 6 or 7a.  Therefore, $C$ sends charge to at most 6 nearby poor 1-clusters.  Then, $C$ sends at most $\frac{6}{24}$ (Proposition \ref{dist3 prop}) and, therefore, $f(C) \geq \frac{40}{24} = 4 \cdot \frac{5}{12}$.  Now suppose $C$ sends $\frac{1}{24}$ by Rule 7c.  Then, one one-turn position is not in $D$; therefore, $C$ has at most 6 nearby poor 1-clusters and if $C$ has exactly 6 such clusters then at least one of the distance-3 poor 1-clusters is stealable (Lemma \ref{linearopen4 lem}) -- that is, at least one of the distance-3 poor 1-clusters will receive charge by Rules 3a-3b and not from $C$ by Rules 3c, 6 or 7a.  Therefore, $C$ sends charge to at most 5 nearby poor 1-clusters by Rules 3, 6 and 7a.  Then, $C$ sends at most $\frac{5}{24}$ by Rules 3, 6 and 7a (Proposition \ref{dist3 prop}) and $\frac{1}{24}$ by Rule 7c; therefore, $f(C) \geq \frac{40}{24}$.  Finally, suppose $C$ sends $\frac{2}{24}$ by Rule 7c.  Then, neither one-turn position is in $D$; therefore, $C$ has at most 4 nearby poor 1-clusters.  Then, $C$ sends at most $\frac{4}{24}$ by Rules 3, 6 and 7a (Proposition \ref{dist3 prop}) and $\frac{2}{24}$ by Rule 7c; therefore, $f(C) \geq \frac{40}{24}$.  

Let $P(C) = 3$.  First, suppose $C$ is adjacent to no one-third vertices.  Then, $f_2(C) = \frac{51}{24}$.  Now, $C$ has at most 9 nearby poor 1-clusters (Lemma \ref{linear4 lem2}); therefore, $C$ sends at most $\frac{9}{24}$ by Rules 3-7a (Proposition \ref{dist3 prop}).  And $C$ sends at most $\frac{2}{24}$ by Rule 7c; therefore, $f(C) \geq \frac{40}{24}$.  Now, suppose $C$ is adjacent to a one-third vertex, $v_{\frac{1}{3}}$.  Then, $f_2(C) \geq \frac{47}{24}$.  Now, $C$ has at most 6 nearby poor 1-clusters (Lemma \ref{linear4 lem2}); therefore, $C$ sends at most $\frac{6}{24}$ by Rules 3-7a (Proposition \ref{dist3 prop}).  Since $P(C) = 3$, one of the leaves of $C$ must be adjacent to $v_{\frac{1}{3}}$; therefore, one of the leaves of $C$ has more than one distance-2 vertex in $D \setminus C$.  Then, $C$ sends at most $\frac{1}{24}$ by Rule 7c; therefore, $f(C) \geq \frac{40}{24}$.  

Now, consider a curved 4-cluster, $C$, and let $P(C) = 2$.  Then, $f_1(C) = \frac{46}{24}$.  Rule 2 does not apply; therefore, $f_2(C) = \frac{46}{24}$.  First, suppose $C$ sends no charge by Rule 7c.  Then, $C$ has at most 8 nearby poor 1-clusters; however, if $C$ has $k$ such clusters, where $k > 6$, then at least $k-6$ of the distance-3 poor 1-clusters are stealable (Lemma \ref{curvedopen4 lem}) -- that is, $C$ sends charge to at most 6 nearby poor 1-clusters by Rules 3-7a.  Then, $C$ sends no charge by Rule 7c and at most $\frac{6}{24}$ by Rules 3-7a (Proposition \ref{dist3 prop}); therefore, $f(C) \geq \frac{40}{24}$.  Now, suppose $C$ sends $\frac{1}{24}$ by Rule 7c.  Then, one backwards position of $C$ is not in $D$; therefore, $C$ has at most 6 nearby poor 1-clusters, and if $C$ has 6 such clusters then at least one is stealable (Lemma \ref{curvedopen4 lem}) -- that is, $C$ sends at most $\frac{5}{24}$ by Rules 3-7a (Proposition \ref{dist3 prop}).  Therefore, $f(C) \geq \frac{40}{24}$.  Finally, suppose $C$ sends $\frac{2}{24}$ by Rule 7c.  Then, neither backwards position of $C$ is in $D$; therefore, $C$ has at most 2 nearby poor 1-clusters (Lemma \ref{curvedopen4 lem}).  Then, $C$ sends at most $\frac{2}{24}$ by Rules 3-7a (Proposition \ref{dist3 prop}); therefore, $f(C) \geq \frac{42}{24}$.  

Let $P(C) = 3$.  First, suppose $C$ is adjacent to no one-third vertices.  Then, $f_2(C) = \frac{51}{24}$.  Now, $C$ has at most 11 nearby poor 1-clusters (Lemma \ref{curved4 lem2}); therefore, if $C$ sends no charge by Rule 7c, then $f(C) \geq \frac{40}{24}$ (Proposition \ref{dist3 prop}).  If $C$ sends $\frac{1}{24}$ by Rule 7c, then one backwards position of $C$ is not in $D$; therefore, $C$ has at most 10 nearby poor 1-clusters (Lemma \ref{curved4 lem2}).  Then, $C$ sends at most $\frac{10}{24}$ by Rules 3-7a (Proposition \ref{dist3 prop}) and $\frac{1}{24}$ by Rule 7c; therefore, $f(C) \geq \frac{40}{24}$.  If $C$ sends $\frac{2}{24}$, then both backwards positions are not in $D$; therefore, $C$ has at most 9 nearby poor 1-clusters (Lemma \ref{curved4 lem2}).  Then, $C$ sends at most $\frac{9}{24}$ by Rules 3-7a (Proposition \ref{dist3 prop}); therefore, $f(C) \geq \frac{40}{24}$.  Now, suppose $C$ is adjacent to a one-third vertex, $v_{\frac{1}{3}}$.  Then, $f_2(C) \geq \frac{47}{24}$.  Since $P(C)=3$ and each leaf of $C$ has at least one distance-2 vertex in $D \setminus C$ (Proposition \ref{leaves prop}), $v_{\frac{1}{3}}$ is adjacent to one of the leaves of $C$; therefore, $C$ sends at most $\frac{1}{24}$ by Rule 7c.  Now, $C$ has at most 6 nearby poor 1-clusters (Lemma \ref{curved4 lem2}).  Therefore, $C$ sends at most $\frac{6}{24}$ by Rules 3-7a (Proposition \ref{dist3 prop}) and at most $\frac{1}{24}$ by Rule 7c; therefore, $f(C) \geq \frac{40}{24}$.  

Consider a linear or curved 4-cluster, $C$, and let $P(C) \geq 4$.  First, suppose $C$ is adjacent to no one-third vertices.  Then, $f_2(C) \geq \frac{56}{24}$.  Now, $C$ has at most 12 nearby poor 1-clusters (Lemma \ref{k+8}); therefore, $C$ sends at most $\frac{12}{24}$ by Rules 3-7a (Proposition \ref{dist3 prop}).  By Rule 7c, $C$ sends at most $\frac{2}{24}$; therefore, $f(C) \geq \frac{42}{24}$.  Now, suppose $C$ is adjacent to exactly one one-third vertex.  Then, $f_2(C) \geq \frac{52}{24}$.  Now, $C$ has at most 12 nearby clusters (Lemma \ref{k+8}).  However, since $C$ is adjacent to a one-third vertex, at least 2 of these clusters are not poor 1-clusters; therefore, $C$ has at most 10 nearby poor 1-clusters.  Then, $C$ sends at most $\frac{10}{24}$ by Rules 3-7a (Proposition \ref{dist3 prop}) and at most $\frac{2}{24}$ by Rule 7c; therefore, $f(C) \geq \frac{40}{24}$.  Finally, suppose $C$ is adjacent to 2 one-third vertices.  Now, if $P(C) = 4$, then each leaf is adjacent to a one-third vertex and $f_2(C) \geq \frac{48}{24}$.  Then, at least 4 of the 12 possible nearby clusters are not poor 1-clusters; therefore, $C$ sends at most $\frac{8}{24}$ by Rules 3-7a (Proposition \ref{dist3 prop}).  Since both leaves have more than one distance-2 vertex in $D \setminus C$, no charge is sent by Rule 7c; therefore, $f(C) \geq \frac{40}{24}$.  If $P(C) \geq 5$ and $C$ is adjacent to 3 one-third vertices, then $f_2(C) \geq \frac{49}{24}$ and at least 5 of the 12 possible nearby clusters are not poor 1-clusters; therefore, $C$ sends at most $\frac{7}{24}$ by Rules 3-7a (Proposition \ref{dist3 prop}).  By Rule 7c, $C$ sends at most $\frac{2}{24}$; therefore, $f(C) \geq \frac{40}{24}$.  If $P(C) \geq 5$ and $C$ is adjacent to exactly 2 one-third vertices, then $f_2(C) \geq \frac{53}{24}$.  Now, at least 3 of the 12 possible nearby clusters are not poor 1-clusters; therefore, $C$ sends at most $\frac{9}{24}$ by Rules 3-7a (Proposition \ref{dist3 prop}).  By Rule 7c, $C$ sends at most $\frac{2}{24}$; therefore, $f(C) \geq \frac{42}{24}$.

Consider a 4-cluster, $C$, and let $C$ have a degree-3 vertex.  First, suppose $P(C) = 3$.  Then, $f_2(C) = \frac{51}{24}$.  Now, $C$ has at most 8 nearby poor 1-clusters; therefore, $C$ sends at most $\frac{8}{24}$ by Rules 3-7a (Proposition \ref{dist3 prop}).  Since $C$ has 3 leaves, $C$ sends at most $\frac{3}{24}$ by Rule 7c.  Therefore, $f(C) \geq \frac{40}{24}$.  Now, suppose $P(C) \geq 4$.  If $C$ is adjacent to a one-third vertex, $v_{\frac{1}{3}}$, then $f_2(C) \geq \frac{52}{24}$ and at least 2 of the 12 possible nearby clusters (Lemma \ref{k+8}) are not poor 1-clusters; therefore, $C$ sends at most $\frac{10}{24}$ by Rules 3-7a (Proposition \ref{dist3 prop}).  From the structure of $C$, we see that $v_{\frac{1}{3}}$ must be adjacent to a leaf of $C$; therefore, at least one of the leaves of $C$ has more than one distance-2 vertex in $D \setminus C$.  Therefore, $C$ sends at most $\frac{2}{24}$ by Rule 7c.  Therefore, $f(C) \geq \frac{40}{24}$.  If $C$ is adjacent to no one-third vertices, then $f_2(C) \geq \frac{56}{24}$.  Now, $C$ has at most 12 nearby poor 1-clusters (Lemma \ref{k+8}), and $C$ sends at most $\frac{3}{24}$ by Rule 7c; therefore, $f(C) \geq \frac{41}{24}$ (Proposition \ref{dist3 prop}).  
\end{proof}

\begin{claim}
\label{5cluster claim}

For every $5$-cluster, $C$, $f(C) \geq 5 \cdot \frac{5}{12}$.

\end{claim}

\begin{proof}

Consider a 5-cluster, $C$ with $\Delta(C) = 2$.  If $C \in \cK_5^c$, then $f_2(C) \geq \frac{65}{24}$.  Now, $C$ has at most 13 nearby poor 1-clusters (Lemma \ref{k+8}); therefore, $C$ sends at most $\frac{13}{24}$ by Rules 3-7a (Proposition \ref{dist3 prop}).  Since $C$ has exactly 2 leaves, $C$ sends at most $\frac{2}{24}$ by Rule 7c.  Therefore, $f(C) \geq \frac{50}{24} = 5 \cdot \frac{5}{12}$.  If $C \in \cK_5^o$, then $f_2(C) \geq \frac{60}{24}$.  Now, $C$ has at most 9 nearby poor 1-clusters; furthermore, if $C$ has exactly 9 such clusters, then at least one is stealable (Lemma \ref{5cluster lem}) -- that is, $C$ sends at most $\frac{8}{24}$ by Rules 3-7a (Proposition \ref{dist3 prop}).  By Rule 7c, $C$ sends at most $\frac{2}{24}$.  Therefore, $f(C) \geq \frac{50}{24}$.

Now, let $C$ have a degree-3 vertex.  Then, $f_2(C) \geq \frac{65}{24}$.  Now, $C$ has at most 12 nearby poor 1-clusters (Lemma \ref{5star lem}); therefore, $C$ sends at most $\frac{12}{24}$ by Rules 3-7a (Proposition \ref{dist3 prop}).  Since $C$ has 3 leaves, $C$ sends at most $\frac{3}{24}$ by Rule 7c.  Therefore, $f(C) \geq \frac{50}{24}$.  
\end{proof}

\begin{claim}
\label{6cluster claim}

For every $6$-cluster, $C$, with $\Delta(C) = 2$, $f(C) \geq 6 \cdot \frac{5}{12}$.

\end{claim}

\begin{proof}

Consider a 6-cluster, $C$ with $\Delta(C) = 2$.  Then, $C$ has exactly 2 leaves.  If $C \in \cK_6^c$, then $f_2(C) \geq \frac{79}{24}$.  Now, $C$ has at most 14 nearby poor 1-clusters; therefore, $C$ sends at most $\frac{14}{24}$ by Rules 3-7a (Proposition \ref{dist3 prop}).  By Rule 7c, $C$ sends at most $\frac{2}{24}$.  Therefore, $f(C) \geq \frac{63}{24} > 6 \cdot \frac{5}{12}$.  If $C \in \cK_6^o$, then $f_2(C) \geq \frac{74}{24}$.  Now, $C$ has at most 10 nearby poor 1-clusters (Lemma \ref{6cluster lem}); therefore, $C$ sends at most $\frac{10}{24}$ by Rules 3-7a (Proposition \ref{dist3 prop}).  By Rule 7c, $C$ sends at most $\frac{2}{24}$.  Therefore, $f(C) \geq \frac{62}{24} > 6 \cdot \frac{5}{12}$.  
\end{proof}

\section{Deferred Proofs}
\label{deferred proofs}

\begin{proof}[Proof of Proposition \ref{nonpoor prop1}]

Let $H$ be the group of non-poor 1-clusters described by $b,d$ and $e$ in Figure \ref{nonpoorgroup2}.  We choose $i \in \poorone$.  Then, $i$ is distance-2 from $e$ and not distance-2 from the other 1-clusters in $H$.  By symmetry, this is the general case.  Since $i \in \poorone$, we have $j,s \not \in D$ and, by Corollary \ref{poor1 cor}, $h \not \in D$.  Therefore, $q \in D_{3^+}$ (Proposition \ref{force3cluster prop}) and $g \in D$ (Proposition \ref{1clusters prop}).  Let $C$ be the \threepluss at $q$.  If $r \in D$, then $r \in C$; therefore, $i$ is distance-2 from a \threeplus.  If $r \not \in D$, then $p \in C$.  If $n \in D$, then $g,n,p,q \in C$ and $C$ is a \fourplus; therefore, $i$ is distance-3 from a \fourplus.  If $n \not \in D$, then $v \in D$ and $p,q,v \in C$.  Now, $g$ closes $C$.  Therefore, $i$ is distance-3 from a closed \threeplus.
\end{proof}

\begin{proof}[Proof of Lemma \ref{nonpoor lem1}]
Let $H$ be the group of non-poor 1-clusters described by $b,d$ and $e$ in Figure \ref{nonpoorgroup2}.  Now, if a poor 1-cluster, $w$, is distance-2 from exactly one of $b,d$ and $e$, then $w$ is distance-2 from an open 3-cluster or within distance-3 of a closed 3-cluster or $4^+$-cluster (Proposition \ref{nonpoor prop1}).  Thus, we need only consider poor 1-clusters which are distance-2 from 2 of the 1-clusters in $H$.  There are 3 possibilities: $a,c$ and $h$.  Suppose by contradiction that each of $a,c$ and $h$ is a poor 1-cluster that is not distance-2 from an open 3-cluster nor within distance-3 of a closed 3-cluster or $4^+$-cluster.  Since $h \in \poorone$, we have $p \in D$ or $r \in D$ but not both (Corollary \ref{poor1 cor}).  By symmetry, we choose $r \in D$.  Now, by hypothesis, $h$ is not distance-2 from any \threeplus; therefore, $r \in D_1$ (Corollary \ref{not in 3+}).  So we have $s \not \in D$.  Since $h \in \poorone$ and $e \in D$, we also have $i \not \in D$ (Corollary \ref{poor1 cor}).  Therefore, $j \in D_{3^+}$ (Proposition \ref{force3cluster prop}).  Also, since $i,s \not \in D$ and $r \in D_1$, we have $t \in D$ (Proposition \ref{1clusters prop}).  Now, $e \in D$ and, by hypothesis, $c \in \poorone$; therefore, $f \not \in D$ (Corollary \ref{poor1 cor}).  Let $C$ be the \threepluss at $j$.  Since $f,i \not \in D$, we have $k \in C$.  Now, if $u \in D$, then $j,k,t,u \in C$; therefore, $c$ and $h$ are distance-3 from a \fourplus, which is a contradiction.  If $u \not \in D$, then $m \in C$ and $t$ closes $C$; therefore, $c$ and $h$ are distance-3 from a closed \threeplus, which is a contradiction.
\end{proof}

\begin{figure}[h]
\begin{center}
\subfloat[Proposition \ref{nonpoor prop1} and Lemma \ref{nonpoor lem1}]{\label{nonpoorgroup2}\includegraphics[width=0.2877\textwidth]{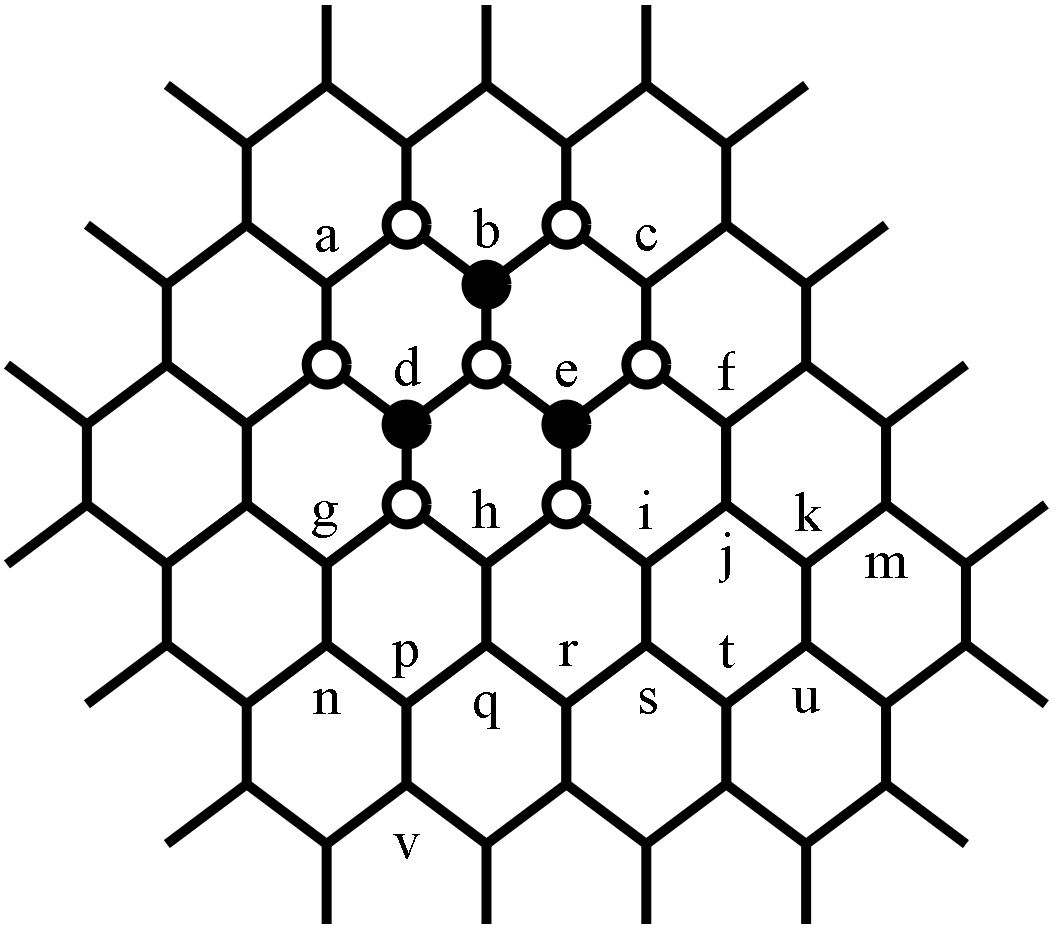}}
\hspace{0.5cm}
\subfloat[Lemma \ref{nonpoor lem2}]{\label{nonpoor2}
\includegraphics[width=0.2355\textwidth]{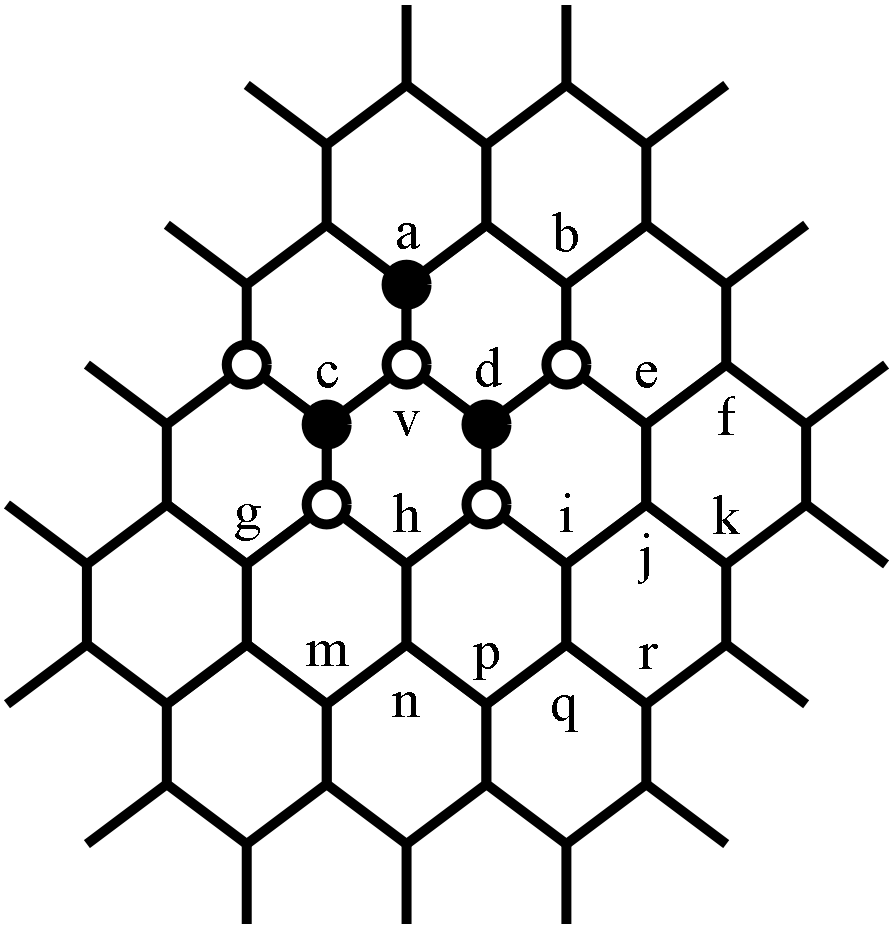}}
\hspace{0.5 cm}
\subfloat[Lemma \ref{nonpoor lem3}]{\label{nonpoor3}\includegraphics[width=0.2355\textwidth]{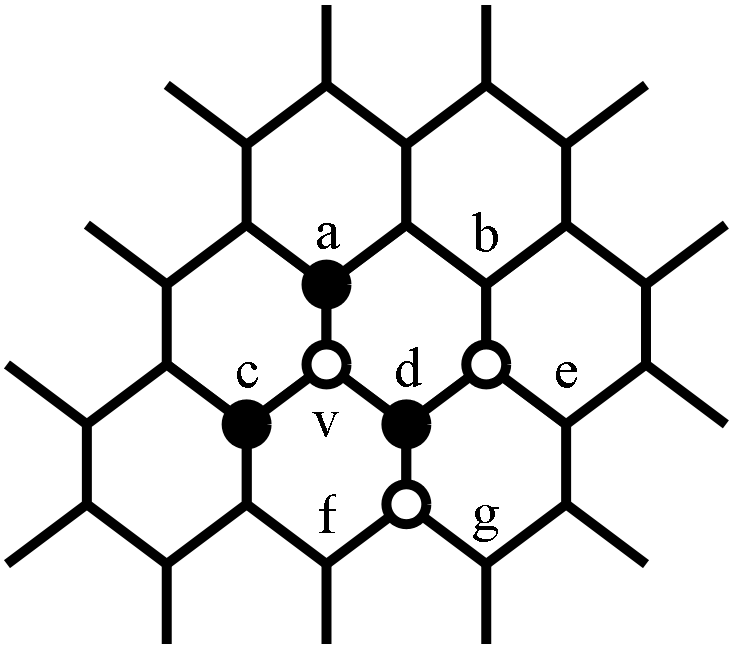}}
\end{center}
\caption{}
\end{figure}

\begin{proof}[Proof of Lemma \ref{nonpoor lem2}]

Let $v$ be the one-third vertex shown in Figure \ref{nonpoor2}, and let $c$ and $d$ be 1-clusters.  Then, by hypothesis, $a \in D \setminus D_1$; therefore, $a \in D_{3^+}$ (Corollary \ref{not in 3+}).  Suppose by contradiction that one of $c$ and $d$ has at least 2 distance-2 poor 1-clusters that are not distance-2 from an open 3-cluster nor within distance-3 of a closed 3-cluster or $4^+$-cluster.  By symmetry, we consider $d$.  Now, there are 4 candidates for distance-2 poor 1-clusters: $b,e,h$ and $i$.  However, $b$ is distance-2 from the \threepluss at $a$, so we need not consider $b$.

If $h \in \poorone$, then $n \not \in D$ and $i \not \in D$ (Corollary \ref{poor1 cor}).  So we have $e \in \poorone$; therefore, $f,j \not \in D$ and $k \in D$ (Proposition \ref{1clusters prop}).  Since $i,j \not \in D$, we have $q \in D_{3^+}$ (Proposition \ref{force3cluster prop}).  Let $C$ be the \threepluss at $q$.  If $p \in D$, then $p \in C$ and $h$ is distance-2 from a \threeplus, which is a contradiction.  If $p \not \in D$, then $r \in C$; therefore, either $k \in C$ and $e$ is distance-2 from a \fourplus, or $k$ closes $C$ and both $e$ and $h$ are distance-3 from a closed \threeplus, which is a contradiction.

If $h \not \in \poorone$, then we have $e,i \in \poorone$.  Therefore, $f,j,q \not \in D$ and, by Corollary \ref{poor1 cor}, $h \not \in D$.  Then $n \in D_{3^+}$ (Proposition \ref{force3cluster prop}) and $g \in D$ (Proposition \ref{1clusters prop}).  Let $C$ be the \threepluss at $n$.  If $p \in D$, then $p \in C$ and $i$ is distance-2 from a \threeplus, which is contradiction.  If $p 
\not \in D$, then $m \in C$.  Then, either $g \in C$ and $i$ is distance-3 from a \fourplus, or $g$ closes $C$ and $i$ is distance-3 from a closed \threeplus, which is a contradiction.
\end{proof}

\begin{proof}[Proof of Lemma \ref{nonpoor lem3}]

Let $v$ be the one-third vertex shown in Figure \ref{nonpoor3}, and let $d$ be a 1-cluster.  Then, by hypothesis, $a,c  \in D \setminus D_1$; therefore, $a,c \in D_{3^+}$ (Corollary \ref{not in 3+}).  Now, $d$ has 6 distance-2 vertices: $a,b,c,e,f$ and $g$.  However, $a,c \in D_{3^+}$.  If $b \in \poorone$, then $e \not \in D$ (Corollary \ref{poor1 cor}); and vice versa.  If $f \in \poorone$, then $g \not \in D$ (Corollary \ref{poor1 cor}); and vice versa.  Therefore, at most 2 of $b$, $e$, $f$ and $g$ are poor 1-clusters.
\end{proof}

\begin{proof}[Proof of Lemma \ref{finless lem}]

Let $C$ be the 3-cluster shown in Figure \ref{3-cluster}.  Suppose by contradiction that $C$ has 2 finless sides; then, $n,p \not \in D$ (Definition \ref{3cluster def}).  If $j \not \in D$, then $p \in D_{3^+}$ (Proposition \ref{force3cluster prop}).  But none of the vertices adjacent to $p$ is in $D$; therefore, $p \in D_1$, which is a contradiction.  If $j \in D$ and $p \in D$, then $j,p \in D_2$, which is a contradiction (Proposition \ref{2clusters prop}).  If $j \in D$ and $p \not \in D$, then $j \in D_{3^+}$ (Proposition \ref{force3cluster prop}).  But none of the vertices adjacent to $j$ is in $D$; therefore, $j \in D_1$, which is a contradiction.
\end{proof}

\begin{proof}[Proof of Lemma \ref{closed3 lem}]
\label{closed3 pf}

Let $C_1$ be the 3-cluster shown in Figure \ref{closed 3-cluster}.  If $C_1 \in \cK_3^c$, then the non-leaf vertex of $C_1$ has at least one distance-2 vertex in $D \setminus C$ (Definition \ref{open-closed}); by symmetry, we choose $f \in D$.  Now, $P(C_1) = 3$ and each leaf of $C$ has at least one distance-2 vertex in $D \setminus C$ (Proposition \ref{leaves prop}); therefore, $e \not \in D$ and $$|\{d,j,p,q\} \cap D| = |\{g,k,r,q\} \cap D|=1$$
There are 11 candidates for nearby poor 1-clusters: $a$, $c/d$, $f$, $h$, $i/j$, $k/m$, $n$, $p/t$, $q/v$, $r/x$ and $s$.

First we consider the cases for which $q \not \in D$.  To begin, we show that there are at most 9 nearby poor 1-clusters.  Now, $v \in D_{3^+}$ (Proposition \ref{force3cluster prop}); therefore, there are at most 10 nearby poor 1-clusters.  If $p \in D$, then $n \not \in \poorone$ and there are at most 9 nearby poor 1-clusters.  If $p \not \in D$ and $t \not \in D$, then there are at most 9 nearby poor 1-clusters.  If $r \in D$, then $s \not \in \poorone$ and there are at most 9 nearby poor 1-clusters.  If $r \not \in D$ and $x \not \in D$, then there are at most 9 nearby poor 1-clusters.  So we consider $p,r \not \in D$ and $t,x \in D$.  Let $C_v$ be the \threepluss at $v$.  At least one of $u$ and $w$ is in $C_v$; therefore, at least one of $t$ and $x$ is not a 1-cluster.  Therefore, there are at most 9 nearby poor 1-clusters.

Now we consider the cases for which at least one of $d$ and $g$ is in $D$.  If $g \in D$, then $f,h \not \in \poorone$.  Therefore, there are at most 7 nearby poor 1-clusters.  So now we consider $d \in D$ and $g \not \in D$.  Either $k \in D$ or $r \in D$ but not both.  If $k \in D$, then $s \not \in \poorone$ and there are at most 8 nearby poor 1-clusters.  Now, if $k \not \in \poorone$, then there are at most 7 nearby poor 1-clusters.  If $k \in \poorone$, then $s \not \in D$.  Since $r \not \in D$, we have $x \in D_{3^+}$ (Proposition \ref{force3cluster prop}).  If $t \not \in \poorone$, then $C_1$ has at most 7 nearby poor 1-clusters.  If $t \in \poorone$, then $t$ is distance-2 from $C_v$.  If $C_v \in \openthree$, then $v,w,x \in C_v$ and $t$ is in an arm position.  If $C_v$ is paired, then it is type-2 paired with $C_1$.  Therefore, the lemma holds with $k \in D$.  If $r \in D$, then $s \not \in \poorone$ and there are at most 8 nearby poor 1-clusters.  If $r \not \in \poorone$, then $C_1$ has at most 7 nearby poor 1-clusters.  If $r \in \poorone$, then $s \not \in D$.  Since $k \not \in D$, we have $m \in D_{3^+}$ (Proposition \ref{force3cluster prop}); therefore, $C_1$ has at most 7 nearby poor 1-clusters and the lemma holds.

Now we consider the cases for which neither shoulder position is in $D$.  Since $d \not \in D$, we have $a,c \in D_{3^+}$.  Therefore, $C_1$ has at most 7 nearby poor 1-clusters.  Either $j \in D$ or $p \in D$; in both cases, $n \not \in \poorone$.  Therefore, $C_1$ has at most 6 nearby poor 1-clusters.  Either $k \in D$ or $r \in D$; in both cases, $s \not \in \poorone$.  Therefore, $C_1$ has at most 5 nearby poor 1-clusters and the lemma holds.
\\
\\
Now we consider the case for which $q \in D$.  If $q \in D$, then $d,g,j,k,p,r \not \in D$.  Then, $a,c \in D_{3^+}$ (Proposition \ref{force3cluster prop}).  Therefore, $C_1$ has at most 9 nearby poor 1-clusters.  If $q \not \in \poorone$, then $C_1$ has at most 8 nearby poor 1-clusters.  If $q \in \poorone$, then either $u \in D$ or $w \in D$ (Proposition \ref{1clusters prop}).  If $u \in D$, then $t \not \in \poorone$ and there are at most 8 nearby poor 1-clusters; if $w \in D$, then $x \not \in \poorone$ and there are at most 8 nearby poor 1-clusters.  Therefore, $C_1$ has at most 8 nearby poor 1-clusters.  Now, if $i \not \in \poorone$, then $C_1$ has at most 7 nearby poor 1-clusters.  If $i \in \poorone$, then $i$ is distance-2 from the \threepluss at $c$, $C_c$; since $a \in D$, we have $C_c \not \in \openthree$.
\end{proof}

\begin{figure}[h]
\setcaptionwidth{0.4\textwidth}
\begin{center}
\includegraphics[width=0.3\textwidth]{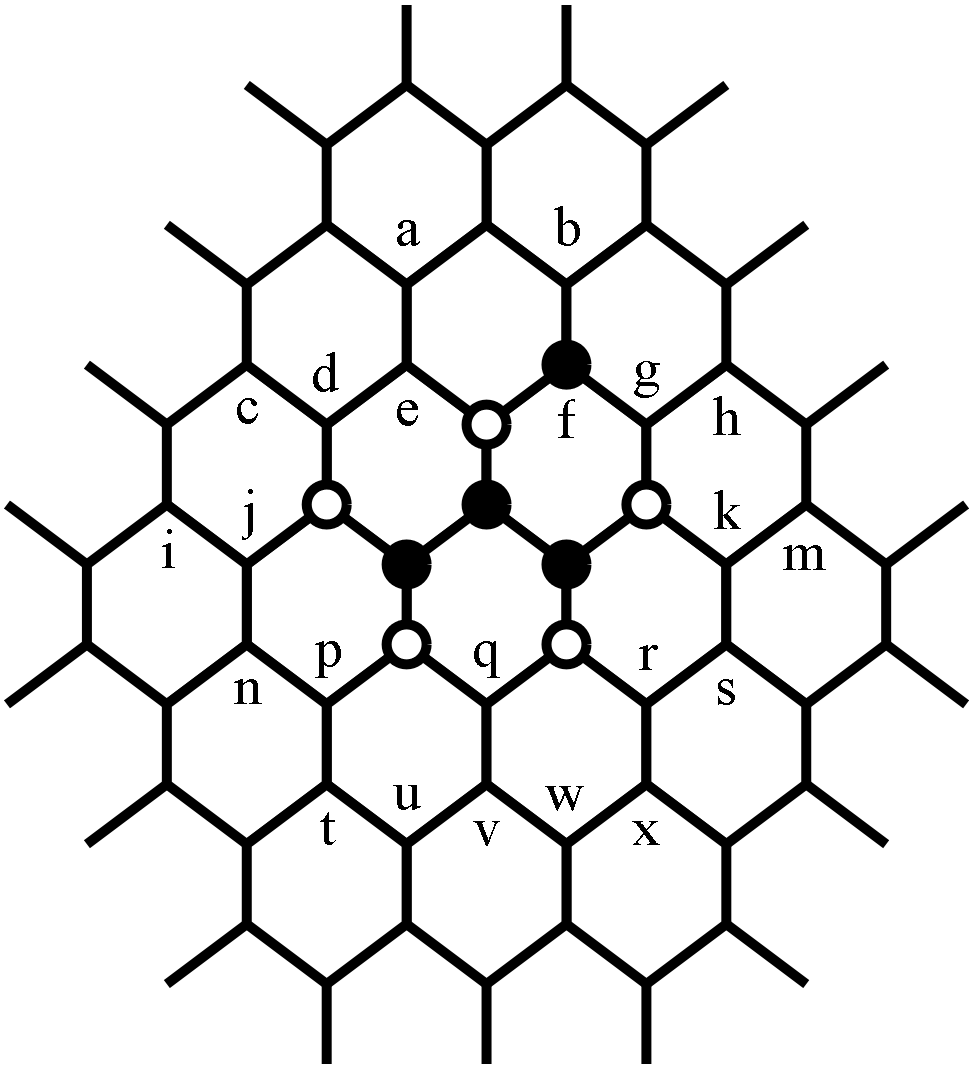}
\end{center}
\caption{Lemma \ref{closed3 lem} and Lemma \ref{closed3 lem2}}
\label{closed 3-cluster}
\end{figure}

\begin{proof}[Proof of Lemma \ref{closed3 lem2}]
\label{closed3 pf2}

Let $C$ be the closed 3-cluster shown in Figure \ref{closed 3-cluster}.  By symmetry, we choose $f \in D$.  There are 5 possible one-third vertices adjacent to $C$, and there are 11 candidates for nearby poor 1-clusters: $a/e$, $c/d$, $f$, $h$, $i/j$, $k/m$, $n$, $p/t$, $q/v$, $r/x$ and $s$.

Suppose $e \in D$.  Then, $a,e,f \not \in \poorone$; therefore, $C$ has at most 9 nearby poor 1-clusters.  Since $P(C) = 4$, each leaf has exactly one distance-2 vertex in $D \setminus  C$.  If $d,g \in D$, then $c,d,h \not \in \poorone$; therefore, $C$ has at most 7 nearby poor 1-clusters.  If $q \in D$, then either $q \in \poorone$ or $q \not \in \poorone$.  If $q \not \in \poorone$, then $C$ has at most 8 nearby poor 1-clusters.  If $q \in \poorone$, then $p,r \not \in D$ (Corollary \ref{poor1 cor}).  If $t \not \in \poorone$ or $x \not \in \poorone$, then $C$ has at most 8 nearby poor 1-clusters.  So assume $t,x \in \poorone$.  Since $q \in \poorone$, we have $u \in D$ or $w \in D$ (Proposition \ref{1clusters prop}); therefore, at least one of $t$ and $x$ is not a poor 1-cluster.  Therefore, $C$ has at most 8 nearby poor 1-clusters.  If $q \not \in D$, then $v \in D_{3^+}$ (Proposition \ref{force3cluster prop}); therefore, $C$ has at most 8 nearby poor 1-clusters.  If $j \in D$ or $p \in D$, then $n \not \in \poorone$; and if $k \in D$ or $r \in D$, then $s \not \in \poorone$.  Therefore, if one foot or arm position is a poor 1-cluster, then $C$ has at most 7 nearby poor 1-clusters; and if 2 foot or arm positions are poor 1-clusters, then $C$ has at most 6 nearby poor 1-clusters. 

Suppose $d,j \in D$.  Then $c,d,i,j,n \not \in \poorone$.  Therefore, $C$ has at most 8 nearby poor 1-clusters.  If $k \in D$ or $r \in D$, then $s \not \in \poorone$; in this case, $C$ has at most 7 nearby poor 1-clusters.

The argument is nearly identical to the one above for the cases in which $g,k \in D$, $p,q \in D$ and $q,r \in D$.
\end{proof}

\begin{proof}[Proof of Lemma \ref{linearopen4 lem}]
\label{linearopen4 pf}

Let $C_1$ be the linear open 4-cluster shown in Figure \ref{linear4 part 2}.  Both leaves of $C_1$ must have at least one distance-2 vertex in $D$ (Proposition \ref{leaves prop}), and, by hypothesis, $P(C_1) = 2$; therefore, $f,g,r,s \not \in D$ and $$ | \{ e,k,q \} \cap D | = | \{ h,m,t\} \cap D|$$


By Proposition \ref{force3cluster prop}, we have $b,w \in D_{3^+}$.  Let $C_b$ and $C_w$ be the \threeplus s at $b$ and $w$, respectively.  There are 10 candidates for nearby poor 1-clusters: $a$, $c/h$, $d/e$, $i$, $j/k$, $m/n$, $p$, $q/v$, $t/u$ and $x$.  Now, at least one of $a$ and $c$ is adjacent to or in $C_b$; therefore, at least one of $a$ and $c$ is not a poor 1-cluster.  A similar argument may be made for $v$ and $x$.  Therefore, there are at most 8 nearby poor 1-clusters.

If both one-turn positions are in $D$, then we have $e,t \in D$.  If $C_1$ has 8 nearby poor 1-clusters, then at most one of $a$ and $c$ is not a poor 1-cluster.  If $a \not \in \poorone$, then $c$ is distance-2 from $C_b$.  If $C_b \in \cK_3$, then $c$ is either in a foot position or arm position.  If $c$ is in an arm position, then $a \in C_b$ and $C_b$ is not paired.  If $c \not \in \poorone$, then a similar argument may be made for $a$.  And a symmetric argument may be made for $v$ and $x$ and $C_w$.  Therefore, at least 2 of the poor 1-clusters at distance-3 are stealable.  If $C_1$ has exactly 7 nearby poor 1-clusters, then at least one of $a$, $c$, $v$ and $x$ is a poor 1-cluster.  Then, the above argument suffices; therefore, at least one of the distance-3 poor 1-clusters is stealable.

If exactly one one-turn position is in $D$, then $e \not \in D$ or $t \not \in D$.  By symmetry we choose $e \not \in D$.  Then, $d \in D_{3^+}$ (Proposition \ref{force3cluster prop}).  Therefore, there are at most 7 nearby poor 1-clusters.  By hypothesis, at least one of $k$ and $q$ is in $D$.  In both cases, $p \not \in \poorone$.  Therefore, there are at most 6 nearby poor 1-clusters.  Now, $a \in D_{3^+}$ (Proposition \ref{force3cluster prop}).  Thus, if $C_1$ has exactly 6 nearby poor 1-clusters, then $c \in \poorone$.  Then $c$ is distance-2 from $C_b$ and $a \in C_b$.  If $C_b \in \cK_3$, then $c$ is in an arm position and $C_b$ is not paired.  Therefore, $c$ is stealable.

If neither one-turn position is in $D$, then $e,t \not \in D$.  Therefore, $a,d,u,x \in D_{3^+}$ (Proposition \ref{force3cluster prop}).  Therefore, $C_1$ has at most 6 nearby poor 1-clusters.  By hypothesis, at least one of $k$ and $q$ and at least one of $h$ and $m$ is in $D$.  If $k \in D$ or $q \in D$, then $p \not \in \poorone$; and if $h \in D$ or $m \in D$, then $i \not \in \poorone$.  Therefore, $C_1$ has at most 4 nearby poor 1-clusters.
\end{proof}

\begin{proof}[Proof of Lemma \ref{linear4 lem2}]
\label{linear4 pf2}

Let $C$ be the linear 4-cluster shown in Figure \ref{linear4 part 2}.  First, suppose $C$ is adjacent to no one-third vertices.  Now, either $g \in D$ or $g \not \in D$.  If $g \in D$, then one of the leaves of $C$ and one of the middle vertices has a distance-2 vertex in $D \setminus C$.  Now, each leaf has at least one distance-2 vertex in $D \setminus C$ (Propostion \ref{leaves prop}) and, by hypothesis, $P(C) = 3$; therefore, $f,h,m,r,s,t \not \in D$ and $| \{e,k,q\} \cap D | = 1$.  Then, $u,w,x \in D_{3^+}$ (Proposition \ref{force3cluster prop}) and there are at most 9 nearby poor 1-clusters: $a$, $c$, $d/e$, $g$, $i$, $j/k$, $n$, $p$ and $q/v$.  Now, suppose $g \not \in D$.  If $r \in D$, then this case can be reduced, by symmetry, to the above case.  So assume $r \not \in D$.  Then, $b, w \in D_{3^+}$ (Propostion \ref{force3cluster prop}).  There are 10 candidates for nearby poor 1-clusters: $a/f$, $c/h$, $d/e$, $i$, $j/k$, $m/n$, $p$, $q/v$, $s/x$ and $t/u$.  Suppose by contradiction that there exist 10 nearby poor 1-clusters.  Then, $i,p \in \poorone$.  Therefore, $h,m \not \in D$ and $n,v \in D$ (Proposition \ref{1clusters prop}).  And we must have $n,v  \in \poorone$; otherwise, $C$ has fewer than 10 nearby poor 1-clusters.  Since $g,h,m \not \in D$, we have $t \in D$ (Proposition \ref{leaves prop}).  And, as above, we must have $t \in \poorone$; therefore, $x \in D$ (Proposition \ref{1clusters prop}).  And, again, we must have $x \in \poorone$.  Therefore, $v,x \in \poorone$.  But $w \in D_{3^+}$ and $r \not \in D$; therefore, at least one of $v$ and $x$ is not a poor 1-cluster, which is a contradiction.

Now, suppose $C$ is adjacent to a one-third vertex, $v_{\frac{1}{3}}$.  Since $P(C) = 3$ and each leaf must have at least one distance-2 vertex in $D \setminus C$ (Proposition \ref{leaves prop}), $v_{\frac{1}{3}}$ must be adjacent to $e$ and $k$ or $m$ and $t$.  By symmetry, we choose $e,k \in D$.  Then, $| \{ h,m,t\} \cap D | = 1$ and $f,g,q,r,s \not \in D$; therefore, $b,w \in D_{3^+}$ (Proposition \ref{force3cluster prop}) and there are at most 7 nearby poor 1-clusters: $a$, $c/h$, $m/n$, $p$, $t/u$, $v$ and $x$.  Suppose by contradiction that $C$ has 7 nearby poor 1-clusters.  Then, $v,x \in \poorone$.  But $w \in D_{3^+}$ and $r \not \in D$; therefore, at least one of $v$ and $x$ is not a poor 1-cluster, which is a contradiction.
\end{proof}

\begin{figure}[h]
\setcaptionwidth{0.3\textwidth}
  \begin{center}
\subfloat[Lemma \ref{linearopen4 lem} and Lemma \ref{linear4 lem2}]{\label{linear4 part 2}\includegraphics[width=0.3\textwidth]{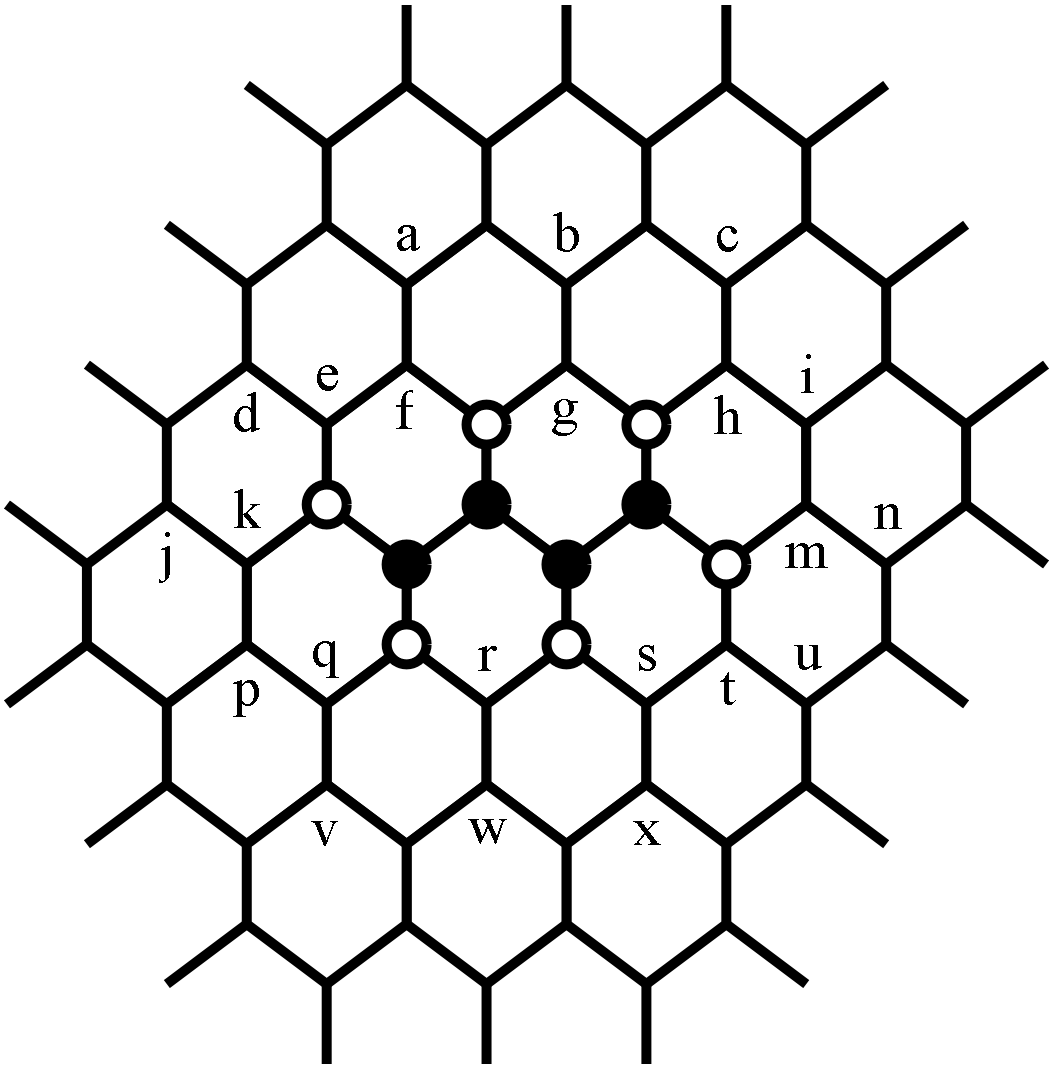}}
\hspace{2.0 cm}
\subfloat[Lemma \ref{curvedopen4 lem} and Lemma \ref{curved4 lem2}]{\label{open4-secondpart2}\includegraphics[width=0.25\textwidth]{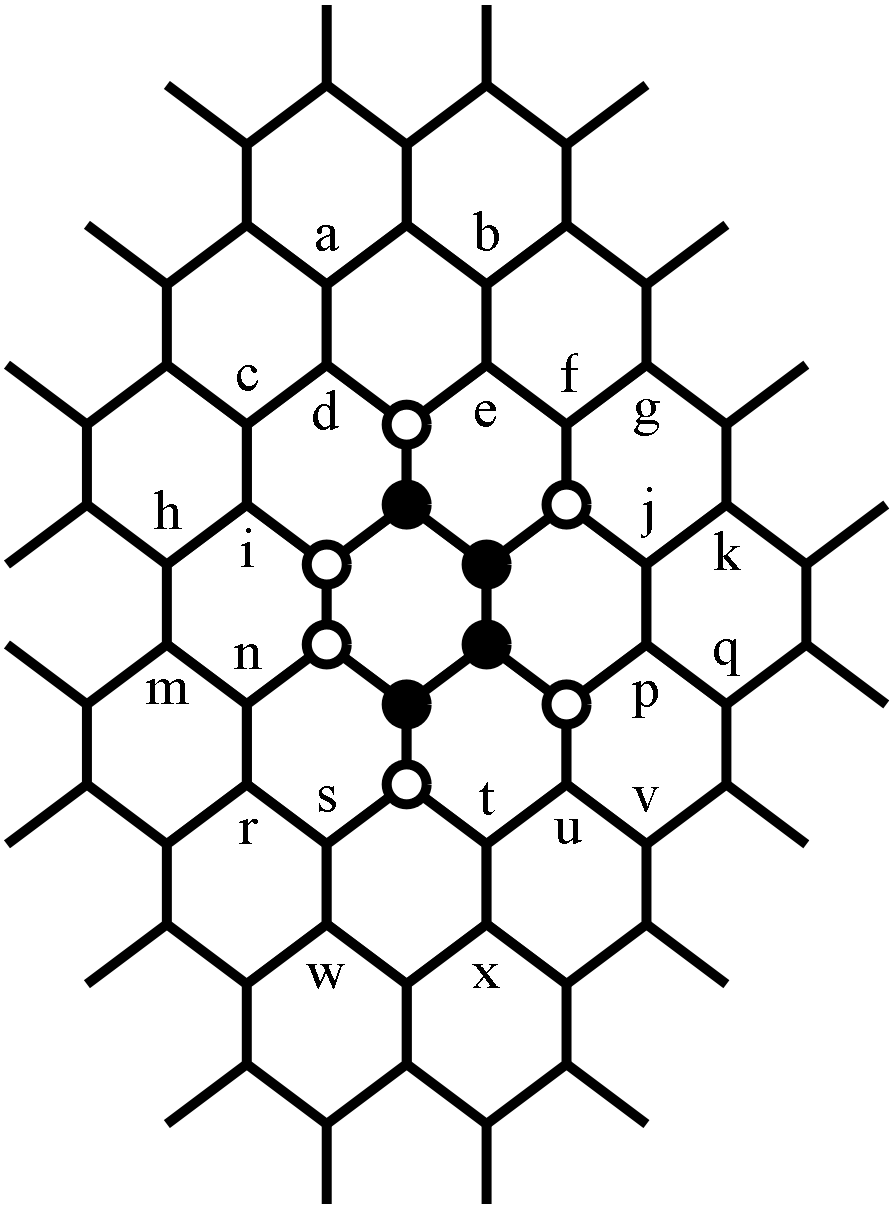}}
  \end{center}
\caption{}
\end{figure}

\begin{proof}[Proof of Lemma \ref{curvedopen4 lem}]
\label{curvedopen4 pf}

Let $C_1$ be the curved open 4-cluster shown in Figure \ref{open4-secondpart2}.  Both leaves of $C_1$ must have at least one distance-2 vertex in $D$ (Proposition \ref{leaves prop}), and, by hypothesis, $P(C_1) = 2$; therefore, $f,j,p,u \not \in D$ and $$ | \{ d,e,i \} \cap D | = | \{ n,s,t\} \cap D | = 1$$
By symmetry, there are only 6 cases to consider: $e,t \in D$; $e,s \in D$; $e,n \in D$; $d,s \in D$; $d,n \in D$; and $i,n \in D$.  Note that $k,q \in D_{3^+}$ in every case (Proposition \ref{force3cluster prop}).  First, we consider the cases with backwards positions.

$e,t \in D$: There are 9 candidates for nearby poor 1-clusters: $a,c,e,g,h,r,s,v$ and $w$.  We could have chosen $m$ instead of $h$, but the proof would be symmetric so we consider only $h$ as a candidate.  Now, at most one of $h$ and $r$ is a poor 1-cluster; therefore, $C_1$ has at most 8 nearby poor 1-clusters.  When $C_1$ has exactly 8 such 1-clusters, all of the candidates other than $h$ and $r$ are poor 1-clusters.  Therefore, $g,v \in \poorone$ and $q$ and $k$ are in the same \fourplus, $C_2$.  Then we have $g$ and $v$ at distance-2 from $C_2$, where $C_2$ is not an open 3-cluster.  When $C_1$ has exactly 7 nearby poor 1-clusters, at most one of $g$ and $v$ is no longer a poor 1-cluster.  Therefore, at least one of $g$ and $v$ is distance-2 from a \threeplus.  If $g \in \poorone$, then $k \in D_{3^+} \setminus D_3^o$ and $g$ is distance-2 from $k$; a symmetric argument can be made for $v$ and $q$.  Therefore, at least one of $g$ and $v$ is distance-2 from a \threeplus, $C_2$, where $C_2$ is not an open 3-cluster.

$e,s \in D$: Since $t \not \in D$, we have $k,q,v,x \in D_{3^+}$.  There are 6 candidates for nearby poor 1-clusters: $a,c,e,g,s$ and $h/m$.  However, $h$ and $s$ cannot both be poor 1-clusters; and $m$ and $c$ cannot both be poor 1-clusters.  Therefore, there are at most 5 nearby poor 1-clusters.

$e,n \in D$: Again, $k,q,v,x \in D_{3^+}$.  There are 7 candidates for nearby poor 1-clusters: $a,c,e,g,h,n$ and $w$.  However, at most one of $n$ and $w$ is a poor 1-cluster.  Therefore, there are at most 6 nearby poor 1-clusters.  If $C_1$ has exactly 6 such 1-clusters, then $g \in \poorone$.  Then, $g$ is distance-2 from the \threepluss at $k$, $C_k$, and $C_k$ is not an open 3-cluster.

$d,s \in D$: Since $e,t \not \in D$, we have $b,g,k,q,v,x \in D_{3^+}$.  There are 3 candidates for nearby poor 1-clusters: $d,h$ and $s$.  We could have chosen $m$ instead of $h$ but the proof would be symmetric.  It cannot be the case that both $h$ and $s$ are poor 1-clusters.  Therefore, there are at most 2 nearby poor 1-clusters.  

$d,n \in D$: Again, $b,g,k,q,v,x \in D_{3^+}$.  There are 4 candidates for nearby poor 1-clusters: $d,h,n$ and $w$.  However, at most one of $d$ and $h$ is a poor 1-cluster; likewise for $n$ and $w$.  Therefore, there are at most 2 nearby poor 1-clusters.

$i,n \in D$: Once again, $b,g,k,q,v,x \in D_{3^+}$.  There are 4 candidates for nearby poor 1-clusters: $a,i,n$ and $w$.  However, at most one of $a$ and $i$ is a poor 1-cluster; likewise for $n$ and $w$.  Therefore, there are at most 2 nearby poor 1-clusters.
\end{proof}

\begin{proof}[Proof of Lemma \ref{curved4 lem2}]
\label{curved4 pf2}

Let $C$ be the curved 4-cluster shown in Figure \ref{open4-secondpart2}.  First, suppose $C$ is adjacent to no one-third vertices and both backwards positions are in $D$; that is, $e,t \in D$.  Then, either $j \in D$ or $j \not \in D$.  Now, $P(C) = 3$; therefore, if $j \in D$ then $d,f,i,n,p,s,u \not \in D$.  Therefore, $C$ has at most 11 nearby poor 1-clusters: $a$, $b$, $c$, $e$, $g$, $h/m$, $j$, $q$, $r$, $t$ and $v$.  Now, consider the case in which $j \not \in D$.  If $p \in D$, then this case can be reduced by symmetry to the previous case.  So we assume $p \not \in D$.  Then, $k,q \in D_{3^+}$ (Proposition \ref{force3cluster prop}), and $C$ has at most 10 nearby poor 1-clusters: $a/d$, $c$, $e$, $g$, $h/i$, $m/n$, $r$, $s/w$, $t$ and $u/v$.

Now, suppose $C$ is adjacent to no one-third vertices and one backwards position is not in $D$.  By symmetry, we choose $e \not \in D$.  Again, either $j \in D$ or $j \not \in D$.  First, assume $j \in D$.  Since $P(C) = 3$ and each leaf has at least one distance-2 vertex in $D \setminus C$ (Proposition \ref{leaves prop}), we must have $f \not \in D$; then, $b, g \in D_{3^+}$ (Proposition \ref{force3cluster prop}), and $C$ has at most 10 nearby poor 1-clusters: $a/d$, $c$, $h/i$, $j$, $m/n$, $p/q$, $r$, $s/w$, $t/x$ and $u/v$.  Now, assume $j \not \in D$.  If $p \in D$, then this case can be reduced by symmetry to the previous case.  So we assume $p \not \in D$; then, $k,q \in D_{3^+}$ (Proposition \ref{force3cluster prop}), and $C$ has at most 10 nearby poor 1-clusters: $a/d$, $b$, $c$, $f/g$, $h/i$, $m/n$, $r$, $s/w$, $t/x$ and $u/v$.

Now, suppose $C$ is adjacent to no one-third vertices and both backwards positions are not in $D$; that is, $e,t \not \in D$.  First, assume $j \in D$.  Since $P(C) = 3$ and each leaf has at least one distance-2 vertex in $D \setminus C$ (Proposition \ref{leaves prop}), we must have $f,p,u \not \in D$; then, $b,g,v,x \in D_{3^+}$ (Proposition \ref{force3cluster prop}), and $C$ has at most 8 nearby poor 1-clusters: $a/d$, $c$, $h/i$, $j$, $m/n$, $q$, $r$ and $s/w$.  Now, assume $j \not \in D$.  If $p \in D$, then this case can be reduced by symmetry to the previous case.  So we assume $p \not \in D$.  Then, $k,q \in D_{3^+}$ (Proposition \ref{force3cluster prop}).  There are 10 candidates for nearby poor 1-clusters: $a/d$, $b$, $c$, $f/g$, $h/i$, $m/n$, $r$, $s/w$, $u/v$ and $x$.  Each leaf of $C$ has at least one distance-2 vertex in $D \setminus C$ (Proposition \ref{leaves prop}).  Therefore, $d \in D$ or $i \in D$; in both cases, $c \not \in \poorone$.  Therefore, $C$ has at most 9 nearby poor 1-clusters.

Finally, suppose $C$ is adjacent to a one-third vertex, $v_{\frac{1}{3}}$.  Since $P(C) = 3$ and each leaf must have at least one distance-2 vertex in $D \setminus C$ (Proposition \ref{leaves prop}), $v_{\frac{1}{3}}$ must be adjacent to a leaf of $C$.  By symmetry, we choose $d,e \in D$.  Then, for the same reasons, we must have $f,i,j,p,u \not \in D$.  Then, $k,q \in D_{3^+}$ (Proposition \ref{force3cluster prop}), and $C$ has at most 6 nearby poor 1-clusters: $g$, $h/m$, $n/r$, $s/w$, $t/x$ and $v$.
\end{proof}

\begin{figure}[h]
\begin{center}
\subfloat[Case 1]{\label{linear5}\includegraphics[width=0.326\textwidth]{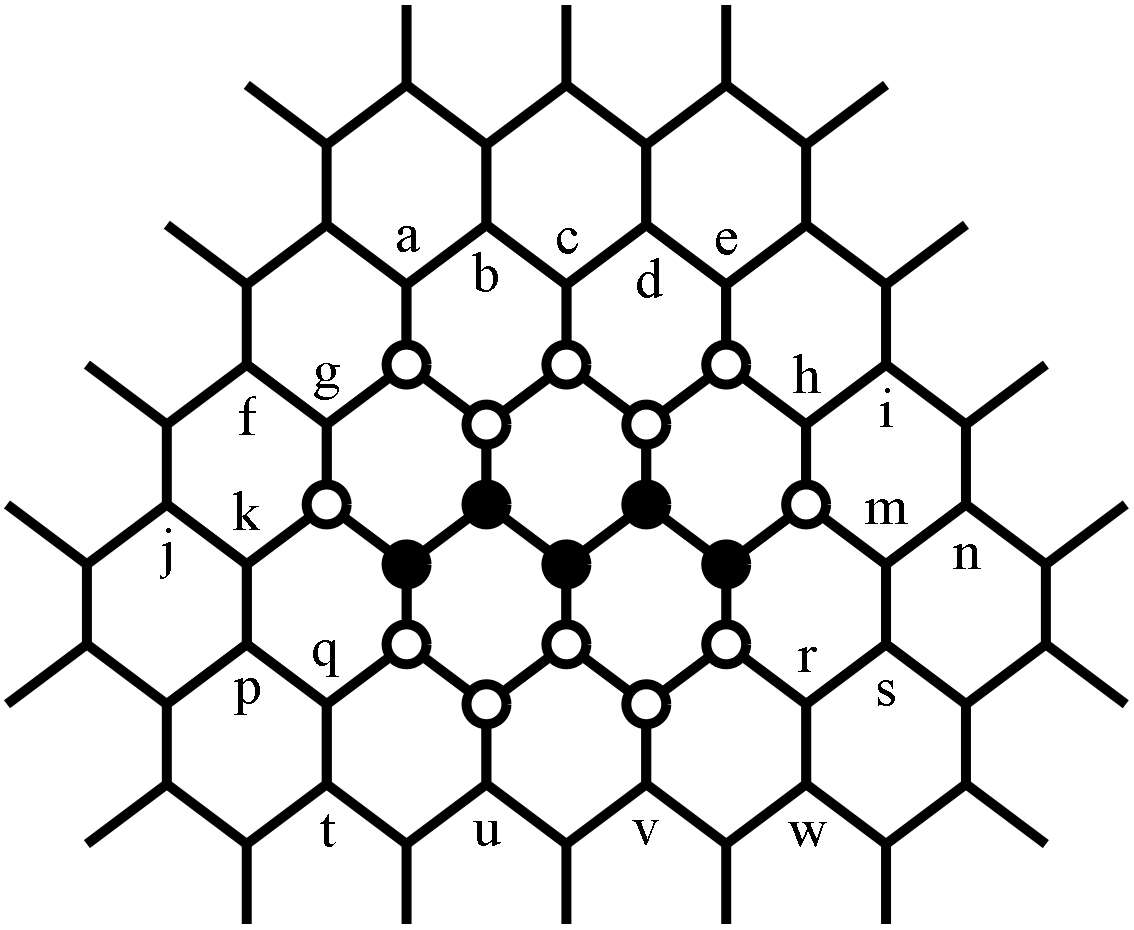}}
\hspace{0.2 cm}
\subfloat[Case 2]{\label{semicurved5}\includegraphics[width=0.302714\textwidth]{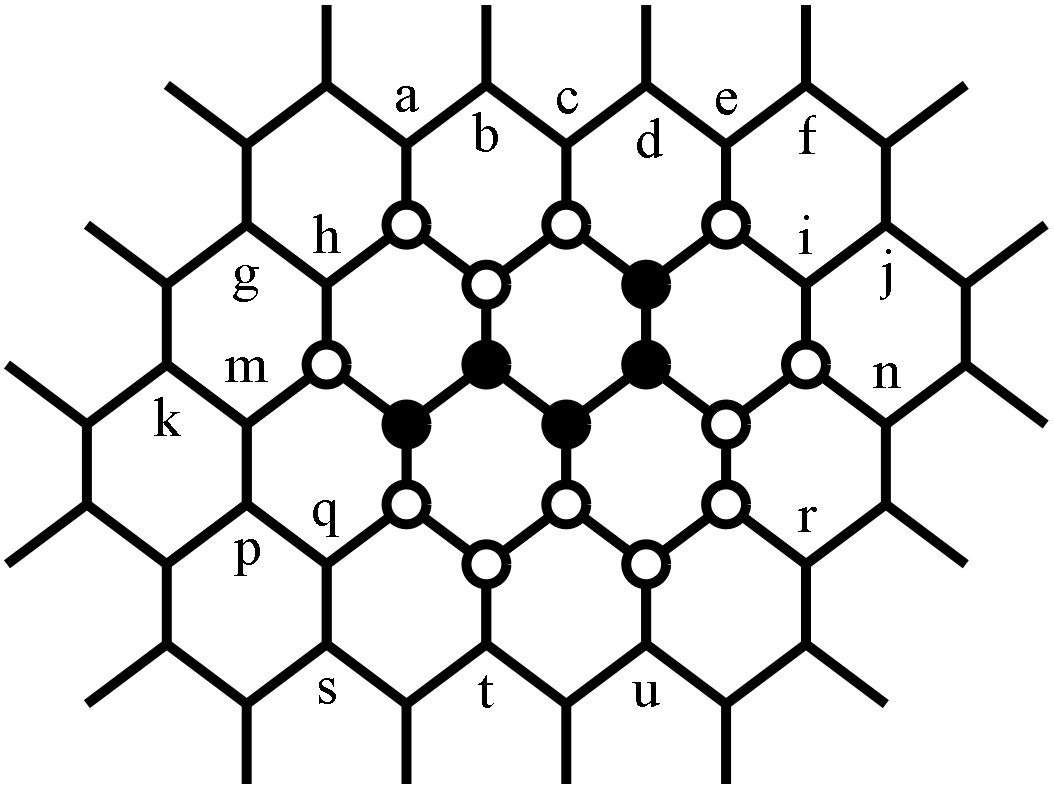}}
\hspace{0.2 cm}
\subfloat[Case 3]{\label{curved5}\includegraphics[width=0.27943\textwidth]{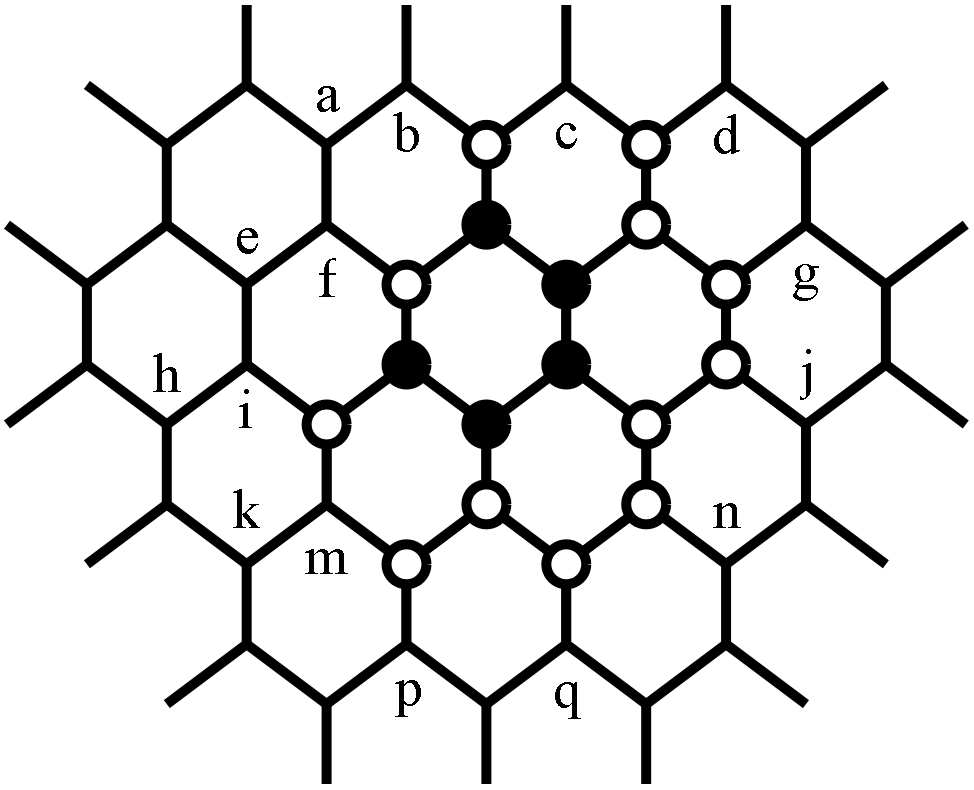}}
\end{center}
\caption{Lemma \ref{5cluster lem}}
\end{figure}

\begin{proof}[Proof of Lemma \ref{5cluster lem}]

By symmetry, there are only 3 cases to consider.

Let $C_1$ be the open 5-cluster shown in Figure \ref{linear5}.  By Proposition \ref{force3cluster prop}, we have $c,u,v \in D_{3^+}$.  Then the candidates for nearby poor 1-clusters are $a$, $e$, $f/g$, $h/i$, $j/k$, $m/n$, $p/q$, $r/s$, $t$ and $w$.  There are 10 candidates in total.  However, $c \in D_{3^+}$.  Let $C_c$ be the \threepluss at $c$.  Now, $b \in C_c$ or $d \in C_c$ or both.  If both, then $a,e \not \in \poorone$ and there are at most 8 nearby poor 1-clusters.  So consider the case in which only one of $b$ and $d$ is in $D$.  By symmetry, we choose $d \in D$; therefore, $e \not \in \poorone$ and there are at most 9 nearby poor 1-clusters.  However, $a$ is distance-2 from $C_c$.  Now, if $e \not \in D$, then $h,i \in D_{3^+}$ (Proposition \ref{force3cluster prop}) and there are at most 8 nearby poor 1-clusters.  Therefore, if $C_c \in \openthree$ and $C_1$ has 9 nearby poor 1-clusters, then $c,d,e \in C_c$.  Then, $a$ is not in a shoulder position and $C_c$ is not paired.

Let $C_2$ be the open 5-cluster shown in Figure \ref{semicurved5}.  By Proposition \ref{force3cluster prop}, we have $r,t,u \in D_{3^+}$.  There are 9 candidates for nearby poor 1-clusters: $a/b$, $c/d$, $e/f$, $g/h$, $i/j$, $k/m$, $n$, $p/q$ and $s$.  However, $s$ is distance-2 from the \threepluss at $t$; let $C_t$ be this \threeplus.  If $C_t \in \openthree$, then $u \in C_t$; furthermore, $s$ is in an arm position but $C_t$ is not paired.

Let $C_3$ be the open 5-cluster shown in Figure \ref{curved5}.  By Proposition \ref{force3cluster prop}, we have $g,j,n,q \in D_{3^+}$.  There are 7 candidates for nearby poor 1-clusters: $a/b$, $c$, $d$, $e/f$, $h/i$, $k/m$ and $p$.
\end{proof}

\begin{proof}[Proof of Lemma \ref{4star lem}]
Let $C$ be the 4-cluster shown in Figure \ref{4star}.  Then $C$ has one degree-3 vertex.  Each of the 3 leaves of $C$ has at least one distance-2 vertex in $D \setminus C$ (Proposition \ref{leaves prop}) and, by hypothesis, $P(C) = 3$; therefore $$ | \{e,f,j,p\} \cap D| = |\{f,g,k,q\} \cap D| = |\{p,q,t,v\} \cap D| = 1$$
By symmetry, we must consider only 2 cases: $f \in D$ and $g \in D$.
There are 12 candidates for nearby poor 1-clusters: $a/e$, $b/f$, $c/g$, $d$, $h$, $i/j$, $k/m$, $n/p$, $q/r$, $s/t$, $u$ and $v/w$.

Now, if $f \in D$, then $e,g,j,k,p,q \not \in D$.  Therefore, $n,r \in D_{3^+}$ (Proposition \ref{force3cluster prop}).  Thus, $C$ has at most 10 nearby poor 1-clusters.  Then, either $t \in D$ or $v \in D$; in both cases, $u \not \in \poorone$.  Therefore, $C$ has at most 9 nearby poor 1-clusters.  By symmetry, we choose $v \in D$ and $t \not \in D$.  If $v \not \in \poorone$, then the lemma holds; so we assume $v \in \poorone$.  Then, $u \not \in D$.  But we also have $t \not \in D$; therefore, $s \in D_{3^+}$ (Proposition \ref{force3cluster prop}).  Therefore, $C$ has at most 8 nearby poor 1-clusters.

If $g \in D$, then $f,k,q \not \in D$.  Therefore, $b,r \in D_{3^+}$ and $h \not \in \poorone$.  Therefore, $C$ has at most 9 nearby poor 1-clusters.  If $g \not \in \poorone$, then the lemma holds; so we assume $g \in \poorone$.  Then, $h \not \in D$.  But we also have $k \not \in D$; therefore, $m \in D_{3^+}$ (Proposition \ref{force3cluster prop}).  Therefore, $C$ has at most 8 nearby poor 1-clusters.
\end{proof}

\begin{figure}[h]
\begin{center}
\subfloat[Lemma \ref{4star lem}]{\label{4star}\includegraphics[width=0.27\textwidth]{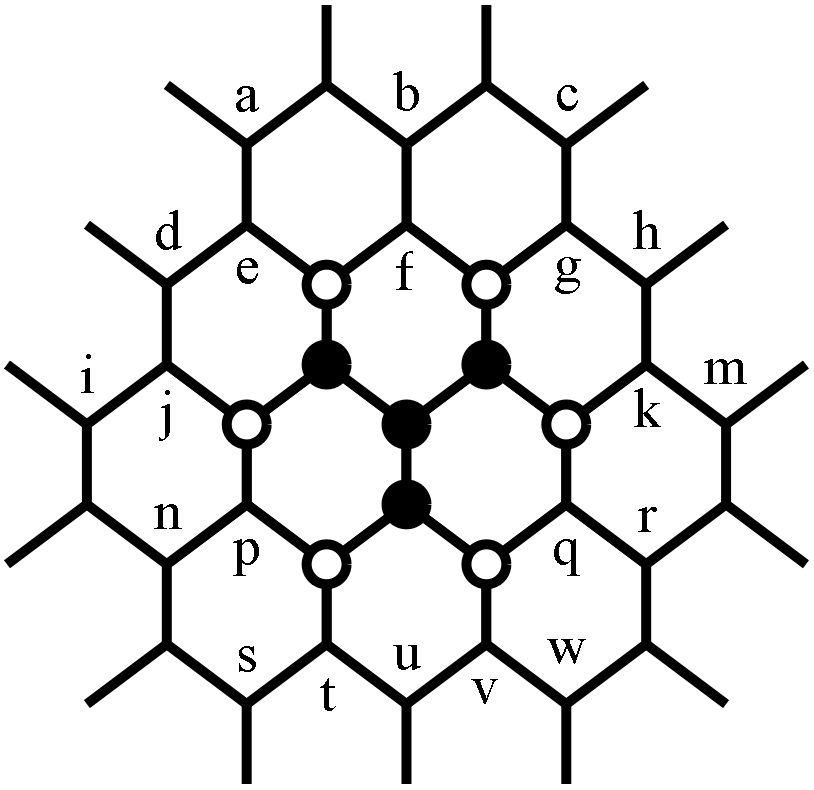}}
\hspace{0.5 cm}
\subfloat[Lemma \ref{5star lem}]{\label{5star}\includegraphics[width=0.275\textwidth]{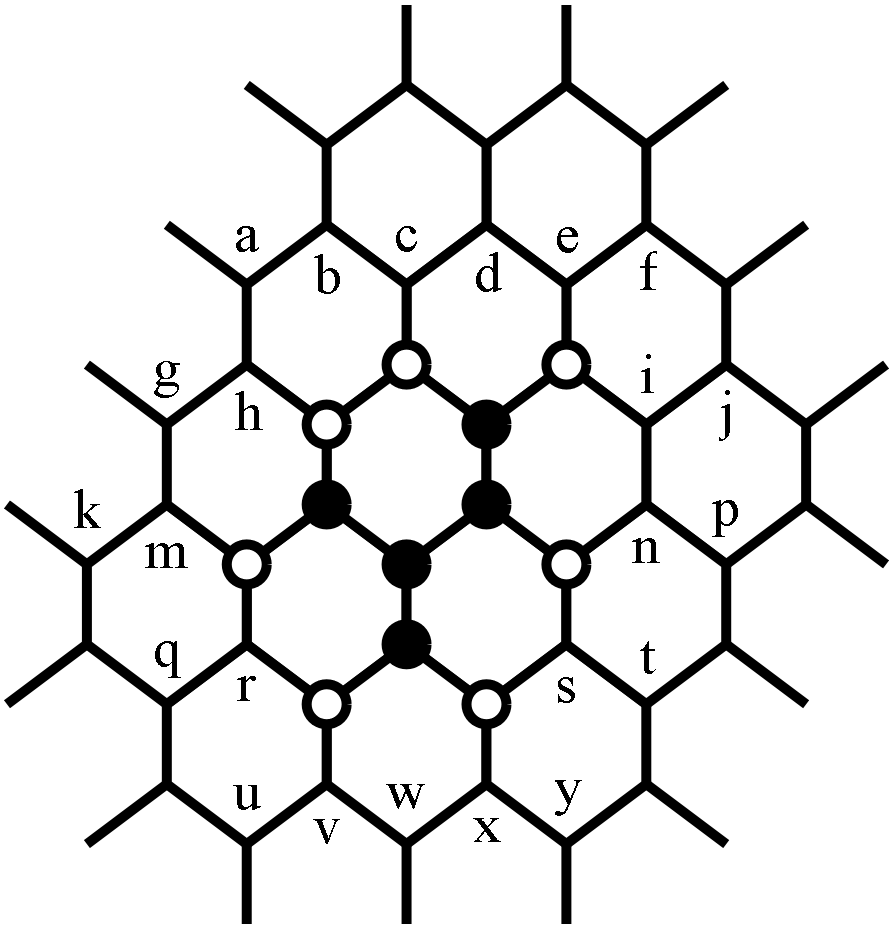}}
\end{center}
\caption{}
\end{figure}

\begin{proof}[Proof of Lemma \ref{5star lem}]

Let $C$ be the 5-cluster shown in Figure \ref{5star}.  Then $C$ has one degree-3 vertex.  There are 13 candidates for nearby poor 1-clusters: $a/h$, $b/c$, $d$, $e/f$, $g$, $i/j$, $k/m$, $n/p$, $q/r$, $s/t$, $u/v$, $w$ and $x/y$.

Now, if $c \in D$, then $d \not \in \poorone$.  Therefore, $C$ has at most 12 nearby poor 1-clusters.

Now we consider the case for which $c \not \in D$.  If $b \not \in \poorone$, then $C$ has at most 12 nearby poor 1-clusters and the lemma holds.  So we assume $b \in \poorone$.  Then we have $a \not \in \poorone$.  If $h \not \in \poorone$, then $C$ has at most 12 nearby poor 1-clusters and the lemma holds.  So we assume $h \in \poorone$.  Then $g \not \in \poorone$; therefore, $C$ has at most 12 nearby poor 1-clusters.
\end{proof}

\begin{figure}[h]
\begin{center}
\subfloat[Case 1]{\label{6cluster1}\includegraphics[width=0.27\textwidth]{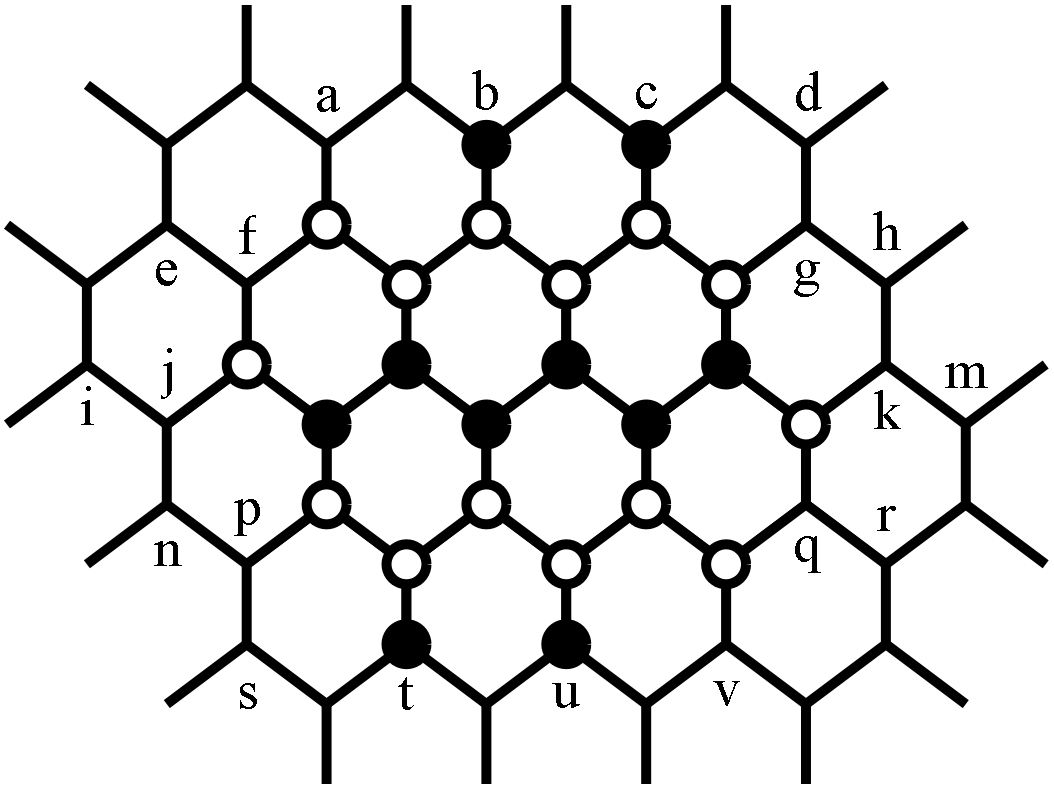}}
\subfloat[Case 2]{\label{6cluster2}\includegraphics[width=0.26\textwidth]{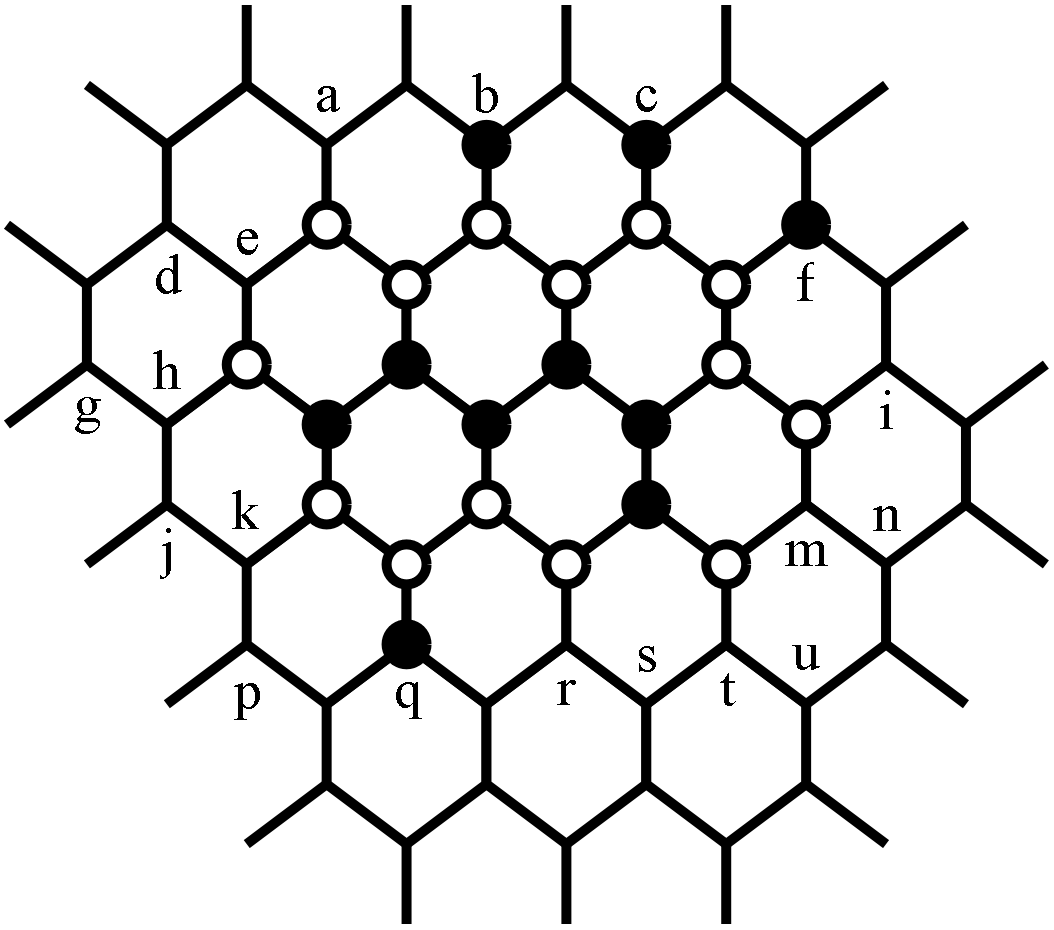}}
\subfloat[Case 3]{\label{6cluster3}\includegraphics[width=0.26\textwidth]{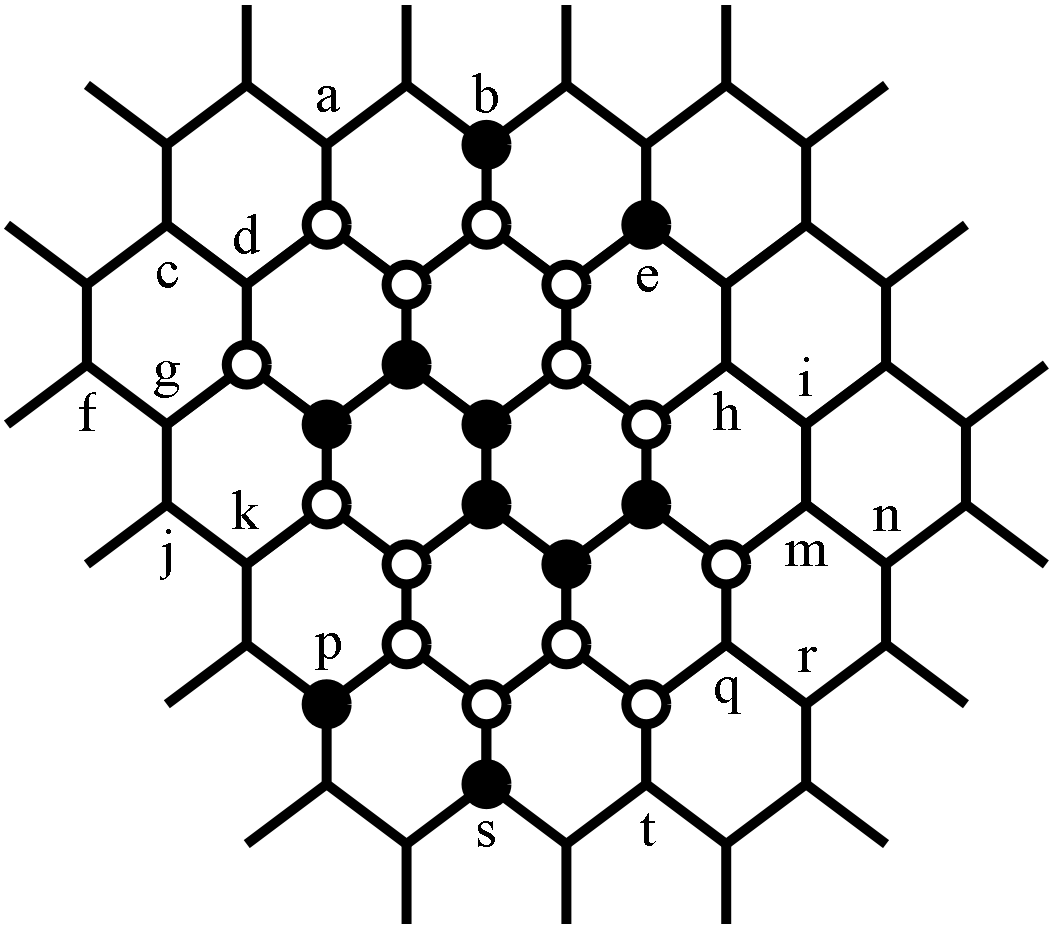}}
\subfloat[Case 4]{\label{6cluster4}\includegraphics[width=0.24\textwidth]{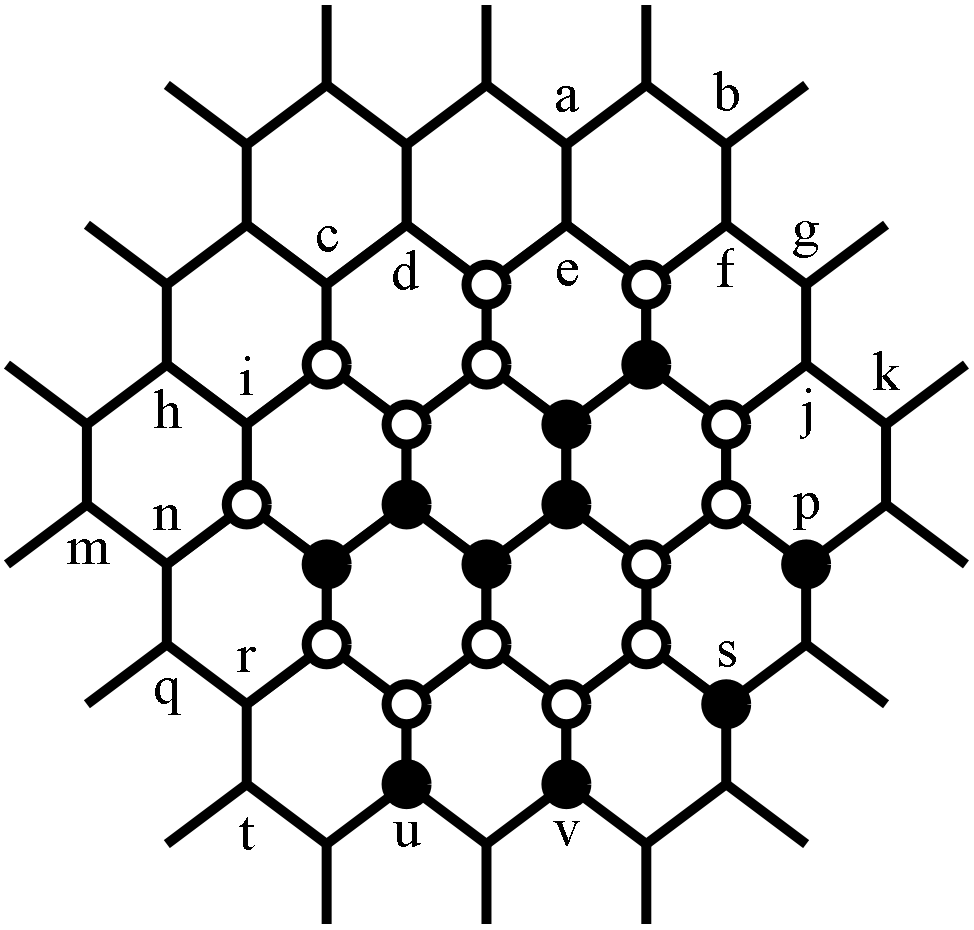}}
\end{center}
\caption{Lemma \ref{6cluster lem}}
\end{figure}

\begin{proof}[Proof of Lemma \ref{6cluster lem}]
By symmetry there are only 4 cases to consider.
Let $C_1$ be the open 6-cluster shown in Figure \ref{6cluster1}.  Then, $b,c,t,u \in D_{3^+}$ (Proposition \ref{force3cluster prop}).  There are 10 candidates for nearby poor 1-clusters: $a$, $d/g$, $e/f$, $h$, $i/j$, $k/m$, $n$, $p/s$, $q/r$ and $v$.  Therefore, $C_1$ has at most 10 nearby poor 1-clusters.
Let $C_2$ be the open 6-cluster shown in Figure \ref{6cluster2}.  Then, $b,c,f,q \in D_{3^+}$.  Therefore, $C_2$ has at most 9 nearby poor 1-clusters: $a$, $d/e$, $g/h$, $i$, $j/k$, $m/n$, $p$, $r/s$ and $t/u$.
Let $C_3$ be the open 6-cluster shown in Figure \ref{6cluster3}.  Then, $b,e,p,s \in D_{3^+}$.  Therefore, $C_3$ has at most 8 nearby poor 1-clusters: $a$, $c/d$, $f/g$, $h/i$, $j/k$, $m/n$, $q/r$ and $t$.
Let $C_4$ be the open 6-cluster shown in Figure \ref{6cluster4}.  Then, $p,s,u,v \in D_{3^+}$.  Therefore, $C_4$ has at most 9 nearby poor 1-clusters: $a/e$, $b/f$, $c/d$, $g$, $h/i$, $j/k$, $m/n$, $q/r$ and $t$.
\end{proof}

\begin{proof}[Proof of Lemma \ref{vpsym lem}]

Let $v,a,b$ and $c$ be as shown in Figure \ref{verypoorsym1 part1}.  Now, $v \in D_1^{vp}$ and $a,b$ and $c$ are distance-2 from $v$; therefore, $a,b,c \in \poorone$.  Therefore, $g,n,p \not \in D$ and $d,m,q \in D$ (Proposition \ref{1clusters prop}).  Then $h,k,s \in D_{3^+}$ (Proposition \ref{force3cluster prop}).  Let $C_h$ be the \threepluss at $h$.  Now, $C_h$ is distance-3 from $v$; therefore, $C_h \in \openthree$.  Since $b,d \in D$, we must have $e,h,i \in C_h$.  Symmetric arguments may be made to show $f,j,k \in D_{3^+}^o$ and $r,s,t \in D_{3^+}^o$.  Therefore, $v$ is in a head position of 3 open 3-clusters and exactly one of $a,b$ and $c$ is in a shoulder position of each of these open 3-clusters.  

\begin{figure}[h]
\begin{center}
\subfloat[]{\label{verypoorsym1 part1}\includegraphics[width=0.3\textwidth]{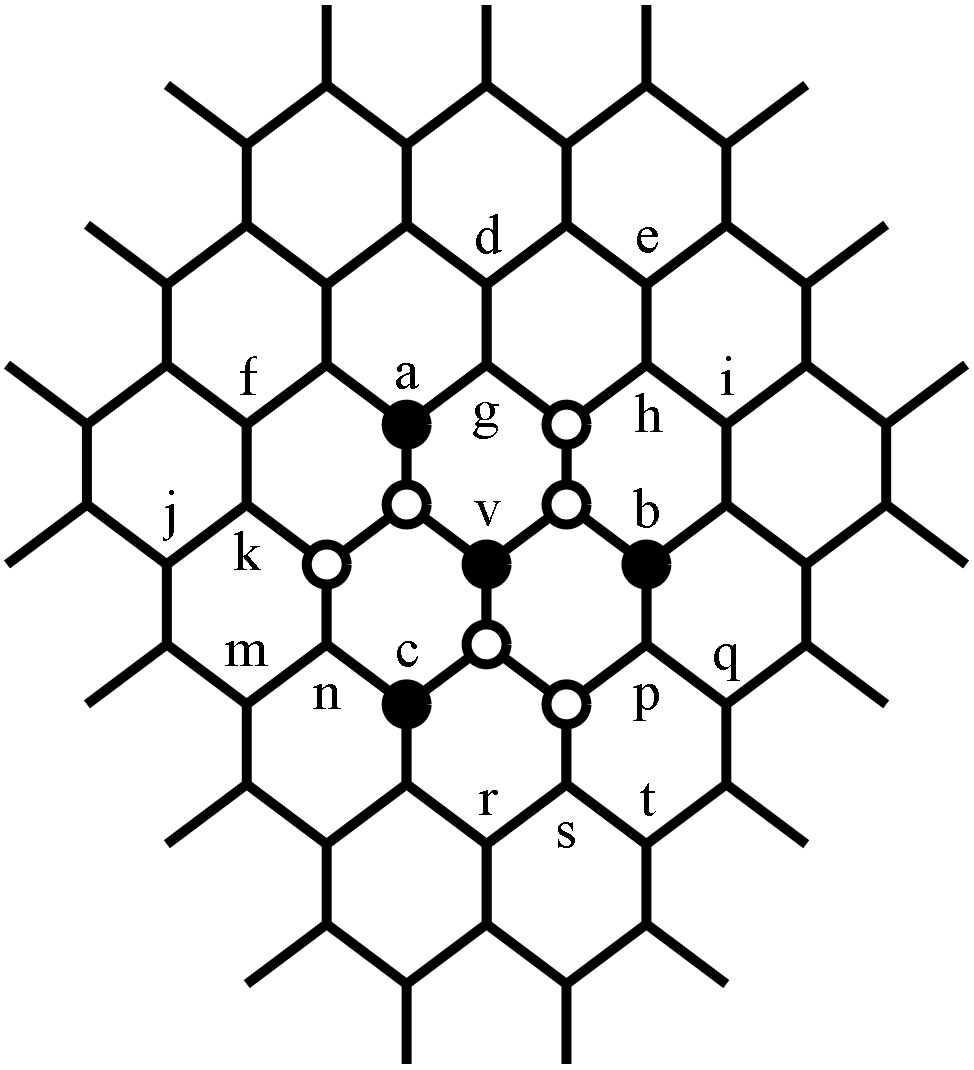}}
\hspace{0.5 cm}
\subfloat[]{\label{verypoorsym1 part2}\includegraphics[width=0.37\textwidth]{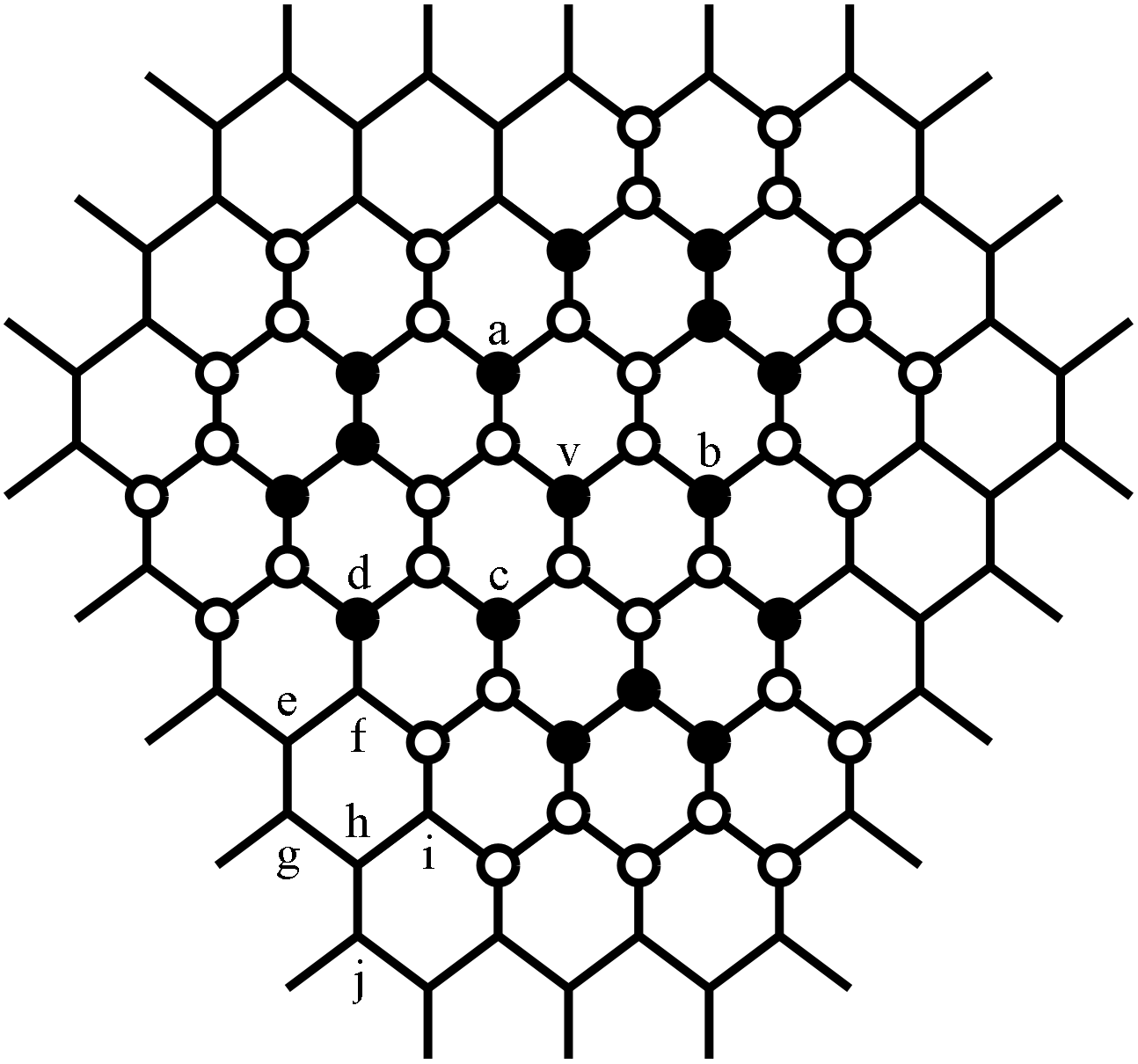}}
\end{center}
\caption{Lemma \ref{vpsym lem}}
\end{figure}

Now suppose each of the open 3-clusters at distance-3 from $v$ is uncrowded as in Figure \ref{verypoorsym1 part2}.  All of the vertices have been relabelled except $v,a,b$ and $c$.  Now, the graph is rotationally symmetric about $v$, so we need only consider one of $a,b$ and $c$.  We choose $c$.  Now, $d$ is in a shoulder position of an open 3-cluster, $C$, which is distance-3 from $v$.  By hypothesis, $C$ is uncrowded; thus, $d \in \poorone$.  Therefore, $f \not \in D$ and $e \in D$ (Proposition \ref{1clusters prop}).  Then we have $i \in D_{3^+}$.  Let $C_i$ be the \threepluss at $i$.  Since 2 of the 3 neighbors of $i$ are not in $D$, we also have $h \in C_i$.  If $g \in D$, then $e,g,h,i \in C_i$.  If $g \not \in D$, then $h,i,j \in C_i$ and $e$ closes $C_i$.  In both cases, $C_i$ is a closed 3-cluster or $4^+$-cluster at distance-3 from $c$.
\end{proof}

\begin{proof}[Proof of Lemma \ref{vp lem}]
\label{vpasym pf}

Let $u,v,w$ and $x$ be as shown in Figure \ref{verypoor1}.  Now, by hypothesis, $v \in D_1^{vp}$; therefore, $u,w,x \in \poorone$.  Then, we have $f \not \in D$; therefore, $n \in D_{3^+}$.  Since $x \in \poorone$, we have $g \in D$ (Proposition \ref{1clusters prop}).  Let $C_0$ be the \threepluss at $n$.  Now, $v$ is distance-3 from $C$, so $C \in \openthree$.  The only possibility is to have $m,n,p \in C$.  Therefore, $v$ and $x$ are in the head positions of an open 3-cluster, and $w$ is in a shoulder position.

\begin{figure}[ht]
\begin{center}
\subfloat[Lemma \ref{vp lem}]{\label{verypoor1}\includegraphics[width=0.3\textwidth]{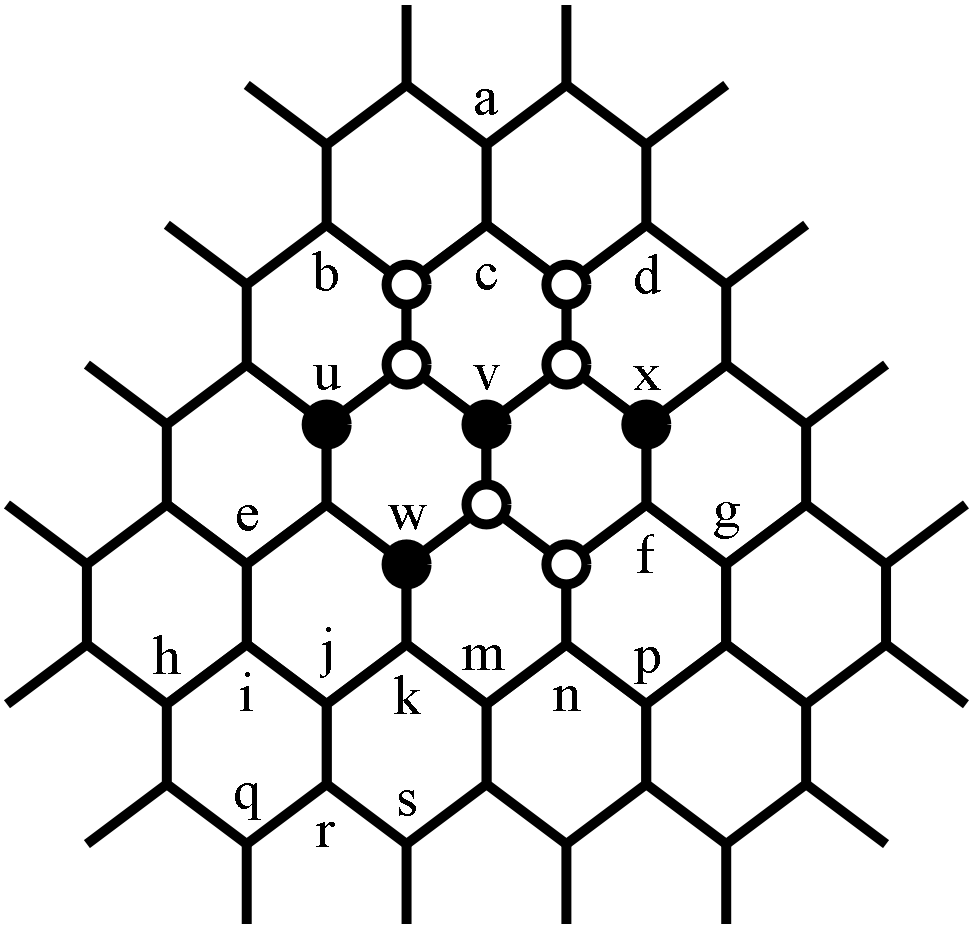}}
\hspace{2cm}
\subfloat[Lemma \ref{vp lem}, where (i) and (ii) are not satisfied.]{\label{verypoor2}\includegraphics[width=0.35\textwidth]{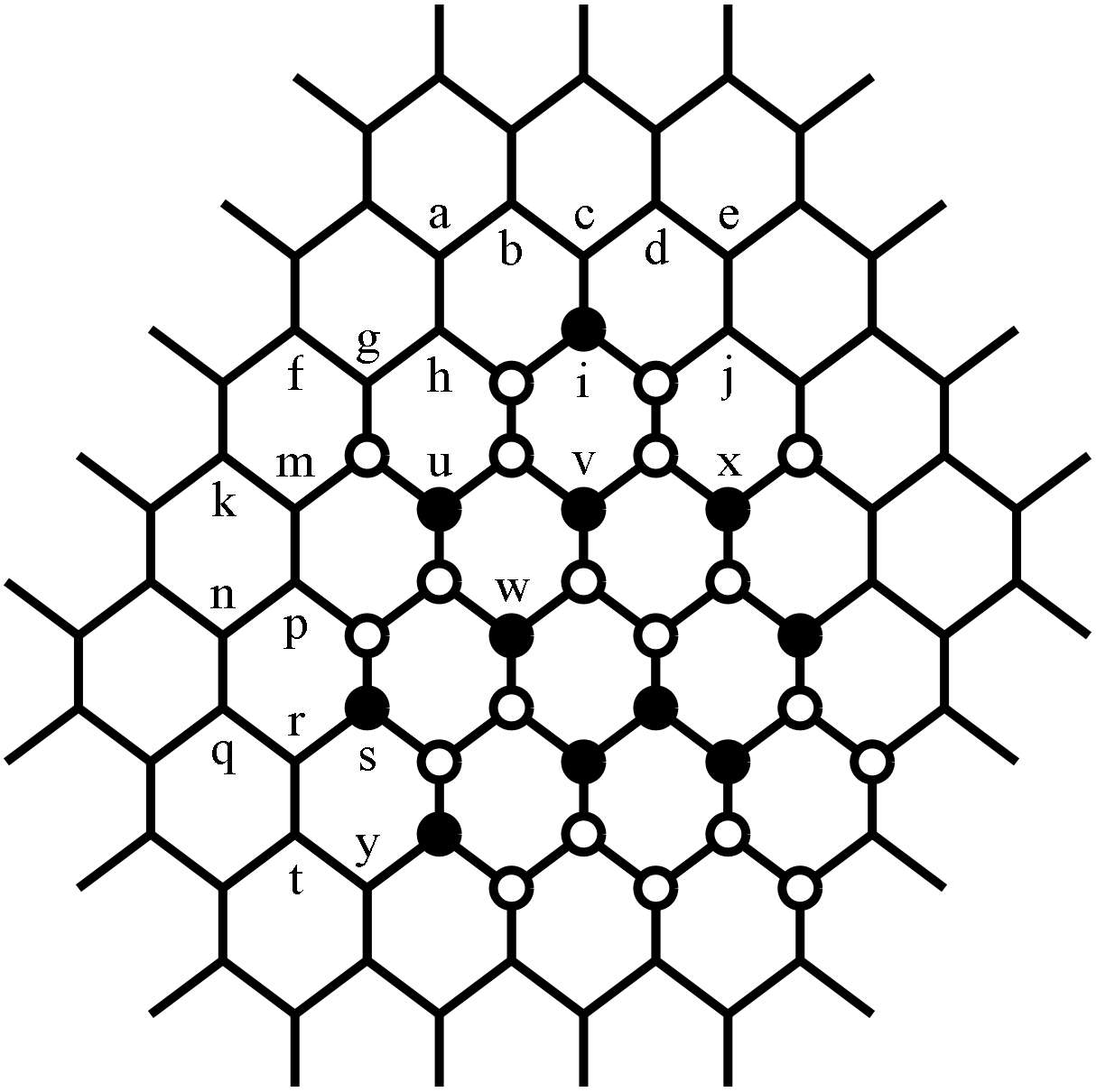}}
\end{center}
\caption{Lemma \ref{vp lem}}
\end{figure}

If $C_0$ is crowded, then (i) is satisfied and the lemma holds.  So assume $C_0$ is uncrowded.  Then, $j,k,s \not \in D$, and since $u \in \poorone$ we have $e \not \in D$ (Corollary \ref{poor1 cor}).  If $i \not \in D$, then $h,r \in D_{3^+}$.  Let $C_r$ be the \threepluss at $r$.  If $h \in C_r$, then $w$ is distance-3 from a \fourplus.  If $h \not \in C_r$, then $h$ closes $C_r$ and $w$ is distance-3 from a closed \threeplus.  In both cases, (ii) is satisfied.  So assume $w$ is not distance-3 from a closed 3-cluster or \fourplus.
Then, we have $i \in D$.  Now, if $r \not \in D$, then $i,q \in D_{3^+}$.  Let $C_i$ be the \threepluss at $i$.  If $q \in C_i$, then $w$ is distance-3 from a \fourplus; and if $q \not \in C_i$, then $q$ closes $C_i$ and $w$ is distance-3 from a closed \threeplus.  But we assumed that (ii) is not satisfied; therefore, $r \in D$.

If $c \not \in D$, then $a,b,d \in D_{3^+}$.  Let $C_b$ be the \threepluss at $b$.  If $b$ is a leaf of $C_b$, then either $u \in C_b$ and $C_b \in \cK_{4^+}$, or $u$ closes $C_b$ and $C_b \in \cK^c_{3^+}$, or $a \in C_b$ and $C_b \in \cK_{4^+}$, or $a$ closes $C_b$ and $C_b \in \cK^c_{3^+}$.  But $v$ is very poor and distance-3 from $b$, so we must have $C_b \in \openthree$.  This is only possible if $b$ is the middle vertex of $C_b$.  Let $C_a$ be the \threepluss at $a$, and let $C_d$ be the \threepluss at $d$.  Now, $C_b \in \openthree$, so $a$ is a leaf of $C_a$ and either $d \in C_a$ or $d$ closes $C_a$.  In the first case, $v$ is distance-3 from a \fourpluss.  But, by hypothesis, $v$ is not distance-3 from a \fourplus.  Therefore, $d$ closes $C_a$.  But then $d$ is a leaf of $C_d$ and $x$ closes $C_d$; therefore, $v$ is distance-3 from a closed \threeplus.  But, by hypothesis, $v$ is not distance-3 from a closed \threeplus.  Therefore, $c \in D$.


Figure \ref{verypoor2} shows the surrounding vertices of $v$ when neither (i) nor (ii) is satisfied.  All the vertices except $u,v,w$ and $x$ have been relabelled.  
Now, $u \in \poorone$; therefore, exactly one of $g$ and $m$ is in $D$ (Corollary \ref{poor1 cor}).

First, we consider the case for which $m \in D$ and $g \not \in D$.  If $p \in D$, then $m,p \in D_{3^+}$.  Let $C_{m,p}$ be the \threepluss at $m$ and $p$.  If $k,m,p \in C_{m,p}$, then $u$ closes $C_{m,p}$.  But $p$ is distance-3 from $w$ and, by assumption, $w$ is not distance-3 from a closed \threeplus.  If $m,n,p \in C_{m,p}$, then $s$ closes $C_{m,p}$.  But, again, by assumption, $w$ is not distance-3 from a closed \threeplus.  Therefore, $p \not \in D$.  Then, by Proposition \ref{force3cluster prop}, we have $s \in D_{3^+}$.  Let $C_s$ be the \threepluss at $s$.  Since 2 of the 3 neighbors of $s$ are not in $D$, we have $r \in C_s$.  Now, $C_s$ is distance-3 from $w$; therefore, by assumption, $C_s \in \openthree$.  Therefore, either $q,r,s \in C_s$ or $r,s,t \in C_s$.  In both cases we have $n,y \not \in D$.  Then, by Proposition \ref{force3cluster prop}, we have $m \in D_{3^+}$.  Let $C_m$ be the \threepluss at $m$.  Since 2 of the 3 neighbors of $m$ are not in $D$, we must have $k \in C_m$.  

\begin{figure}[ht]
\begin{center}
\subfloat[Lemma \ref{vp lem}, where (i) and (ii) are not satisfied and $m \in D$.]{\label{verypoor2m}\includegraphics[width=0.4\textwidth]{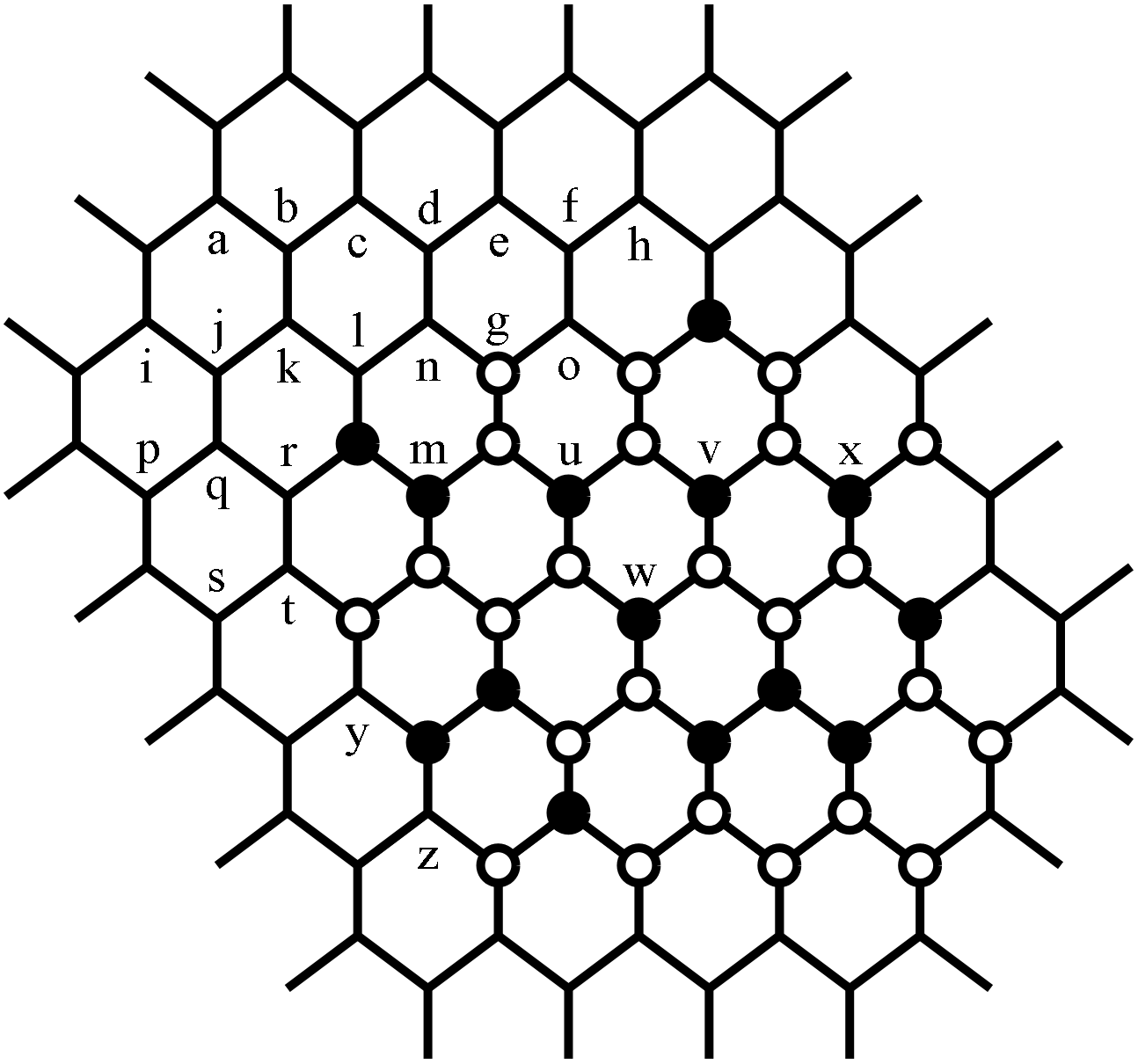}}
\hspace{1.5cm}
\subfloat[Lemma \ref{vp lem}, where (i) and (ii) are not satisfied and $g \in D$.]{\label{verypoor2g}\includegraphics[width=0.4\textwidth]{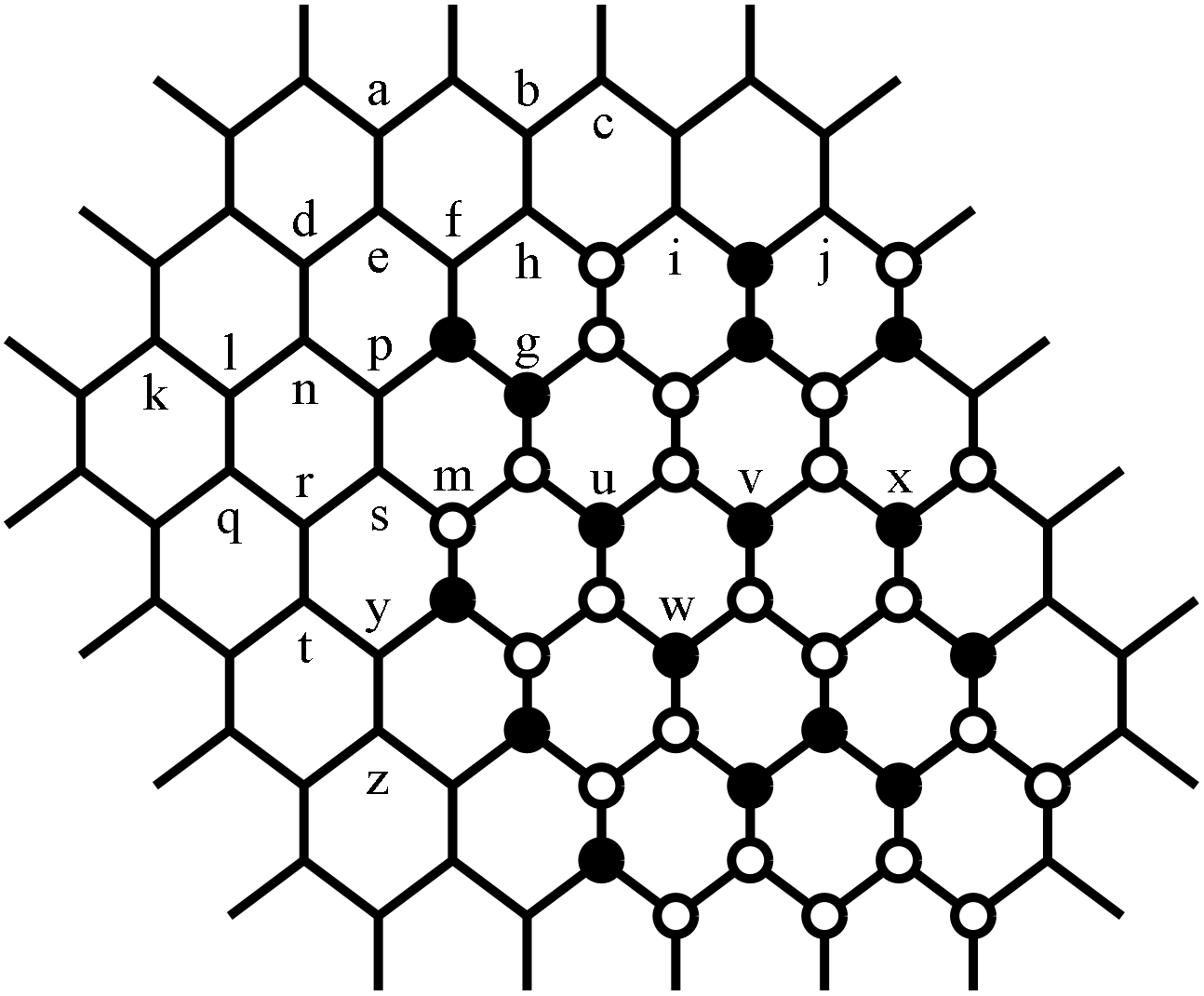}}
\end{center}
\caption{Lemma \ref{vp lem}}
\end{figure}

Figure \ref{verypoor2m} shows the surrounding vertices of $v$ when neither (i) nor (ii) is satisfied and $m \in D$.  All of the vertices except $g,m,u,v,w$ and $x$ have been relabelled.  

Now, if $z \in D$, then $z \in D_{3^+}$.  Let $C_z$ be the \threepluss at $z$.  By assumption, (ii) is not satisfied, so $C_z \in \openthree$.  Then $w$ is in the hand position of an open 3-cluster, $C_z$, such that the tail position is in $D$ and $w$ is on the finless side.  Therefore, (iii) is satisfied.

Now assume (iii) is not satisfied.  Then $z \not \in D$ and $y \in D_{3^+}$.  Let $C_y$ be the \threepluss at $y$.  Since $w$ is distance-3 from $C_y$, we have $C_y \in \openthree$.  Let $C_m$ be the \threepluss at $m$.  Then, $m$ is a leaf of $C_m$ and $u$ is the only vertex in $D \setminus C_m$ at distance-2 from $m$.  If $C_m$ is a linear 4-cluster, then either $q,r \in C_m$ or $l,k \in C_m$; in both cases, the one-turn position at distance-2 from $m$ is not in $D$.  If $C_m$ is a curved 4-cluster, then either $r,t \in C_m$ or $l,n \in C_m$; in both cases, the backwards position at distance-2 from $m$ is not in $D$.  Therefore, if $C_m \in \cK_{4^+}$, then (iv) is satisfied.

Now assume (iv) is not satisfied.  Then, $C_m \in \cK_3$.  Either $l \in C_m$ or $r \in C_m$; in both cases, $u$ is in a foot or arm position.  First, we consider the case in which $r \in C_m$.  Then, $l,q,t \not \in D$.  If $n \not \in D$, then $d,o \in D_{3^+}$ (Proposition \ref{force3cluster prop}).  Since 2 of the 3 neighbors of $o$ are not in $D$, we also have $f \in D_{3^+}$.  Let $C_{f,o}$ be the \threepluss at $f$ and $o$.  If $e \in D$, then $d,e,f,o \in C_{f,o}$ and $C_{f,o} \in \cK_{4^+}$.  If $e \not \in D$, then $h \in C_{f,o}$ and $d$ closes $C_{f,o}$.  But, by hypothesis, $v$ is not within distance-3 of a closed 3-cluster or $4^+$-cluster.  Therefore, $n \in D$ and $C_m \in \cK_3^c$.  Then, $C_m$ is type-2 paired with $C_y$ and $u$ is in the arm position on the closed side of $C_m$.  Then, (v) is satisfied.  Therefore, with $r \in D$, the lemma holds.

Now assume (v) is not satisfied.  So we have $l \in C_m$.  Recall $C_m \in \cK_3$; therefore, $k,n,r \not \in D$.  Then, $f,o \in D_{3^+}$ (Proposition \ref{force3cluster prop}).  Again, let $C_{f,o}$ be the \threepluss at $f$ and $o$.  Now, $C_{f,o}$ is distance-3 from $v$, so we must have $C_{f,o} \in \openthree$.  If $h \in C_{f,o}$, then the tail position of $C_{f,o}$ is in $D$ and $u$ is in the hand position on the finless side.  But, by assumption, (iii) is not satisfied.  Therefore, $e \in C_{f,o}$ and $d,h \not \in D$.  If $C_{f,o} \in \cK^c_3$, then $C_{f,o}$ has at most 6 neary poor 1-clusters: $a/b, c, i/j, p/q, s/t$ and $u$.  But, by assumption, (v) is not satisfied.  So we have $C_{f,o} \in \openthree$.  Then, $q,t \not \in D$; therefore, $s \in D_{3^+}$.  Let $C_s$ be the \threepluss at $s$.  If $C_s$ occupies the arm position of $C_y$, then $w$ is in the hand position of an open 3-cluster satisfying (vi).  

Now assume (vi) is not satisfied.  Then, then arm position of $C_y$ is not in $D$.  Therefore, $u$ is in the foot position of an open 3-cluster which is type-1 paired on top.  Therefore, (vii) is satisfied, and the lemma holds.
\\
\\
Now we return to Figure \ref{verypoor2} and consider the case for which $g \in D$ and $m \not \in D$.

If $h \in D$, then $g,h \in D_{3^+}$ and either $a \in D_{3^+}$ or $f \in D_{3^+}$.  In both cases, $v$ is distance-3 from a closed \threeplus.  Therefore, $h \not \in D$.  Then, by Proposition \ref{force3cluster prop}, we have $i \in D_{3^+}$.  Let $C_i$ be the \threepluss at $i$.  Since 2 of the 3 neighbors of $i$ are not in $D$, we have $c \in C_i$ and $i$ is a leaf of $C_i$.  Therefore, $j \in D$ (Proposition \ref{leaves prop}).  Now, $C_i$ is distance-3 from $v$; therefore, $C_i \in \openthree$.  Either $b \in C_i$ or $d \in C_i$.  In both cases, we have $a,e \not \in D$.  Then, $g \in D_{3^+}$ (Proposition \ref{force3cluster prop}).  Let $C_g$ be the \threepluss at $g$.  Since 2 of the 3 neighbors of $g$ are not in $D$, we have $f \in C_g$.  If $p \not \in D$, then $n,s \in D_{3^+}$ (Proposition \ref{force3cluster prop}).  Let $C_s$ be the \threepluss at $s$.  Since 2 of the 3 neighbors of $s$ are not in $D$, we have $r \in C_s$.  If $q \in D$, then $n,q,r,s \in C_s$ and $C_s \in \cK_{4^+}$.  If $q \not \in D$, then $r,s,t \in C_s$ and $n$ closes $C_s$.  But $s$ is distance-3 from $w$ and we assumed (ii) is not satisfied.  Therefore, $p \in D$.


Figure \ref{verypoor2g} shows the surrounding vertices of $v$ when neither (i) nor (ii) is satisfied and $g \in D$.  All of the vertices except $g,m,u,v,w$ and $x$ have been relabelled.

If $j \in D$, then $j \in D_{3^+}$.  Let $C_j$ be the \threepluss at $j$.  Since, $C_j$ is distance-3 from $v$, we must have $C_j \in \openthree$.  Then, the tail position of $C_j$ is in $D$ and $v$ is in the hand position on the finless side.  Therefore, (iii) is satisfied.

Now assume (iii) is not satisfied.  Then we have $i \in D_3^o$ and $j \not \in D$.  Let $C_i$ be the open 3-cluster at $i$.  Now, $u$ is distance-2 from the \threepluss at $g$; let $C_g$ be this \threeplus.  Then, $g$ is a leaf of $C_g$ and $u$ is the only distance-2 vertex of $g$ in $D \setminus C_g$.  If $C_g$ is a linear 4-cluster, then either $n,p \in C_g$ or $e,f \in C_g$; in both cases, the one-turn position at distance-2 from $g$ is not in $D$.  If $C_g$ is a curved 4-cluster then either $f,h \in C_g$ or $p,s \in C_g$; in both cases, the backwards position at distance-2 from $g$ is not in $D$.  Therefore, if $C_g \in \cK_{4^+}$, then (iv) is satisfied.

Now assume (iv) is not satisfied.  Then, $C_g \in \cK_3$.  Either $f \in C_g$ or $p \in C_g$; in both cases, $u$ is in a foot or arm position.  First we consider the case in which $f \in C_g$.  Then, $e,h,p \not \in D$.  If $s \not \in D$, then $r,y \in D_{3^+}$ (Proposition \ref{force3cluster prop}).  Let $C_y$ be the \threepluss at $y$.  If $t \in D$, then $r,t \in C_y$ and $C_y \in \cK_{4^+}$.  If $t \not \in D$, then $z \in C_y$ and $r$ closes $C_y$.  In both cases, $w$ is distance-3 from a closed 3-cluster or $4^+$-cluster.  But we assumed that (ii) is not satisfied.  Therefore, $s \in D$.  Then, $u$ is in an arm position on the closed side of $C_g$ , and $C_g$ is type-2 paired with $C_i$.  Then, (v) is satisfied.  Therefore, with $f \in D$, the lemma holds.

Now assume (v) is not satisfied.  Then $f \not \in D$.  Therefore, $p \in C_g$ and $f,n,s \not \in D$.  Then, $y \in D_{3^+}$ (Proposition \ref{force3cluster prop}).  Let $C_y$ be the \threepluss at $y$.  Now, $u$ is distance-3 from $C_y$ and we assumed (ii) is not satisfied.  Therefore, $C_y \in \openthree$.  If $z \in C_y$, then the tail position of $C_y$ is in $D$ and $u$ is in the hand position on the finless side of $C_y$.  But we assumed that (iii) is not satisfied.  Therefore, $t \in C_y$.  Now, if $C_g$ is a closed 3-cluster, then there are at most 6 nearby poor 1-clusters: $b/h, a/e, d, k/l, q$ and $u$.  But we assumed (v) is not satisfied.  Therefore, $C_g \in \openthree$.  Then, $h \not \in D$.  By Proposition \ref{force3cluster prop}, we have $b \in D_{3^+}$.  If $c \in D$, then $v$ is in the hand position of $C_i$ and the hand and arm positions on the other side of $C_i$ are both in $D$.  Therefore, (vi) is satisfied.  

Now assume (vi) is not satisfied.  Then we have $c \not \in D$.  Therefore, $u$ is in the foot position of an open 3-cluster which is type-1 paired on top.  Therefore, (vii) is satisfied and the lemma holds.
\end{proof}

\begin{proof}[Proof of Corollary \ref{vpasym cor2}]

Let $x$ be in the $x$-position of $v$.  Now, $v \in D_1^{vp}$; therefore, $u,w,x \in \poorone$.  Additionally, $v$ is in a head position of an open 3-cluster, $C_0$, and $w$ is in a shoulder position (Lemma \ref{vp lem}).  Let $u,v,w,x$ and $C_0$ be as shown in Figure \ref{verypoor1u}.

Suppose by contradiction that $u$ is distance-2 from a very poor 1-cluster other than $v$.  There are 2 possibilities: $c \in D_1^{vp}$ or $h \in D_1^{vp}$.  If $c \in D_1^{vp}$, then $d \not \in D$ and $a \in D$ (Proposition \ref{1clusters prop}).  Then, $e \in D_{3^+}$ (Proposition \ref{force3cluster prop}).  Let $C_e$ be the \threepluss at $e$.  Either $a \in C_e$ or $a$ closes $C_e$; in both cases $c$ is within distance-3 of a closed 3-cluster or \fourplus.  Therefore, $c \not \in D_1^{vp}$.  If $h \in D_1^{vp}$, then $k \not \in D$ and $j \in D$ (Proposition \ref{1clusters prop}).  Then, $m \in D_{3^+}$ (Proposition \ref{force3cluster prop}).  Let $C_m$ be the \threepluss at $m$.  Either $j \in C_m$ or $j$ closes $C_m$; in both cases $h$ is within distance-3 of a closed 3-cluster or \fourplus.  Therefore, $h \not \in D_1^{vp}$.

Now, suppose by contradiction that $w$ is distance-2 from a very poor 1-cluster other than $v$.  The only possibility is $u$.  If $u \in D_1^{vp}$, then $c \in \poorone$ or $h \in \poorone$.  But we already saw that $c \in \poorone$ or $h \in \poorone$ implies $e \in D_3^c \cup D_{4^+}$ or $m \in D_3^c \cup D_{4^+}$.  Since $u$ is distance-3 from both $e$ and $m$, we have $u \not \in D_1^{vp}$.
\end{proof}

\begin{figure}[h]
\setcaptionwidth{0.2\textwidth}
\begin{center}
\subfloat[Corollary \ref{vpasym cor2} and Corollary \ref{vpasym cor3}]{\label{verypoor1u}\includegraphics[width=0.3\textwidth]{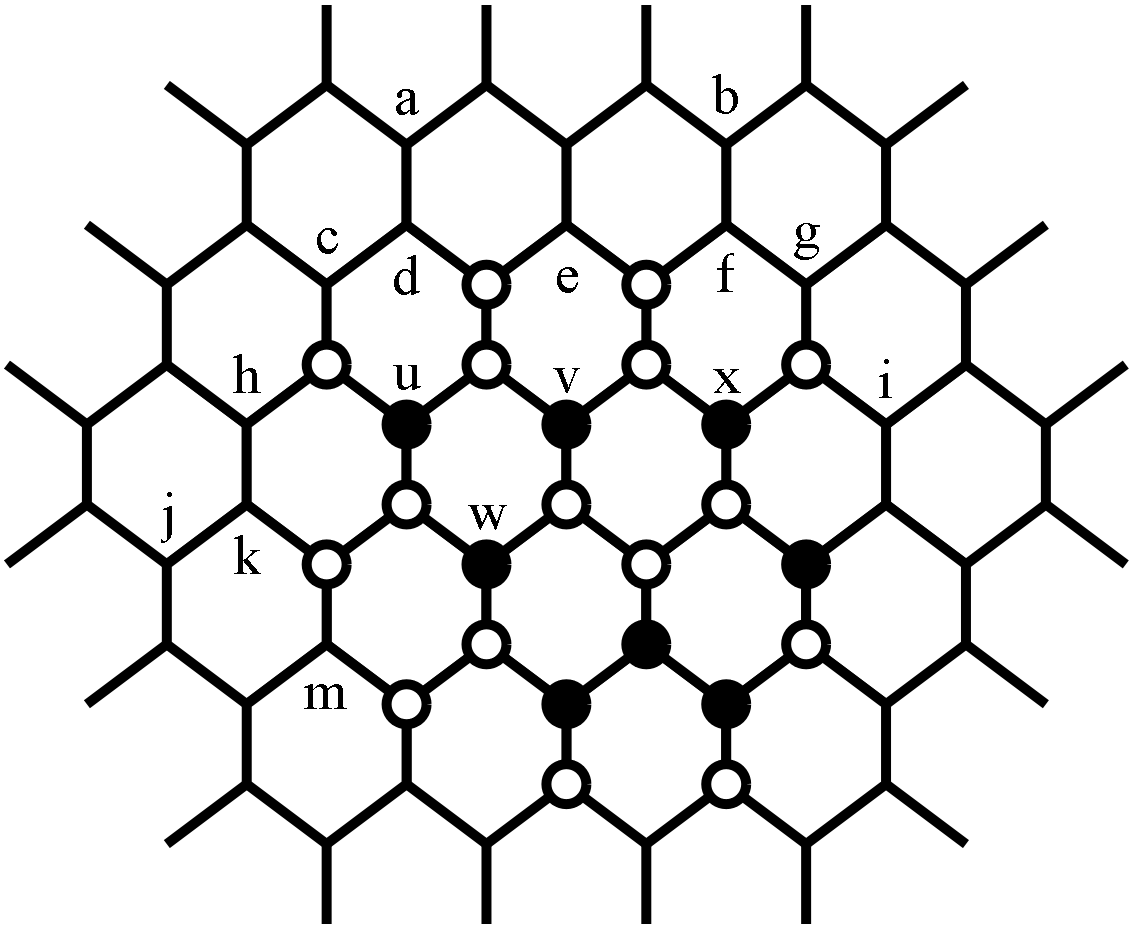}}
\hspace{2 cm}
\subfloat[Lemma \ref{vpshoulders lem}]{\label{verypoorshoulders}\includegraphics[width=0.2\textwidth]{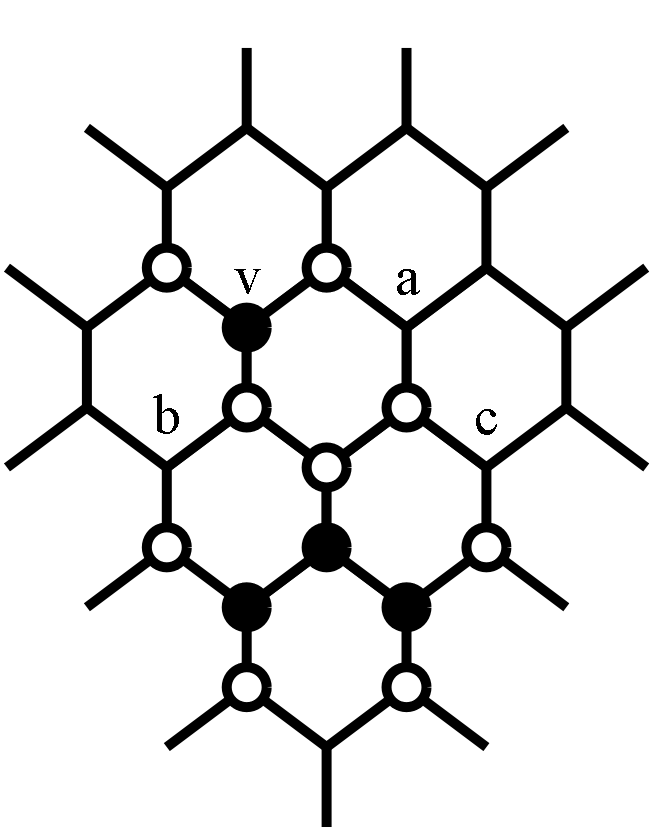}}
\end{center}
\caption{}
\end{figure}

\begin{proof}[Proof of Corollary \ref{vpasym cor3}]

Let $v$ and $x$ be as shown in Figure \ref{verypoor1u}.  If $x \in D_1^{vp}$, then $g \in \poorone$ or $i \in \poorone$.  If $g \in \poorone$, then $f \not \in D$ and $b \in D$ (Proposition \ref{1clusters prop}).  Therefore, $e \in D_{3^+}$ (Proposition \ref{force3cluster prop}).  Let $C_e$ be the \threepluss at $e$.  Either $b \in C_e$ or $b$ closes $C_e$; in both cases, $x$ is distance-3 from a closed 3-cluster or \fourplus.  But, by hypothesis, $v \in D_1^{vp}$.  Therefore, $i \in \poorone$.  Therefore, $x$ is in an asymmetric orientation and $v$ is in the $x$-position of $x$.
\end{proof}

\begin{proof}[Proof of Lemma \ref{vpshoulders lem}]

Let $C$ be the open 3-cluster shown in Figure \ref{verypoorshoulders}, and let $v$ be a very poor 1-cluster.  Then, $v$ is in a head position of $C$.  Since $v \in D_1$, we have $b \in D$ (Proposition \ref{1clusters prop}).  Suppose by contradiction that $c \not \in D$.  Then, $a \in D_{3^+}$ (Proposition \ref{force3cluster prop}).  Then $v \in D_1^{vp}$ and $v$ is distance-2 from a \threeplus, which is a contradiction.  Therefore, $c \in D$.
\end{proof}

\begin{proof}[Proof of Lemma \ref{type1paired lem}]
\label{type1paired pf}
Let $C_1$ be the open 3-cluster described by $j,k$ and $m$ in Figure \ref{type1pairedexpanded}.  Then $C_1$ is type-1 paired on top.  Now, $e \in D_{3^+}$ (Proposition \ref{force3cluster prop}).  Let $C_e$ be the \threepluss at $e$.  Since 2 of the 3 neighbors of $e$ are not in $D$, we have $d \in C_e$.  If $C_1$ has a poor 1-cluster in a shoulder or arm position, then either $h \in \poorone$ or $i \in \poorone$.  

First suppose $h \in \poorone$.  Then $c \in D$ (Proposition \ref{1clusters prop}).  Therefore, $c \in C_e$ and $h$ is distance-2 from $C_e$.  If $C_e \in \openthree$, then $h$ is in a foot position and $C_e$ is not paired.

Now suppose $i \in \poorone$.  Then $h \not \in D$ (Corollary \ref{poor1 cor}).  Therefore, $c,g \in D_{3^+}$ (Proposition \ref{force3cluster prop}).  Let $C_g$ be the \threepluss at $g$.  If $f \in D$, then $i$ is distance-2 from $C_g$.  If $f \not \in D$, then either $a,b,c,g \in C_g$ or $c$ closes $C_g$; in both cases, $i$ is distance-3 from a closed 3-cluster or $4^+$-cluster.  Thus, if $C_g \in \openthree$, then $a,f,g \in C_g$.  Then, $c$ and $i$ are in the shoulder positions of $C_g$ and, hence, $C_g$ is not type-1 paired on top.
\end{proof}

\begin{figure}[h]
\begin{center}
\subfloat[Lemma \ref{type1paired lem}]{\label{type1pairedexpanded}\includegraphics[width=0.29\textwidth]{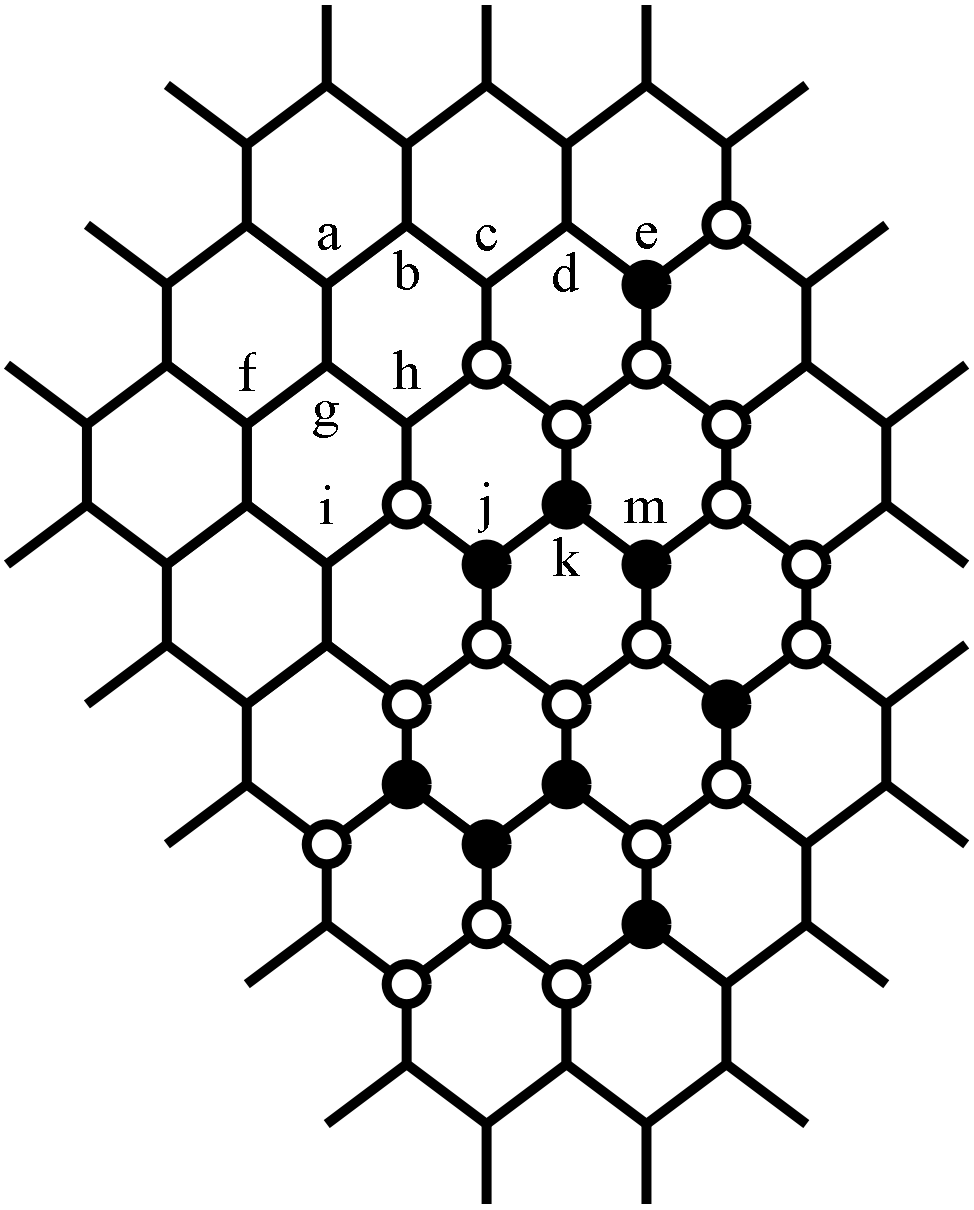}}
\hspace{1cm}
\subfloat[Lemma \ref{type2paired lem}]{\label{type2pairedexpanded}\includegraphics[width=0.3\textwidth]{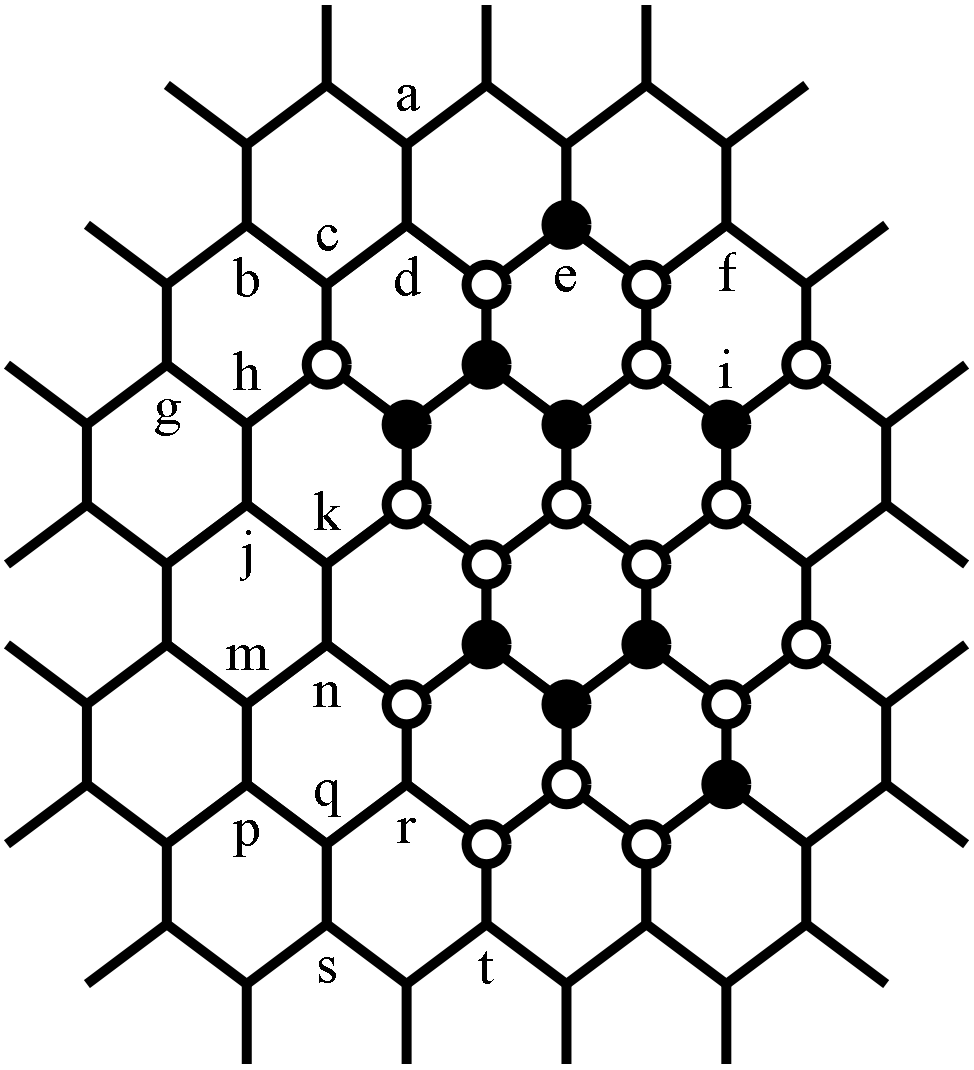}}
\end{center}
\caption{}
\end{figure}

\begin{proof}[Proof of Lemma \ref{type2paired lem}]
\label{type2paired pf}

Let $C_1$ be the type-2 paired closed 3-cluster shown in Figure \ref{type2pairedexpanded}, and let $C_2$ be the type-2 paired open 3-cluster.  Suppose by contrapositive that $n \not \in D$ or $k \in D$.  First, we deal with case in which $n \not \in D$.  Then $C_1$ has 7 candidates for nearby poor 1-clusters: $a/d$, $b/c$, $e$, $f$, $g/h$, $i$ and $j/k$.  It suffices to eliminate one of these candidates.  If $k \not \in D$, then $j \in D_{3^+}$ (Proposition \ref{force3cluster prop}) and the lemma holds.  If $k \not \in \poorone$, then the lemma holds; so assume $k \in \poorone$.  Then, $j \not \in D$.  If $h \not \in D$, then $g \in D_{3^+}$ (Proposition \ref{force3cluster prop}) and the lemma holds.  If $h \not \in \poorone$, then the lemma holds; so assume $h \in \poorone$.  Then, $c \not \in D$ (Corollary \ref{poor1 cor}).  If $b \not \in \poorone$, then the lemma holds; so assume $b \in \poorone$.  Then, $d \in D$ (Proposition \ref{1clusters prop}).  But then $e$ is adjacent to a one-third vertex; therefore, $e \not \in \poorone$.  Therefore, $C_1$ does not have 7 nearby poor 1-clusters.  Now, we deal with the case in which $k \in D$.  Then, $C_1$ has 7 candidates for nearby poor 1-clusters: $a/d$, $b/c$, $e$, $f$, $g/h$, $i$ and $k$.  It suffices to eliminate one of these candidates.  If $k \not \in \poorone$, then the lemma holds; so assume $k \in \poorone$.  Then, $j \not \in D$.  If $h \not \in D$, then $g \in D_{3^+}$ (Proposition \ref{force3cluster prop}) and the lemma holds; so assume $h \in D$.  If $h \not \in \poorone$, then the lemma holds; so assume $h \in \poorone$.  Then, $c \not \in D$ (Corollary \ref{poor1 cor}).  If $b \not \in \poorone$, then the lemma holds; so assume $b \in \poorone$.  Then, $d \in D$ (Proposition \ref{1clusters prop}).  But then $e$ is adjacent to a one-third vertex; therefore, $e \not \in \poorone$.  Therefore, $C_1$ does not have 7 nearby poor 1-clusters.

If $n \in \poorone$, then $m \not \in D$ and $r \not \in D$ (Corollary \ref{poor1 cor}).  Then, $q,t \in D_{3^+}$ (Proposition \ref{force3cluster prop}).  Let $C_q$ be the \threepluss at $q$.  If $C_q \in \openthree$, then $p,q,s \in C_q$.  However, $n$ and $t$ are in shoulder positions; therefore, $C_q$ is not type-1 paired on top.  If $C_q \not \in \openthree$, then $n$ closes $C_q$ or $t$ closes $C_q$ or $t \in C_q$; in each case, $n$ is nearby a closed 3-cluster or \fourplus.
\end{proof}

\section*{Acknowledgement}

We would like to thank Chase Albert for his help in constructing the first two codes shown in Figure \ref{3codes} and David Phillips for introducing us to the integer programming techniques which led to these constructions.  We would also like to thank Jeff Soosiah for his help in constructing the third code in Figure \ref{3codes} and for his support in the early stages of this project.

\end{document}